\documentclass[reqno,11pt,a4paper]{amsart}

\usepackage{tikz}

\usepackage[mathscr]{eucal}
\usepackage{graphics,epic}
\usepackage{amsfonts}
\usepackage{amscd}
\usepackage{latexsym}
\usepackage{amsmath,amssymb, amsthm, stmaryrd, bm, bbm}
\usepackage[all,2cell]{xy}

\newtheorem{theorem}{Theorem}[section]
\newtheorem*{theorem*}{Theorem}

\newtheorem{lemma}[theorem]{Lemma}
\newtheorem{proposition}[theorem]{Proposition}
\newtheorem{corollary}[theorem]{Corollary}

\newtheorem*{conjecture*}{Conjecture}
\newtheorem{question}[theorem]{Question}
\newtheorem*{question*}{Question}
\theoremstyle{remark}
\newtheorem{remark}[theorem]{Remark}
\newtheorem{example}[theorem]{Example}
\theoremstyle{definition}
\newtheorem{definition}[theorem]{Definition}

\newcommand{\ie}{{\em i.e.~}\ }

\newcommand{\opname}[1]{\operatorname{\mathsf{#1}}}

\renewcommand{\mod}{\opname{mod}\nolimits}

\newcommand{\proj}{\opname{proj}\nolimits}

\newcommand{\pd}{\opname{pd}\nolimits}
\newcommand{\id}{\opname{id}\nolimits}

\newcommand{\thick}{\opname{thick}\nolimits}

\newcommand{\Hom}{\opname{Hom}}
\newcommand{\End}{\opname{End}}

\newcommand{\Ext}{\opname{Ext}}

\newcommand{\ca}{{\mathcal A}}
\newcommand{\cb}{{\mathcal B}}
\newcommand{\cc}{{\mathcal C}}
\newcommand{\cd}{{\mathcal D}}

\newcommand{\cq}{{\mathcal Q}}

\newcommand{\ct}{{\mathcal T}}

\numberwithin{equation}{section}
\begin{document}

\title[silted algebras of hereditary algebras]{Classification of silted algebras for two quivers of Dynkin type $\mathbb{D}_{n}$}

\author{Houjun Zhang}
\thanks{MSC2020: 16G99, 16E35, 16E45.}
\thanks{Key words: tilted algebra, 2-term silting complex, silted algebra, strictly shod algebra, string algebra, realization functor}

\address{Houjun Zhang, School of Science, Nanjing University of Posts and Telecommunications, Nanjing 210023, P. R. China
}

\email{zhanghoujun@njupt.edu.cn}

\begin{abstract}
Let $Q$ be the Dynkin quiver of type $\mathbb{D}_{n}$ with linear orientation and let $Q'$ be the quiver formed by reversing the arrow at the unique source in $Q$. In this paper, we present a complete classification of both silted algebras and strictly shod algebras associated with these two quivers. Based on the classification, we derive formulas for counting the number of silted algebras and strictly shod algebras. 
 Furthermore, we establish that all strictly shod algebras corresponding to $Q$ and $Q'$ are string algebras. As an application, we provide a way to construct examples such that the realization functor which is induced from the $t$-structure does not extend to a derived equivalence. 
\end{abstract}
\maketitle


\section{Introduction}\label{s:introduction}
\medskip
As a generalization of tilted algebras, Buan and Zhou \cite{BuanZhou16b} introduced the concept of silted algebras in 2016, which are defined as endomorphism algebras of 2-term silting complexes. Silted algebras serve as a crucial bridge connecting hereditary algebras and derived categories. Their importance stems from the deep connection between their module categories and those of hereditary algebras: via torsion pair equivalences, the module category of a silted algebra decomposes into components equivalent to those of two subcategories within the hereditary algebra’s module category. This is similar to classical tilting theory but includes more types of algebras, such as those not induced by tilting modules, see \cite{BuanZhou16,BuanZhou16b} (also \cite{XieYangZhang23}).

Let $\mathcal{A}$ be an abelian category and let $\mathcal{B}$ denote the heart of the bounded derived category $\mathcal{D}^{b}(\mathcal{A})$. By \cite[Th\'{e}or\`{e}me 1.3.6]{BeilinsonBernsteinDeligne82}, $\mathcal{B}$ is also abelian, and there exists an embedding map from $\mathcal{B}$ to $\mathcal{D}^{b}(\mathcal{A})$. Beilinson, Bernstein and Deligne \cite{BeilinsonBernsteinDeligne82} extended this functor to a triangle functor, which is called a realization functor.
Recently, Martin Kalck observed an intriguing phenomenon: $\mathcal{A}$ and $\mathcal{B}$ are derived equivalent, but the embedding map from $\mathcal{B}$ to $\mathcal{D}^{b}(\mathcal{A})$ does not extend to a derived equivalence. This phenomenon was investigated by Yang in \cite{Yang20}. Furthermore, it was observed to arise in the silted algebras of a hereditary algebra of Dynkin type $\mathbb{D}_{5}$ \cite{Xing21}. 

Silted algebras are an important class of algebras, through research into them remains limited currently. One of the most significant results on silted algebras was established by Buan and Zhou \cite{BuanZhou16b}, who proved that a silted algebra is either a tilted algebra or a strictly shod algebra (\ie shod algebra of global dimension $3$). For hereditary algebras of finite representation type (e.g., Dynkin types $\mathbb{A}$, $\mathbb{D}$, $\mathbb{E}$), the number of isomorphism classes of basic 2-term silting complexes is finite \cite{ObaidNaumanFakiehRingel15}. We thus focus on classifying (up to isomorphism) basic silted algebras of Dynkin type, along with the strictly shod algebras among them. Moreover, by classifying silted algebras, can we derive additional examples exhibiting this phenomenon?

In \cite{XieYangZhang25}, we provided a classification of silted algebras for two hereditary algebras of Dynkin type $\mathbb{A}_{n}$ and showed that there are no strictly shod algebras among them. Furthermore, we proved that there are no strictly shod algebras for the path algebra of any quiver of type $\mathbb{A}$, (also see \cite{ZhangLiu25}).

In this paper, we extend our study to silted algebras of hereditary algebras of Dynkin type $\mathbb{D}_{n}$. Precisely, let $Q$ be the quiver of type $\mathbb{D}_{n}$ with linear orientation, \ie
$$Q=\begin{xy}
(-10,5)*+{1}="0",
(-10,-5)*+{2}="1",
(0,0)*+{3}="2",
(12,0)*+{\cdots}="3",
(28,0)*+{n-1}="4",
(42,0)*+{n}="5",
\ar"2";"0",\ar"2";"1", \ar"3";"2", \ar"4";"3", \ar"5";"4",
\end{xy}$$
and set $\Lambda_{n}=kQ$.
Let $Q'$ be the quiver obtained from $Q$ by reversing the arrow at the unique source, \ie
$$Q'=\begin{xy}
(-10,5)*+{1}="0",
(-10,-5)*+{2}="1",
(0,0)*+{3}="2",
(12,0)*+{\cdots}="3",
(28,0)*+{n-1}="4",
(42,0)*+{n}="5",
\ar"2";"0",\ar"2";"1", \ar"3";"2", \ar"4";"3", \ar"4";"5",
\end{xy}$$
and set $\Gamma_{n}=kQ'$.
We classify the basic silted algebras of type $\Lambda_{n}$ and type $\Gamma_{n}$, by introducing the concept of effective intersection, we further classify the strictly shod algebras among them. Based on the classification, we obtain some formulas for counting the number of these silted algebras and strictly shod algebras. Additionally, we show that all strictly shod algebras of type $\Lambda_{n}$ and type $\Gamma_{n}$ are string algebras. As a consequence, we obtain numerous examples such that the realization functor which is induced from the $t$-structure does not extend to a derived equivalence.

The paper is organized as follows. In Section~\ref{s:preliminaries}, we recall the notions of tilted algebras and silted algebras.  In Section ~\ref{s:classification-of-2-term-silting}, we provide a classification of the 2-term silting complexes over $\Lambda_n$. In Section~\ref{s:silted-algebra-of-A_{n}}, we study the silted algebras of type $\Lambda_n$: we first review some basic results on silted algebras of type $A_n$, the path algebra of quiver $$\begin{xy}
(-10,0)*+{1}="1",
(0,0)*+{2}="2",
(12,0)*+{\cdots}="3",
(28,0)*+{n-1}="4",
(42,0)*+{n}="5",
\ar"2";"1", \ar"3";"2", \ar"4";"3", \ar"5";"4",
\end{xy}$$ and then give a classification of silted algebras of type $\Lambda_n$ based on the classification of 2-term silting in Section ~\ref{s:classification-of-2-term-silting}. Furthermore, by introducing the concept of effective intersection, we determine the global dimension of silted algebras and establish a classification of strictly shod algebras. Using this classification, we also obtain some formulas for counting the number of these silted algebras and strictly shod algebras, and prove that all strictly shod algebras are string algebras. Based on the classification of silted algebras of type $\Lambda_n$, we give a complete classification of silted algebras of type $\Gamma_n$ in Section~\ref{s:silted-algebra-mutated-orientation}. In the final Section, we present a method for constructing classes of examples exhibiting the phenomenon that the realization functor which is induced from the $t$-structure does not extend to a derived equivalence.

\medskip
\noindent\emph{Notations and conventions.} Throughout, let $k$ be an algebraically closed field and all algebras are finite-dimensional $k$-algebras. For a $k$-algebra $A$, let $\mod A$ denote the category of all finite-dimensional right $A$-modules. Let $K^{b}(\proj A)$ be the bounded homotopy category of finitely generated projective $A$-modules,  and let $\tau$ denote the Auslander--Reiten translation in both $\mod A$ and in $K^b(\proj A)$. Let $K^{[-1,0]}(\proj A)$ be the full subcategory of $K^{b}(\proj A)$ consisting of complexes concentrated in degrees $-1$ and $0$. For any $A$-module $X$, $\pd X$ (resp., $\id X$) denotes the projective (resp., injective) dimension of $X$, and $\mathrm{gl.dim} A$ denotes the global dimension of $A$. We further denote by $t(A)$ the number of basic tilting modules over $A$, by $a_{t}(A)$ the number of tilted algebras of type $A$ and by $a_{s}(A)$ the number of silted algebras of type $A$.
The symbol $|X|$ denotes the number of non-isomorphic indecomposable direct summands of $X$.
For a Dynkin quiver $Q$ and vertex $i$ of $Q$, $P(i)$, $I(i)$ and $S(i)$ denote the indecomposable projective, indecomposable injective and simple modules at $i$, respectively.

\medskip
\noindent\emph{Acknowledgement}. {The author acknowledges support by National Natural Science Foundation of China No. 12301051 and by Natural Science Research Start-up Foundation
of Recruiting Talents of Nanjing University of Posts and Telecommunications No. NY222092. He is deeply grateful to Dong Yang for his consistent encouragement, support and guidance.} 

\section{Preliminaries}
\label{s:preliminaries}
In this section, we recall some related fundamental materials on tilting theory and silting theory. We refer the reader to \cite{AssemSimsonSkowronski06,KellerVossieck88,AiharaIyama12,AdachiIyamaReiten14,BuanZhou16b,HappelRingel82} for more details. Let $A$ be an algebra. 

\subsection{Tilted algebras}

\begin{definition}
A module $T\in \mod A$ is called a {\it tilting} module, if it satisfies the following three conditions:
 \begin{itemize}
\item[(1)] $\pd_A T \leq 1 $.
\item[(2)] $\Ext _{A}^{1} (T,T) =0$.
\item[(3)] $|T|=|A|$.
\end{itemize}
\end{definition}

\begin{definition}
Assume that $A$ is hereditary. An algebra $B$ is said to be {\it tilted} of type $A$ if there exists a tilting module $T$ over $A$ such that $B=\End_{A}(T)$.
\end{definition}

As a generalization of tilting modules, Adachi, Iyama and Reiten \cite{AdachiIyamaReiten14} introduced $\tau$-tilting theory.

\begin{definition}
Let $M$ be an $A$-module.
\begin{itemize}
\item[(1)] $M$ is called {\it $\tau$-rigid} if $\Hom_{A}(M,\tau M)=0$, and
$M$ is called {\it $\tau$-tilting} if $M$ is $\tau$-rigid and $|M|=|A|$.
\item[(2)] $M$ is called {\it support $\tau$-tilting} if there exists an idempotent
$e$ of $A$ such that $M$ is a $\tau$-tilting $A/ \langle e \rangle$-module.
\end{itemize}
\end{definition}


\begin{remark}\label{rem:number-of-tilting-and-support-tilting}
In \cite{ObaidNaumanFakiehRingel15}, the authors gave the number of tilting modules and support $\tau$-tilting
modules for any Dynkin algebra.
\end{remark}

\subsection{Silted algebras}

\begin{definition}\label{def:silting-objects}
Let $P$ be a complex in  $K^{b}(\proj A)$.
 \begin{itemize}
\item[(1)] $P$ is called {\it presilting} if $\Hom_{K^{b}(\proj A)}(P,P[i])=0$ for $i>0$.
\item[(2)] $P$ is called {\it silting} if it is presilting and generates $K^{b}(\proj A)$ as a triangulated category.
\item[(3)] $P$ is called {\it tilting} if $\Hom_{K^{b}(\proj A)}(P,P[i])=0$ for $i\neq 0$ and generates $K^{b}(\proj A)$ as a triangulated category.
\item[(4)] $P$ is called {\it 2-term} if it only has non-zero terms in degrees $-1$ and $0$, \ie $P\in K^{[-1,0]}(\proj A)$.
\end{itemize}
\end{definition}

It should be noted that a 2-term presilting complex $P$ in $K^{b}(\proj A)$ is silting if and only if $|P|=|A|$(see \cite[Proposition 3.14]{BrustleYang13}). Tilting modules are 2-term silting complexes. As a generalization of tilted algebras, in \cite{BuanZhou16b}, Buan and Zhou introduced silted algebras.

\begin{definition}\label{def:silted-algebras}
Assume that $A$ is hereditary. An algebra $B$ is called {\it silted} of type $A$ if there exists a 2-term silting complex $P$ over $A$ such that $B=\End_{K^{b}(\proj A)} (P)$.
\end{definition}

Next we recall the definition of strictly shod algebras in \cite{CoelhoLanzilotta99}.

\begin{definition}\label{def:strictly-shod-algebras}
An algebra $A$ is called {\it shod} (for small homological dimension) if for each indecomposable $A$-module $X$, either $\pd X\leq 1$ or $\id X\leq 1$. $A$ is called {\it strictly shod} if it is shod and $\mathrm{gl.dim} A=3$.
\end{definition}

Note that tilted algebras are silted. Moreover, any silted algebras is shod. In particular, Buan and Zhou \cite[Theorem 2.13]{BuanZhou16b} obtained the following result. 

\begin{theorem} \label{thm: silted-algebra-to-strictly-shod-algebra}
Let $A$ be a connected algebra. Then the following are equivalent.
\begin{itemize}
\item[(1)] $A$ is a silted algebra;

\item[(2)] $A$ is a tilted algebra or a strictly shod algebra.
\end{itemize}
\end{theorem}

\begin{remark}
In \cite[Theorem 3.2]{AdachiIyamaReiten14}, Adachi,Iyama and Reiten showed that there is a bijection between the
set of isomorphism classes of basic 2-term silting complexes over A and the set of isomorphism
classes of basic support $\tau$-tilting $A$-modules. Thus, by Remark \ref{rem:number-of-tilting-and-support-tilting}, we can obtain the number of basic 2-term silting complexes over any Dynkin algebra.
\end{remark}

\begin{corollary}\label{cor:tilting complex to tilted algebra}
Assume that $A$ is hereditary. Let $T$ be a 2-term tilting complex and $B=\End_{K^{b}(\proj A)}(T)$. Then $B$ is a tilted algebra of type $A'$ for some hereditary algebra $A'$ which is derived equivalent to $A$.
\end{corollary}

At the end of this section, we recall some results from \cite[Lemma 2.1]{XieYangZhang25} that will be useful for the classification of silted algebras.
Let $X$ and $Y$ be two finite sets. Denote by $X\times_s Y$ the set of all non-ordered pairs $\{x,y\}$, where $x\in X$ and $y\in Y$. 
\begin{lemma}
\label{lem:symmetric-product}
\begin{itemize}
\item[(a)] $X\times_s Y=Y\times_s X$.
\item[(b)] If $X$ and $Y$ have no intersection, then $X\times_s Y\cong X\times Y$.
\item[(c)] $|X\times_s X|=\frac{|X|(|X|+1)}{2}$.
\item[(d)] If $X'$ is a subset of $X$, then $|X'\times_s X|=|X'|\times |X|-\frac{|X'|(|X'|-1)}{2}$.
\end{itemize}
\end{lemma}

\section{A classification of the 2-term silting complexes over $\Lambda_{n}$}\label{s:classification-of-2-term-silting}
In this section, in order to study the silted algebras of type $\Lambda_{n}$, based on the classification of all basic 2-term silting complexes for the Dynkin quiver of type $\mathbb{A}_{n}$ with linear orientation, we give a classification of the basic 2-term silting complexes of $\Lambda_{n}$. 

Let $A_{n}$ be the path algebra of the following quiver
$$\begin{xy}
(-10,0)*+{1}="1",
(0,0)*+{2}="2",
(12,0)*+{\cdots}="3",
(28,0)*+{n-1}="4",
(42,0)*+{n}="5",
\ar"2";"1", \ar"3";"2", \ar"4";"3", \ar"5";"4",
\end{xy}.$$
Recall that \cite[3.3]{Ringel84} for an indecomposable module $P$ over $A_{n}$, a full translation subquiver $\Gamma'$ of the AR-quiver $\Gamma$ of $\mod A_{n}$ is called the \textit{wing} of $P$, if for $z$ of $\Gamma'$, all direct predecessors of $z$ in $\Gamma$ belong to $\Gamma'$, $\Gamma'$ is of the form of the AR-quiver of $\mod A_{m}$ for some $m$ and $P$ is the projective-injective vertex of $\Gamma'$. 

First, according to the proof of \cite[Proposition 2.1]{HappelRingel81} by Happel and Ringel, we have the following well-known classification of the isoclasses of basic tilting modules over $A_n$, see \cite[Proposition 4.4]{XieYangZhang25}.

\begin{proposition}
\label{prop:tilting-modules-of-A}
Let $T$ be a basic tilting module over $A_n$. Then $T$ is of one of the following two forms:
\begin{itemize}
\item[(1)] $T=P(1)\oplus T'$, where $T'$ is a basic tilting module of the wing of $P(2)$ or a basic tilting module of the wing of  $I(n{-}1)$;

\item[(2)] $T=P(1)\oplus T'\oplus T''$, where $T'$ is a basic tilting module of the wing of $P(i)$  for some $3\leq i\leq n$ and $T''$ is a basic tilting module of the wing of $I(i-2)$.
\end{itemize}
\end{proposition}

To give the classification of the isoclasses of basic tilting modules and 2-term silting complexes over $\Lambda_n$, we need the following lemmas. 

\begin{lemma} \label{A-D}
Let $S=P(i)[1]\oplus T$ be a basic 2-term silting complex of $\Lambda_{n}$ for some vertex $i$ in $Q_{0}$.
\begin{itemize}
  \item [(1)] If $i=3$, then $T=T'\oplus T''$, where $T'\oplus P(3)[1]$ can be viewed as a basic 2-term tilting complex over the path algebra of quiver $\begin{xy}
(-10,0)*+{1}="2",
(0,0)*+{3}="3",
(10,0)*+{2}="4",
\ar"3";"2", \ar"3";"4",
\end{xy}$ and $T''$ is a basic tilting module of the wing of $P(4)$. 
  \item [(2)] If $i\geq4$, then $T=T'\oplus T''$, where $T'\oplus P(i)[1]$ can be viewed as a basic 2-term tilting complex over $\Lambda_{i}$ and $T''$ is a basic tilting module of the wing of $P(i+1)$.
\end{itemize}
\end{lemma}

\begin{proof}
Note that $i$ divides $Q$ into two parts, where one part is the subquiver $\widetilde{Q}$ of Dynkin type $\mathbb{A}_{n}$ or $\mathbb{D}_{n}$ with linear orientation. Since $\Hom(P(i), \tau I(i-1))\neq 0$, it follows that $\Hom(I(i-1),P(i)[1]) \neq 0$. On the other hand, according to the AR-quiver of $\mod \Lambda_{n}$, we have $\Hom(P(i)[1], T'[-1])= 0$ and $\Hom( T', P(i))= 0$. So $T'\oplus P(i)[1]$ is a basic 2-term tilting complex.
\end{proof}

\begin{remark}
Indeed, $T'\oplus P(i)[1]$ is the form of $\tau^{-1}\widetilde{T}$ for some tilting module $\widetilde{T}$.
\end{remark}

\begin{lemma} \label{classification}
Let $T$ be a basic tilting module over $\Lambda_{n}$ and let $S$ be a basic 2-term silting complex which is not tilting over $\Lambda_{n}$. Assume that $T$ and  $S$ are induced by the idempotent element $e_{i}$ for some vertex $i$ in $Q_{0}$.
\begin{itemize}
  \item [(1)] If $i\geq4$, then $T=\tau^{-m}(T_{1}\oplus T_{2})$ for some integer $m\geq 0$, where $T_{1}=\tau^{-1}T'$ for some basic tilting module $T'$ over some subquiver of $Q$ and $T_{2}$ is a basic tilting module of the wing of $P(i)$. 
  \item [(2)] If $i\geq3$, then $S=S_{1}\oplus S_{2}$, where $S_{1}$ can be viewed as a basic 2-term tilting complex over the path algebra of some subquiver of $Q$ and  $S_{2}$ is a basic tilting module of the wing of $P(i+1)$.
\end{itemize}
\end{lemma}
\begin{proof}
 (1) Let $\ P(i)=e_{i}\Lambda_{n},\ \Lambda_{n}(i)=\Lambda_{n}/\langle e_{i} \rangle$. 
 For each basic tilting $\Lambda_{n}(i)$-module $N$, considered as an $\Lambda_{n}$-module, which has no non-trivial injective direct summands, form the $\Lambda_{n}$-module $M=P(i)\oplus \tau^{-1}_{\Lambda_{n}}N_{\Lambda_{n}}$. If $i\geqslant 4$, according to the AR-quiver of $\mod \Lambda_{n}$, we obtain that $P(i)$ can be viewed as the projective-injective module over the path algebra of quiver 
$\begin{xy}
	(-18,0)*+{i}="1",
	(-4,0)*+{i+1}="2",
	(12,0)*+{\cdots}="3",
	(25,0)*+{n}="4",
	\ar"2";"1", \ar"3";"2", \ar"4";"3",
\end{xy}$. This shows that $T=T_{1}\oplus T_{2}$, where $T_{2}$ is a basic tilting module of the wing of $P(i)$ for any $i\geq 4$.

(2)  By Lemma \ref{A-D}.
\end{proof}

\begin{remark} 
\begin{itemize}
\item [(1)] In \cite{Xing21}, Xing gave a algorithm to produce all basic tilting modules and 2-term silting complexes over any path algebra of a Dynkin quiver.

\item [(2)] Assume that $I$ is any not-empty subset of the set of vertexes $Q_{0}$ in $Q$ and $i$ is the minimal element in $I$. Then, Lemma \ref{classification}(1) is also hold for the subset $I$.
    
\item [(3)] Assume that $I$ is any not-empty subset of the set of vertexes $Q_{0}$ in $Q$ and $i$ is the maximal element in $I$. Then, Lemma \ref{classification}(2) is also hold for the subset $I$.
\end{itemize}
\end{remark}

\begin{remark} \label{rem:assume}
Let $S=T\oplus P[1]$ be a basic 2-term silting complex over $\Lambda_{n}$. We have $P\in \proj \Lambda_{n}$ and $T\in \mod \Lambda_{n}$ by \cite[Proposition 2.1]{HappelRingel81}. In particular, if $P=0$ or $T$ has no non-trivial projective direct summands, then $\End(T)$ is a tilted algebra of type $\Lambda_{n}$. Thus, we will divide silted algebras of type $\Lambda_{n}$ into the following two classes:
\begin{itemize}
  \item  Tilted algebras of type $\Lambda_{n}$,
  \item  $\End(S)$, where $S=T\oplus P[1]$ be a basic 2-term silting complex such that $P\neq 0$ and $T$ has a non-zero projective direct summand over $\Lambda_{n}$.
\end{itemize}
\end{remark}

\subsection{Tilting modules over $\Lambda_{n}$}
In this subsection, we recall some facts on tilting modules over $\Lambda_n$. For convenience, we first present the AR-quiver of $\mod \Lambda_{n}$ in Figure~1. 

$$\begin{xy}
0;<3pt,0pt>:<0pt,3pt>::
(-20,-12)*+{}="21",
(-50,-43)*+{}="22",
(10,-43)*+{}="23",
(0,-50)*+{\mathrm{Figure \ 1: The~AR{-}quiver~of~\mathrm{mod}~\Lambda_{n}}},
(-10,0)*+{P(1)}="1",
(10,0)*+{\cdots}="2",
(30,0)*+{\cdots}="3",
(50,0)*+{\cdots}="4",
(-10,-10)*+{P(2)}="5",
(10,-10)*+{\cdots}="6",
(30,-10)*+{\cdots}="7",
(50,-10)*+{\cdots}="8",
(-20,-20)*+{P(3)}="9",
(0,-20)*+{\cdots}="10",
(20,-20)*+{\cdots}="11",
(40,-20)*+{I(3)}="12",
(-30,-30)*+{\cdots}="13",
(-10,-30)*+{\cdots}="14",
(10,-30)*+{\cdots}="15",
(30,-30)*+{\cdots}="16",
(-40,-40)*+{P(n)}="17",
(-20,-40)*+{\cdots}="18",
(0,-40)*+{\cdots}="19",
(20,-40)*+{I(n)}="20",
\ar"1";"10", \ar"2";"11", \ar"3";"12", \ar"5";"10", \ar"6";"11", \ar"7";"12",
\ar"9";"1", \ar"9";"5", \ar"10";"2", \ar"10";"6", \ar"11";"3",
\ar"11";"7", \ar"12";"4", \ar"12";"8", \ar"13";"9",\ar"13";"18",\ar"14";"10",\ar"10";"15",\ar"15";"11",
\ar"15";"20",\ar"20";"16",\ar"16";"12",\ar"9";"14",\ar"17";"13",
\ar@{.}"21";"22",\ar@{.}"21";"23",\ar@{.}"22";"23"
\save "1"*[F-]\frm{}, \save "3"*[F-]\frm{}, \save "6"*[F-]\frm{}, \save "8"*[F-]\frm{}
\save "2"*[F.]\frm{}, \save "4"*[F.]\frm{}, \save "5"*[F.]\frm{}, \save "7"*[F.]\frm{}
\end{xy}$$

Based on the AR-quiver of $\mod \Lambda_n$, we have the following important observation.

\begin{remark}\label{rem:AR-quiver-parts}
\begin{itemize}
  \item [(1)] The additive closure of all indecomposable modules associated with the dotted triangle and the dotted rectangles is equivalent to $\mod A_{n-1}$ for the quiver $\begin{xy}
(-10,0)*+{2}="1",
(0,0)*+{3}="3",
(12,0)*+{\cdots}="4",
(24,0)*+{n}="5",
\ar"3";"1", \ar"4";"3", \ar"5";"4",
\end{xy}.$ The additive closure of all indecomposable modules associated with the dotted triangle and the solid rectangles is equivalent to $\mod A_{n-1}$ for the quiver $\begin{xy}
(-10,0)*+{1}="2",
(0,0)*+{3}="3",
(12,0)*+{\cdots}="4",
(24,0)*+{n}="5",
\ar"3";"2", \ar"4";"3", \ar"5";"4",
\end{xy}.$
  \item [(2)] Consider the quiver $\begin{xy}
(-10,0)*+{2}="1",
(0,0)*+{3}="3",
(12,0)*+{\cdots}="4",
(24,0)*+{n}="5",
\ar"3";"1", \ar"4";"3", \ar"5";"4",
\end{xy}$ as a full subquiver of $Q$. Then $P(2)$ can be viewed as the projective-injective module over the path algebra $A_{n-1}$ of this quiver.
  Let $T$ be a tilting module over $A_{n-1}$. By Proposition \ref{prop:tilting-modules-of-A},
  \begin{itemize}
    \item [(a)] $T=P(2)\oplus T'$, where $T'$ is a basic tilting module of the wing $I(n{-}1)$. When this quiver is regarded as a full subquiver of $Q$,
 the indecomposable module $I(n{-}1)$ in the AR-quiver of $\mod A_{n-1}$ corresponds to the indecomposable module $\tau^{-1}P(1)$ in the AR-quiver of $\mod \Lambda_{n}$. 
    \item [(b)] $T=P(2)\oplus T'\oplus T''$, where $T'$ is a basic tilting module of the wing of $P(i)$ for some $4\leq i\leq n$ and $T''$ is a basic tilting module of the wing of $I(i{-}2)$. When this quiver is regarded as a full subquiver of $Q$,
 the indecomposable module $I(i{-}2)$ in the AR-quiver of $\mod A_{n{-}1}$ corresponds to the indecomposable module $\tau^{-(n{-}i{+}2)}P(1)$ (or $\tau^{-(n{-}i{+}2)}P(2)$ ) in the AR-quiver of $\mod \Lambda_{n}$. 
\end{itemize}
In the following of this paper, consider the quiver $\begin{xy}
(-10,0)*+{2}="1",
(0,0)*+{3}="3",
(12,0)*+{\cdots}="4",
(24,0)*+{n}="5",
\ar"3";"1", \ar"4";"3", \ar"5";"4",
\end{xy}$ as a full subquiver of $Q$. When we say $T$ is a basic tilting module of the wing of $I(i{-}2)$ in the AR-quiver of $\mod A_{n{-}1}$, we mean that $I(i{-}2)$ corresponds to the indecomposable module $\tau^{-(n{-}i{+}2)}P(1)$ or $\tau^{-(n{-}i{+}2)}P(2)$ in the AR-quiver of $\mod \Lambda_{n}$.
  \end{itemize}
\end{remark}

It is easy to see that the tilting modules in (1) and (2) of Remark \ref{rem:AR-quiver-parts} are of the forms $P(1)\oplus \tau^{-1}T_{1}$ and $P(2)\oplus \tau^{-1}T_{2}$, respectively.  Note that $T_{1}$ and $T_{2}$ are tilting modules over $\mod A_{n-1}$. Thus, if $\End(T_{1})\cong \End(T_{2})$ in the AR-quiver of $\mod \Lambda_n$, then 
\begin{center}
$\End(P(1)\oplus \tau^{-1}T_{1})\cong\End(P(2)\oplus \tau^{-1}T_{2})$. 
\end{center}
Moreover, if $T$ and $\tau^{-1}T$ are tilting modules over $\Lambda_n$, then $\End(T)\cong\End(\tau^{-1}T)$. 

In fact, the tilting modules over any path algebra of a Dynkin quiver have been studied quite maturely. In order to calculate the number of tilted algebras of type $\Lambda_n$, We assume that the tilting modules with isomorphism endomorphism algebras are isomorphisms. We note that the following result is well-known.

\begin{proposition}\label{prop:the-forms-of-tilting-modules}
Let $T$ be a basic tilting module over $\Lambda_n$. Then $T$ is of one of the following six forms up to isomorphism:
\begin{itemize}
\item[(1)] $T=P(1)\oplus \tau^{-1}T_{1}$, where $T_{1}$ is a basic tilting module of the wing of $P(2)$;
\item[(2)] $T=P(1)\oplus P(2)\oplus \tau^{-1} T_{2}$, where $T_{2}$ is a basic tilting module of the wing of $P(3)$;
\item[(3)] $T=P(1)\oplus P(2)\oplus T_{2}$, where $T_{2}$ is a basic tilting module of the wing of $P(3)$;
\item[(4)] $T=T_{1}\oplus T_{2}$, where $T_{1}=\tau^{-1}T'$ for some basic tilting module $T'$ over some subquiver of $Q$ and $T_{2}$ is a basic tilting module of the wing of $P(i)$ for some $4\leq i\leq n$;
\item[(5)] $T=P(1)\oplus T_{1}\oplus T_{2}$, where $T_{1}=\tau^{-1}T''$ for some basic tilting module $T''$ over some subquiver of $Q$ and $T_{2}$ is a basic tilting module of the wing of $P(i)$ for some $4\leq i\leq n$.
\item[(6)] $T=P(1)\oplus P(2)\oplus T_{1}\oplus T_{2}$, where $T_{1}=\tau^{-1}T''$ for some basic tilting module $T''$ over some subquiver of $Q$ and $T_{2}$ is a basic tilting module of the wing of $P(i)$ for some $4\leq i\leq n$.
\end{itemize}
\end{proposition}

\begin{proof} By Lemma \ref{classification} and the proof of \cite[Proposition 2.1]{HappelRingel81}.
\end{proof}

\begin{remark}\label{rem:I(1)-injective}
  \begin{itemize}
    \item[(1)] Let $T$ be a basic tilting module over $A_n$. Then $T$ must contain $P(1)$ as a direct summand, because it is a project-injective module. However, if we consider the indecomposable modules of $\mod A_n$ in the AR-quiver of $\mod \Lambda_{n}$, then $P(1)$ is not a injective module. Indeed, by Remark \ref{rem:AR-quiver-parts}, among the indecomposable modules in dotted rectangles, only $I(1)$ is an injective module. Thus, the module $\tau^{-1}T_{1}$, where $T_{1}$ is a tilting module of the wing of $P(2)$, in Proposition \ref{prop:the-forms-of-tilting-modules} (1) is exist and $|\tau^{-1}T_{1}|=n-1$.
    \item[(2)] Let $T$ be a basic tilting module of type of Proposition \ref{prop:the-forms-of-tilting-modules} (4). If $i=4$, then $T_{1}=I(1)\oplus I(2)\oplus I(n)$; If $i>4$, then $T_{1}$ can be viewed as a basic tilting module over $\Lambda_{i-1}$.
  \end{itemize}
\end{remark}

\subsection{2-term silting complexes over $\Lambda_{n}$}
In this subsection, we give a classification of the basic 2-term silting complexes over $\Lambda_{n}$. By \cite[Section I.5.6]{Happel88}, we know that the AR-quiver of $K^b(\proj \Lambda_n)$ is $\mathbb{Z}\overrightarrow{\mathbb{D}}_{n}$. We consider the AR-quiver of $\mod \Lambda_n$ as a full subquiver, and then draw the AR-quiver of $K^{[-1,0]}(\proj \Lambda_n)$ in Figure~2.
$$\begin{xy}
0;<3pt,0pt>:<0pt,3pt>::
(70,0)*+{P(1)[1]}="21",
(70,-10)*+{P(2)[1]}="22",
(60,-20)*+{P(3)[1]}="23",
(50,-30)*+{\cdots}="24",
(40,-40)*+{P(n)[1]}="25",
(10,-50)*+{\mathrm{Figure \ 2: The~AR{-}quiver~of~K^{[-1,0]}(\proj \Lambda_n)}},
(-10,0)*+{P(1)}="1",
(10,0)*+{\cdots}="2",
(30,0)*+{\cdots}="3",
(50,0)*+{\cdots}="4",
(-10,-10)*+{P(2)}="5",
(10,-10)*+{\cdots}="6",
(30,-10)*+{\cdots}="7",
(50,-10)*+{\cdots}="8",
(-20,-20)*+{P(3)}="9",
(0,-20)*+{\cdots}="10",
(20,-20)*+{\cdots}="11",
(40,-20)*+{I(3)}="12",
(-30,-30)*+{\cdots}="13",
(-10,-30)*+{\cdots}="14",
(10,-30)*+{\cdots}="15",
(30,-30)*+{\cdots}="16",
(-40,-40)*+{P(n)}="17",
(-20,-40)*+{\cdots}="18",
(0,-40)*+{\cdots}="19",
(20,-40)*+{I(n)}="20",
\ar"1";"10", \ar"2";"11", \ar"3";"12", \ar"5";"10", \ar"6";"11", \ar"7";"12",
\ar"9";"1", \ar"9";"5", \ar"10";"2", \ar"10";"6", \ar"11";"3",
\ar"11";"7", \ar"12";"4", \ar"12";"8", \ar"13";"9",\ar"13";"18",\ar"14";"10",\ar"10";"15",\ar"15";"11",
\ar"15";"20",\ar"20";"16",\ar"16";"12",\ar"9";"14",\ar"17";"13",
\ar"4";"23",\ar"8";"23",\ar"12";"24",\ar"16";"25",
\ar"25";"24",\ar"24";"23",\ar"23";"21",\ar"23";"22",
\end{xy}$$

Let $S_1$ and $S_2$ be two basic 2-term silting complexes over $\Lambda_n$ induced by the idempotents $e_1$ and $e_2$, respectively. By \cite[Proposition 2.1]{HappelRingel81} (also see \cite[Algorithm 3.1]{Xing21}), $S_1=P(1)[1]\oplus T_1$ and $S_2=P(2)[1]\oplus T_2$, where $T_1$ and $T_2$ are tilting modules over the subquivers $\begin{xy}
(-10,0)*+{2}="2",
(0,0)*+{3}="3",
(12,0)*+{\cdots}="4",
(24,0)*+{n}="5",
\ar"3";"2", \ar"4";"3", \ar"5";"4",
\end{xy}$ and $\begin{xy}
(-10,0)*+{1}="1",
(0,0)*+{3}="3",
(12,0)*+{\cdots}="4",
(24,0)*+{n}="5",
\ar"3";"1", \ar"4";"3", \ar"5";"4",
\end{xy}.$
Thus, if $\End(T_{1})\cong \End(T_{2})$ within the AR-quiver of $\mod \Lambda_n$, then 
\begin{center}
$\End(P(1)[1]\oplus T_{1})\cong\End(P(2)[1]\oplus T_{2})$. 
\end{center}
We assume that the 2-term basic silting complexes with isomorphism endomorphism algebras are isomorphisms. The following well-known proposition is a classification of the basic 2-term silting complexes over $\Lambda_n$ up to isomorphism.

\begin{proposition}\label{prop:classification-of-silting}
Let $S$ be a basic 2-term silting complex over $\Lambda_n$. Then $S$ is of one of the following six forms up to isomorphism:
\begin{itemize}
  \item[(1)] $T$, where $T$ is a basic tilting module over $\Lambda_n$;
  \item[(2)] $\tau^{-1}T$, where $T$ is a basic tilting module over $\Lambda_n$ which contains at least one injective module as a direct summand;
  \item[(3)] $S=P(1)[1]\oplus T$, where $T$ is a basic tilting module in the wing of $P(2)$;
  \item[(4)] $S=P(1)[1]\oplus P(2)[1]\oplus T$, where $T$ is a basic tilting module in the wing of $P(3)$;
  \item[(5)] $S=T_1\oplus T_2$, where $T_{1}$ can be viewed as a basic 2-term tilting complex over the path algebra of quiver $\begin{xy}
(-10,0)*+{1}="2",
(0,0)*+{3}="3",
(10,0)*+{2}="4",
\ar"3";"2", \ar"3";"4",
\end{xy}$ and $T_{2}$ is a basic tilting module in the wing of $P(4)$. 
  \item[(6)] $S=T_1\oplus T_2$, where $T_1$ can be viewed as a basic 2-term tilting complex over $\Lambda_{i}$ and $T_2$ is a basic tilting module in the wing of $P(i+1)$ for some $i\geq 4$.
\end{itemize}
\end{proposition}

\begin{proof} By Lemma \ref{classification}.
\end{proof}

\begin{remark}
Let $S$ be a basic 2-term silting complex over $\Lambda_n$. By Proposition \ref{prop:classification-of-silting} and Remark \ref{rem:assume}, we know that if $P(n)[1]$ is a direct summand of $S$, then $S$ is a basic 2-term tilting complex over $\Lambda_n$.
\end{remark}

\section{Silted algebras of type $\Lambda_n$}
\label{s:silted-algebra-of-A_{n}}
In this section, we give a classification of all basic silted algebras and strictly shod algebras among them of type $\Lambda_n$ up to isomorphism. 
Moreover, we calculate the number of these silted algebras, especially the strictly shod algebras. According to Propositions \ref{prop:classification-of-silting} and \ref{prop:the-forms-of-tilting-modules}, we know that the classification of silted algebras of type $\Lambda_n$ is closely related to the classification of tilted algebras of type $A_n$.

\subsection{A classification of the silted algebras of type $\Lambda_n$}
Put
\begin{align*}
\ca_t(A_n)&:=\{\text{basic tilted algebras of type $A_n$}\}/\cong~,\\
\ca_t(\Lambda_n)&:=\{\text{basic tilted algebras of type $\Lambda_n$}\}/\cong~,\\
\ca_t(\mathbb{D}_n)&:=\{\text{basic tilted algebras of type $\mathbb{D}_n$ but not of type $\Lambda_n$}\}/\cong~,\\
\ca_s(\Lambda_n)&:=\{\text{basic silted algebras of type $\Lambda_n$}\}/\cong~,\\
\ca_{ss}(\Lambda_n)&:=\{\text{basic strictly shod algebras of type $\Lambda_n$}\}/\cong~.
\end{align*}
Let $a_t(A_n)$, $a_t(\Lambda_n)$, $a_{ss}(\Lambda_n)$ and $a_s(\Lambda_n)$ denote the cardinalities of $\ca_t(\Lambda_n)$, $\ca_{ss}(\Lambda_n)$ and $\ca_s(\Lambda_n)$, respectively. Then we have a classification of silted algebras of type $\Lambda_n$ as follows:

\begin{theorem}\label{thm:classification-of-silted-algebras-of-D}
$\ca_s(\Lambda_n)=\cb_1\sqcup\cb_2\sqcup\cb_3\sqcup\cb_4\sqcup\cb_5\sqcup\cb_6\sqcup\cb_7$, where
\begin{itemize}
  \item [(1)] $\cb_1=\ca_t(\Lambda_n)$;
  \item [(2)] $\cb_2=\ca_t(\mathbb{D}_n)$;
  \item [(3)] $\cb_3=\bigsqcup_{m=4}^{n-1}(\ca_t(\Lambda_m)\times\ca_t(A_{n-m}))$;
  \item [(4)] $\cb_4=\ca_t^{3}(A_{n-1})\times_s\ca_t(A_{1})$;
  \item [(5)] $\cb_5=(\ca_t(A_{n-2})\times_s\ca_t(A_{1}))\times_s\ca_t(A_{1})$;
  \item [(6)] $\cb_6=\ca_t(A_{n-3})\times_s\ca_t(A_{3})$;
  \item [(7)] $\cb_7=\ca_{ss}(\Lambda_n)$. 
\end{itemize}
\end{theorem}

We present the proof of this Theorem in Section~\ref{sec:the-proof-of-theorem-4.1}. In \cite{XieYangZhang25}, we showed that there are no strictly shod algebras in silted algebras of type $A_n$. However, among silted algebras of type $\Lambda_n$, there are many strictly shod algebras. Indeed, silted algebras of type $\Lambda_{n}$ forming the following families: (1) elements in $\cb_1$ are tilted algebras of type $\Lambda_{n}$; (2) elements in $\cb_2$ are tilted algebras of type $\mathbb{D}_n$ but not of type $\Lambda_n$; (3) elements in $\cb_3$ are tilted algebras of type $\Lambda_{m}\times A_{n-m}$, where $4\leq m\leq n-1$; (4) elements in $\cb_4$ are tilted algebras of type $A_{n-1}\times A_{1}$; (5) elements in $\cb_5$ are tilted algebras of type $A_{n-2}\times A_{1}\times A_{1}$; (6) elements in $\cb_6$ are tilted algebras of type $A_{n-3}\times A_{3}$; (7) elements in $\cb_7$ are strictly shod algebras of type $\Lambda_{n}$.

\subsection{Tilted algebras of type $\Lambda_{n}$} \label{s:tilted-algebras-of-type-D}
In this subsection, we study tilted algebras of type $\Lambda_n$. First, we recall some results for $A_n$ in \cite{XieYangZhang25}. To classify the silted algebras of type $A_n$, we proposed the \emph{rooted quiver with relation}--a quiver with relation that includes a vertex (referred to as the \emph{root}) of the quiver. In this context, the path algebra of a rooted quiver with relation is defined as the path algebra of the underlying quiver with relation, that is, the quotient of the path algebra of the underlying quiver modulo the ideal generated by the relations. 

A \emph{rooted subquiver with relation} is a subquiver with relation that contains the root, and it is \emph{full} if its subquiver is a full subquiver and its relations encompass all relations involving the subquiver. It is widely known that the endomorphism algebra of a tilting module over $A_n$ is a connected subquiver of the genealogical tree. 

Let $n$ be a positive integer. We put 
\begin{align*}
\cq(n)&=\{\text{full connected rooted subquivers with relation of the genealogical tree}\\
&\hspace{20pt} \text{ with $n$ vertices}\}.
\end{align*}
It is evident that elements of $\cq(n)$ are pairwise non-isomorphic as rooted quivers with relation, but they can be isomorphic as quivers with relation. Furthermore, the path algebras of two elements of $\cq(n)$ are isomorphic if and only if these two elements are isomorphic as quivers with relation. Put
\begin{align*}
\cq_h(n)&=\{R\in\cq(n)\mid R \text{ has trivial relation}\},\\
\cq_{nh}(n)&=\{R\in\cq(n)\mid R \text{ has non-trivial relations}\}.
\end{align*}
We say that the rooted quivers with relation in $\cq_h(n)$ are \emph{hereditary} and those in $\cq_{nh}(n)$ are \emph{non-hereditary}. Clearly, $|\cq_h(n)|=2^{n-1}$. Moreover, if $R$ and $R'$ are different elements of $\cq_{nh}(n)$, then they are not isomorphic as quivers with relation.

Let $\ct(A_n)$ be the set of isomorphism classes of basic tilting modules over $A_n$, and let $\varepsilon_1\colon\ct(A_n)\to\ca_t(A_n)$ be the map of taking the endomorphism algebra. With each $T\in \ct(A_n)$, we associate a rooted quiver with relation $\varepsilon'(T) \in \cq(n)$. Indeed, by \cite[Lemma 4.8]{XieYangZhang25}, the map $\varepsilon': \ct(A_n)\rightarrow \cq(n)$ is a bijective and $\varepsilon_1$ is the composition of $\varepsilon'$ with the map of taking a rooted quiver with relation to its path algebra. In particular, Let 
\begin{align*}
\ct_{nh}(A_n)&=\{T\mid  \End(T) \text{ is non-hereditary}\},\\
\ca_{nht}(A_n)&=\{C\mid C \text{ is non-hereditary}\}.
\end{align*}
We have
\begin{corollary}\label{cor:A_n-bijective}
$\varepsilon_1\colon\ct_{nh}(A_n)\to\ca_{nht}(A_n)$ is a bijective.
\end{corollary}

Now we consider the set $\ct(\Lambda_n)$, $\ca_{nht}(\Lambda_n)$ and the map $\varepsilon\colon\ct(\Lambda_n)\to\ca_t(\Lambda_n)$. By Proposition \ref{prop:the-forms-of-tilting-modules}, we have the following useful corollary.

\begin{corollary}\label{cor:bijective-of-tilting-and-tilted-algebra}
$\varepsilon\colon\ct_{nh}(\Lambda_n)\to\ca_{nht}(\Lambda_n)$ is a bijective for any $n\geq 5$.
\end{corollary}
\begin{proof}
Let $T$ be a tilting module over $\Lambda_n$. If $T$ is of the form (1) in Proposition \ref{prop:the-forms-of-tilting-modules} and $\varepsilon_1(T_1)\in \ca_{nht}(A_n)$, then $\varepsilon(T)\in \ca_{nht}(\Lambda_n)$, because the quiver with relation of $T$ is obtained by adding an arrow from the vertex of $P(1)$ to the root of $\tau^{-1}T_1$, this operation does not change the relation. In this case, by Corollary \ref{cor:A_n-bijective}, $\varepsilon\colon\ct_{nh}(\Lambda_n)\to\ca_{nht}(\Lambda_n)$ is a bijective. Furthermore, assume that $T$ is of the form (4) in Proposition \ref{prop:the-forms-of-tilting-modules}. In this case, the quiver with relation of $T$ is obtained by adding an arrow from the root of $T_2$ to some vertex of $T_1$. If $\varepsilon_1(T_2)\in \ca_{nht}(A_n)$, then $\varepsilon(T)\in \ca_{nht}(\Lambda_n)$. On the other hand, if $\varepsilon_1(T_2)\in \ca_{nht}(\Lambda_{i-1})$, then we also have $\varepsilon(T)\in \ca_{nht}(\Lambda_n)$. Other cases can be proved similarly.
\end{proof}

\begin{remark}\label{rem:tilted-algebras-non-isomorphism}
\begin{itemize}
  \item[(1)] Corollary \ref{cor:bijective-of-tilting-and-tilted-algebra} is not hold for $\Lambda_4$, See \cite[Section 3.3.1]{Xing21}.
  \item[(2)] In Proposition \ref{prop:the-forms-of-tilting-modules} (4), $T'$ is a tilting module over some sunquiver of $Q$. If $\tau^{-1}T'$ has no non-trivial injective direct summands, then $\tau^{-2}T'$ is also a tilting module over the sunquiver of $Q$. Thus, we can obtain two tilting modules $\tau^{-1}T'\oplus T_2$ and $\tau^{-2}T'\oplus T_2$. Note that $\End(\tau^{-1}T')\cong \End(\tau^{-2}T')$, but $\End(\tau^{-1}T'\oplus T_2)\not\cong \End(\tau^{-2}T'\oplus T_2)$.
\end{itemize}
\end{remark}

\subsubsection{The hereditary tilted algebras of type $\Lambda_n$}

Let $a_{ht}(\Lambda_n)$ be the number of isoclasses of basic hereditary tilted algebras of type $\Lambda_n$. We have

\begin{proposition}\label{prop:hereditary-tilted-algebras}
The number of the hereditary tilted algebras of type $\Lambda_{n}$ is
\[a_{ht}(\Lambda_{n})=\left\{\begin{array}{ll}

4&\mathrm{if}\ n=4,\\

3\times 2^{n-3}&\mathrm{if}\ n\geq 5.

\end{array}\right.\]
\end{proposition}

\begin{proof}
If $n=4$, then by \cite[Example 3.10]{Xing21}, there are 4 hereditary tilted algebras of type $\Lambda_{4}$. If $n\geq 5$, then the hereditary tilted algebras of type $\Lambda_{n}$ can only occur in Proposition \ref{prop:the-forms-of-tilting-modules} (1), (2) and (3). In (1), by Remark \ref{rem:AR-quiver-parts}, if $T_1$ is a tilting module of form $P(2)\oplus \widetilde{T}$, where $\widetilde{T}$ is a tilting module in the wing of $P(3)$ and $\varepsilon'(\widetilde{T})\in \cq_{h}(n-2)$, then the quiver with relation of $T$ is obtained by adding an arrow from the vertex of $P(1)$ to the root of $\tau^{-1}(T_1)$. This shows that $\End(T)$ is hereditary. Similarly, in (2) and (3), if $\varepsilon'(T_2)\in \cq_{h}(n-2)$, then $\End(T)$ are also hereditary. Recall that $|\cq_h(n)|=2^{n-1}$, thus we have $a_{ht}(\Lambda_{n})=3\times 2^{n-3}$.
\end{proof}

\subsubsection{The non-hereditary tilted algebras of type $\Lambda_n$}
In this subsection, we count the number of non-hereditary tilted algebras of type $\Lambda_{n}$.
By Remark \ref{rem:tilted-algebras-non-isomorphism}, to count the number of tilted algebras of type $\Lambda_n$, we must compute the number of tilting modules that do not contain non-trivial injective direct summands. We first recall some facts on the tilting modules over $A_n$. The AR-quiver of $\mod A_{n}$ is as follows.
$$\begin{xy}
0;<3pt,0pt>:<0pt,3pt>::
(0,5)*+{}="16",
(-50,-45)*+{}="17",
(-40,-45)*+{}="18",
(10,5)*+{}="19",
(10,-5)*+{}="20",
(20,-5)*+{}="21",
(-30,-45)*+{}="22",
(-20,-45)*+{}="23",
(30,-25)*+{}="24",
(40,-25)*+{}="25",
(10,-45)*+{}="26",
(20,-45)*+{}="27",
(0,-50)*+{\mathrm{Figure \ 3: The~AR{-}quiver~of~\mod A_{n}}},
 (0,0)*+{P(1)}="1",
(-10,-10)*+{P(2)}="2",
(10,-10)*+{I(n{-}1)}="3",
(-20,-20)*+{\cdots}="4",
(0,-20)*+{\cdots}="5",
(20,-20)*+{\cdots}="6",
(-30,-30)*+{P(n{-}1)}="7",
(-10,-30)*+{\cdots}="8",
(10,-30)*+{\cdots}="9",
(30,-30)*+{I(2)}="10",
(-40,-40)*+{P(n)}="11",
(-20,-40)*+{S(n-1)}="12",
(0,-40)*+{\cdots}="13",
(20,-40)*+{S(2)}="14",
(40,-40)*+{I(1)}="15",
\ar"2";"1", \ar"2";"5", \ar"5";"3", \ar"1";"3", \ar"4";"2", \ar"3";"6",
\ar"7";"4", \ar"6";"10", \ar"5";"9", \ar"9";"14", \ar"8";"5",
\ar"11";"7", \ar"7";"12", \ar"10";"15",\ar"14";"10", \ar"12";"8",
\ar@{.}"16";"17",\ar@{.}"16";"19",\ar@{.}"19";"18",\ar@{.}"17";"18"
\ar@{.}"20";"21",\ar@{.}"21";"23",\ar@{.}"20";"22",\ar@{.}"22";"23"
\ar@{.}"24";"25",\ar@{.}"25";"27",\ar@{.}"24";"26",\ar@{.}"26";"27"
\end{xy}$$
We partition all indecomposable modules in Figure 3 into $n$ groups via dotted borders, labeling these borders sequentially from left to right as $1, \ldots, n$. For each $i\in\{1,2,\ldots,n\}$, we define $\delta(n)_i$ as the number of tilting modules over $A_n$  that include all indecomposable modules from groups labeled $\leq i$ as direct summands and contain at least one direct summand from the $i{-}th$ group. It is easy to see that
\begin{center}
$\delta(n)_{1}=\delta(n-1)_{1}=\cdots=\delta(1)_{1}=1.$
\end{center}
We denote by $t(A_n)$ the number of isoclasses of basic tilting modules over $A_n$. Then we have $\sum\limits_{i=1}^{n}\delta(n)_{i}=t(A_n)=\frac{1}{n+1}(\begin{smallmatrix}2n\\n\end{smallmatrix})$ (See \cite[Theorem 1]{ObaidNaumanFakiehRingel15}). Moreover, we have the following useful formula.

\begin{proposition} \label{prop:the-number-of-tilting-modules-of-A_{n}}
$\delta(n)_{i}=\delta(n)_{i-1} +\delta(n-1)_{i}$.
\end{proposition}

\begin{proof}
By the AR-quiver of $\mod A_{n}$, we obtain that
\begin{center}
$\delta(n)_{i}=\delta(n-1)_{1} +\delta(n-1)_{2}+\cdots+\delta(n-1)_{i-1}+\delta(n-1)_{i}$.
\end{center}
 Thus, we have $\delta(n)_{i}=\delta(n)_{i-1} +\delta(n-1)_{i}$.
\end{proof}

\begin{remark} \label{rem:number-of-tilting-modules}
(1) $\delta(n)_{2}=n-1$;

(2) $\delta(n)_{n-1}=\delta(n)_{n}=t(A_{n-1})$.
\end{remark}

The following table contains the first values of $\delta(n)_{i}$. The numbers in the first row of the table represent $i$.

\begin{remark}
$$\begin{tabular}{|c||c|c|c|c|c|c|c|c| p{<20>}}
\hline
$\delta(n)_{i}$ & 1 & 2 & 3 & 4 & 5 & 6 & 7 & 8\\
\hline
$\delta(3)$ &  1 & 2 & 2 & 0 & 0 & 0 & 0 & 0\\
\hline
$\delta(4)$ &  1 & 3  & 5 & 5 & 0  & 0 & 0 & 0  \\
\hline
$\delta(5)$ &  1 & 4  & 9 & 14 & 14  & 0 & 0 & 0 \\
\hline
$\delta(6)$ &  1 & 5  & 14 & 28 & 42  & 42 & 0 & 0 \\
\hline
\end{tabular} $$
\end{remark}

\begin{corollary}\label{cor:number-of-tilting-contains-I(1)}
The number of isoclasses of basic tilting modules over $A_n$ with $I(1)$ as a direct summand is $t(A_{n-1})$.
\end{corollary}

We specify that $t(A_0)=1$, then we have the following result, See \cite[Remark 4.3]{XieYangZhang25}.

\begin{corollary}\label{cor:number-of-tilting-contains-P(n)}
The number of isoclasses of basic tilting modules over $A_n$ with $P(n)$ as a direct summand is $t(A_{n-1})$.
\end{corollary}

Next, we compute the number of isomorphism classes of basic tilting modules over $\Lambda_{n}$ that exclude injective modules as direct summands, up to isomorphism. By Remark \ref{rem:AR-quiver-parts}, we will consider the subquiver $\overrightarrow{\mathbb{A}}_{n-1}$: $$\begin{xy}
(-10,0)*+{1}="1",
(0,0)*+{3}="3",
(12,0)*+{\cdots}="4",
(28,0)*+{n-1}="5",
(42,0)*+{n}="6",
\ar"3";"1", \ar"4";"3", \ar"5";"4", \ar"6";"5",
\end{xy}$$ in $Q$. 
Here, we focus on the indecomposable modules over the path algebra of the subquiver $\overrightarrow{\mathbb{A}}_{n-1}$ within the AR-quiver of $\mod \Lambda_{n}$. In this context, $I(1)$ stands as the only indecomposable injective module.

For $m\in \mathbb{N}$, we define $t^{m}(A_{n})$ as the total number of isomorphism classes consisting of two types of modules: basic tilting modules $M$ over $A_{n}$ which do not contain injective modules as direct summands and modules $N$ satisfying $N=\tau^{-m}_{A_{n}} M$ with $|N|=|A_{n}|$, where $A_{n}$ corresponds to the path algebra of the subquiver of $Q$. Moreover, we denote by $t^{m}(\Lambda_{n})$ the number of isomorphism classes of basic tilting modules $N$ over $\Lambda_{n}$ such that $|\tau^{-m}_{\Lambda_{n}} N|=|\Lambda_{n}|$. Then, we have

\begin{proposition} \label{prop:the-tilting-modules-of-A-not-contain-injective-modules}
$t^{m}(A_{n})=\sum\limits_{i=1}^{n-m}\delta(n)_{i}\times(n-m-i+1)$.
\end{proposition}

\begin{proof}
Clearly, the tilting module corresponding to $\delta(n)_1$ has all projective modules of $A_n$ as direct summands. Thus, by the AR-quiver of $\mod A_n$, we conclude that the number of modules M over $A_n$ satisfying $|\tau^{-m}_{A_n} M|$ = $|A_n|$ is $\delta(n)_1 \times (n-m)$.
\end{proof}

According to the classification of the tilting modules over $\Lambda_n$ in Proposition \ref{prop:the-forms-of-tilting-modules} , we have the following result.

\begin{proposition} \label{prop:number-of-tilting-not-contain-injective}
$t^{m}(\Lambda_{n})=(t_{1})^{m}(\Lambda_{n})+(t_{2})^{m}(\Lambda_{n})+(t_{3})^{m}(\Lambda_{n})+(t_{4})^{m}(\Lambda_{n}),$ where
\begin{itemize}
\item[(1)] $(t_{1})^{m}(\Lambda_{n})=\delta(n-1)_{1}\times(n-m-1)+\cdots+\delta(n-1)_{n-2}\times(2-m)$;

\item[(2)] $(t_{2})^{m}(\Lambda_{n})=\delta(n-1)_{1}\times(n-m-2)+\cdots+\delta(n-1)_{n-3}\times(2-m)$;

\item[(3)] $(t_{3})^{m}(\Lambda_{n})=t(A_{n-4})\times t_{\Lambda_{4}}^{m+1}+\cdots+t(A_{1})\times t_{\Lambda_{n-1}}^{m+1}$;

\item[(4)] $(t_{4})^{m}(\Lambda_{n})=t(A_{n-4})\times t_{A_{3}}^{m+1}+\cdots+t(A_{1})\times t_{A_{n-2}}^{m+1}$.
\end{itemize}
\end{proposition}

\begin{proof}
If $T$ is of one of the forms (2), (3) and (6) in Proposition \ref{prop:the-forms-of-tilting-modules}, then by Proposition \ref{prop:tilting-modules-of-A}, $T=P(1)\oplus T'$, where $T'$ can be viewed as a tilting module in the wing of $P(2)$. Using the AR-quiver of $\mod A_{n-1}$, the number of isomorphism classes of basic tilting modules over $\Lambda_n$ in this case is therefore given by $\delta(n-1)_{1}\times(n-m-1)+\cdots+\delta(n-1)_{n-2}\times(2-m)$. 
Similarly, if $T$ is of one of the forms (1), (4) and (5) in Proposition \ref{prop:the-forms-of-tilting-modules}, we can derive the counts for $(t_{2})^{m}(\Lambda_{n})$, $(t_{3})^{m}(\Lambda_{n})$ and $(t_{4})^{m}(\Lambda_{n})$, respectively. 
\end{proof}

For convenience, we give the first values of $t^{m}(A_{n})$ and $t^{m}(\Lambda_{n})$ in the following table.

\begin{remark}
$$\begin{tabular}{|c||c|c|c|c|c|c| p{<20>}}
\hline
m & $t^{m}(A_{3})$ & $t^{m}(A_{4})$ & $t^{m}(A_{5})$ & $t^{m}(\Lambda_{4})$ & $t^{m}(\Lambda_{5})$ & $t^{m}(\Lambda_{6})$ \\
\hline
1 &  4 & 14 & 48 & 5 & 21 & 83 \\
\hline
2 &  1 & 5  & 20 & 1 & 6  & 28 \\
\hline
\end{tabular} $$
\end{remark}

\begin{corollary}
The number of isomorphism classes of basic tilting modules over $\Lambda_{n}$ that exclude injective modules as direct summands, up to isomorphism, is $t^{1}(\Lambda_{n})$. 
\end{corollary}

\begin{remark}
Let $T$ be a tilting module of the form (4) in Proposition \ref{prop:the-forms-of-tilting-modules} over $\Lambda_n$. Put $i=5$. Consequently, $T_1$ can be regarded as a tilting module over $\Lambda_4$. Note that there exist two tilting modules $T'$ and $T''$ over $\Lambda_4$ which do not contain injective modules as direct summands and have isomorphic endomorphism algebras, see \cite[Section 3.3.1]{Xing21}. In this context, $\End(\tau^{-1}T'\oplus \widetilde{T})\cong\End(\tau^{-1}T''\oplus \widetilde{T})$, where $\widetilde{T}$ is a tilting module of the wing of $P(5)$.
\end{remark}  

Let $a_{nht}(A_{n})$ and $a_{nht}(\Lambda_{n})$ be the number of isoclasses of basic non-hereditary tilted algebras of type $A_n$ and $\Lambda_n$, respectively. By \cite[Proposition 4.10]{XieYangZhang25}, we have 
\begin{align}
a_{nht}(A_n)&=|\cq(n)|-|\cq_h(n)|=\frac{1}{n+1}(\begin{smallmatrix}2n\\n\end{smallmatrix})-2^{n-1},\\
a_{t}(A_n)&=\frac{1}{n+1}(\begin{smallmatrix}2n\\n\end{smallmatrix})+[1-(-1)^{n}]\times 2^{[\frac{n}{2}]-2}-2^{n-2}\nonumber.
\end{align}
By Corollary \ref{cor:bijective-of-tilting-and-tilted-algebra}, we can compute the number of non-hereditary tilted algebras of type $\Lambda_n$.

\begin{proposition}\label{prop:number-of-non-hy-tilted}
\begin{align*}
a_{nht}(\Lambda_{n})&=a_{nht}(A_{n-2})+a_{nht}(A_{n-1})+t(A_{n-2})+\sum\limits^{n-4}_{i=1}t(A_{i})\times(t(A_{n-i-2})-t(A_{n-i-3}))\\
&+\sum\limits^{n}_{j=5} t^{1}(\Lambda_{j-1})\times t(A_{n+1-j})+\sum\limits^{n}_{k=4}t(A_{n+1-k})\times(t(A_{k-2})-t(A_{k-3})).
\end{align*}
\end{proposition}

\begin{proof}

Case (1): $T$ is of the form (1) in Proposition \ref{prop:the-forms-of-tilting-modules}. By Proposition \ref{prop:tilting-modules-of-A},  
\begin{itemize}
  \item If $T_1=P(2)\oplus T'$, where $T'$ is a basic tilting module of the wing of $P(3)$, then by the AR-quiver of $\mod \Lambda_{n}$, the number of non-hereditary tilted algebras $\End(T)$ of type $\Lambda_{n}$ is $a_{nht}(A_{n-2})$; 
  \item If $T_1=P(2)\oplus T'$, where $T'$ is a basic tilting module of the wing of $\tau^{-1}P(1)$ (see Remark \ref{rem:AR-quiver-parts}), then the number of non-hereditary tilted algebras $\End(T)$ of type $\Lambda_{n}$ is $t(A_{n-2})-t(A_{n-3})$. Indeed, if $T'$ is a basic tilting module of the wing of $\tau^{-1}P(1)$, then the tilted algebras $\End(T)$ all satisfy a commutative relation. By Remark \ref{rem:I(1)-injective} and Corollary \ref{cor:number-of-tilting-contains-I(1)}, we obtain the formula for $\End(T)$ in this subcase; 
  \item If $T_1$ is of the form (2) in Proposition \ref{prop:tilting-modules-of-A}, by Remark \ref{rem:AR-quiver-parts}, then $\End(T)$ has a relation through $\tau^{-1}P(2)$. Thus, the number of non-hereditary tilted algebras $\End(T)$ of type $\Lambda_{n}$ is $\sum\limits^{n-4}_{i=1}t(A_{i})\times(t(A_{n-i-2})-t(A_{n-i-3}))$.
\end{itemize}

Case (2): $T$ take the forms (2), (3) and (6) as specified in Proposition \ref{prop:the-forms-of-tilting-modules}. In this case, by Proposition \ref{prop:tilting-modules-of-A}, $T$ can be regarded as the form $P(1)\oplus T'$, where $T'$ is a tilting module of the wing of $P(2)$. Thus, the count of non-hereditary tilted algebras $\End(T)$ of type $\Lambda_{n}$ is $a_{nht}(A_{n-1})$.

Case (3): $T$ is of the form (4) in Proposition \ref{prop:the-forms-of-tilting-modules}.
\begin{itemize}
  \item If $i=4$, then $T_1=\tau^{-1}T'$, where $T'$ can be viewed as a tilting module over $\begin{xy}
(-10,0)*+{1}="2",
(0,0)*+{3}="3",
(10,0)*+{2}="4",
\ar"2";"3", \ar"4";"3",
\end{xy}$. This implies that $\End(T_1\oplus T_2)$ has a relation $$\begin{xy}
 (-15,15)*+{P(4)}="1",
(0,0)*+{I(n)}="2",
(15,15)*+{I(1) (I(2))}="3",
(-6,6)*+{ }="4",
(6,6)*+{ }="5",
\ar_{\beta}"1";"2", \ar_{\alpha}"2";"3", \ar@/^0.6pc/@{.}"4";"5",
\end{xy}$$ with $\alpha\beta=0$. Thus, the number of non-hereditary tilted algebras $\End(T_1\oplus T_2)$ of type $\Lambda_{n}$ is $t(A_{n-3})$.
  \item If $i\geq 5$, then $T_1=\tau^{-1}T'$, where $T'$ can be viewed as a tilting module over $\Lambda _i$. Note that $\Hom(P(i),\tau^{-m}P(1))=0$ or $\Hom(P(i),\tau^{-m}P(2))=0$ with $m\geq 2$. Thus, the number of non-hereditary tilted algebras $\End(T_1\oplus T_2)$ of type $\Lambda_{n}$ is $\sum\limits^{n}_{i=5} t^{1}(\Lambda_{i-1})\times t(A_{n+1-i})$.
\end{itemize}

Case (4): $T$ is of the form (5) in Proposition \ref{prop:the-forms-of-tilting-modules}.
In this case, there are no paths from $P(i)$ to $\tau^{-j}P(1)$ or $\tau^{-j}P(2)$ for some positive integer $j$. Thus, by Corollary \ref{cor:number-of-tilting-contains-I(1)}, the number of non-hereditary tilted algebras $\End(P(1)\oplus T_1\oplus T_2)$ of type $\Lambda_{n}$ is $\sum\limits^{n}_{i=4}t(A_{n+1-i})\times(t(A_{i-2})-t(A_{i-3}))$. This completes the proof.
\end{proof}

Now we can give a formula of the number of tilted algebras of type $\Lambda_{n}$.

\begin{proposition} \label{prop:tilted-algebras-of-type-D}
\[a_{t}(\Lambda_{n})
=a_{ht}(\Lambda_{n})+a_{nht}(\Lambda_{n})
=\left\{\begin{array}{ll}

7&\mathrm{if}\ n=4,\\

3\times 2^{n-3}+a_{nht}(\Lambda_{n})&\mathrm{if}\ n\geq 5,

\end{array}\right.\]
where $a_{nht}(\Lambda_{n})$ comes from Proposition \ref{prop:number-of-non-hy-tilted}.
\end{proposition}

\begin{example}\label{exm:tilted-algebras-of-5-and-6}
\begin{itemize}
\item[(1)] If $n=4$, then by Proposition \ref{prop:number-of-non-hy-tilted}, we have 4 non-hereditary tilted algebras of type $\Lambda_4$. However, in this case, the non-hereditary tilted algebras $\End(P(1)\oplus \tau^{-1}T_1)$ and $\End(P(1)\oplus P(2)\oplus P(4)\oplus\tau^{-1}S(3))$ are isomorphic. Thus, by Proposition \ref{prop:hereditary-tilted-algebras}, there are 7 tilted algebras of type $\Lambda_{4}$. See \cite[Example 3.10]{Xing21}.
    
\item[(2)] If $n=5$, then $a_{nht}(\Lambda_{5})=5+2+5+6+2+3=23$. Thus, $a_{t}(\Lambda_{5})=12+23=35$. See \cite[Section 3.3.2]{Xing21}.

\item[(3)] If $n=6$, then $a_{nht}(\Lambda_{6})=20+5+10+21+26+5+6+9=102$. Thus, $a_{t}(\Lambda_{6})=24+102=126$.
\end{itemize}
\end{example}

\subsubsection{Some subset of $\ca_t(\Lambda_n)$} Throughout this subsection, assume $n\geq 5$.
Let $\ct^1(\Lambda_n)$ denote the set of isoclasses of  basic tilting modules over $\Lambda_n$ that have $P(n)$ as a direct summand. Define $\ca_t^1(\Lambda_n)=\{\End(T)\mid T\in\ct^1(\Lambda_n)\}$ and let $a_t^1(\Lambda_n)=|\ca_t^1(\Lambda_n)|$.

\begin{lemma}
\label{lem:tilting-modules-and-tilted-algebras-with-P(n)}
$\varepsilon\colon\ct^1(\Lambda_n)\to\ca_t^1(\Lambda_n)$ is bijective.
\end{lemma}

\begin{proof}
By Proposition \ref{prop:the-forms-of-tilting-modules}, elements in $\ca_t^1(\Lambda_n)$ are all non-hereditary except $\End(P(1)\oplus P(2)\oplus\ldots \oplus P(n))$. 
Thus, by Corollary \ref{cor:bijective-of-tilting-and-tilted-algebra}, $\varepsilon\colon\ct^1(\Lambda_n)\to\ca_t^1(\Lambda_n)$ is bijective.
\end{proof}

\begin{lemma}\label{lem:the-number-of-tilted-with-P(n)}
\begin{align*}
a_t^1(\Lambda_n)&=t(\Lambda_{n-2})+t(\Lambda_{n-4})+\sum\limits^{n}_{i=5} t^{1}(\Lambda_{i-1})\times t(A_{n-i})\\
&+\sum\limits^{n}_{j=4}t(A_{n-j})\times(t(A_{j-2})-t(A_{j-3})).
\end{align*}
\end{lemma}

\begin{proof}
According to the proof of Proposition \ref{prop:number-of-non-hy-tilted}, elements of $\ct^1(\Lambda_n)$ can only occur in Case (2), (3) and (4). In case (2), $T'$ is a tilting module of the wing of $P(2)$, by Corollary \ref{cor:number-of-tilting-contains-P(n)} and Lemma \ref{lem:tilting-modules-and-tilted-algebras-with-P(n)}, the number of isoclasses of basic tilting modules over $\Lambda_n$ with $P(n)$ as a direct summand is $t(A_{n-2})$. Similarly, we can obtain the other cases.
\end{proof}

Let $\ca_t^2(\Lambda_n)$ denote the set of isoclasses of endomorphism algebras of basic tilting modules over $\Lambda_n$ which does not contain $P(n)$ as a direct summand and let $a_t^2(\Lambda_n)=|\ca_t^2(\Lambda_n)|$. Then 
\begin{lemma}
$a_t^2(\Lambda_n)=a_t(\Lambda_n)-a_t^1(\Lambda_n)$.
\end{lemma}

\begin{proof}
By Corollary \ref{cor:bijective-of-tilting-and-tilted-algebra} and Lemma \ref{lem:tilting-modules-and-tilted-algebras-with-P(n)}, $\ca_t^1(\Lambda_n)\cap \ca_t^2(\Lambda_n)=\emptyset$. Since $\ca_t(\Lambda_n)=\ca_t^1(\Lambda_n)\cup \ca_t^2(\Lambda_n)$, it follows that $a_t^2(\Lambda_n)=a_t(\Lambda_n)-a_t^1(\Lambda_n)$.
\end{proof}

\begin{remark}\label{rem:some-subsets}
We can also consider the sets $\ca_t^1(A_n)$ and $\ca_t^2(A_n)$. Let $a_t^1(A_n)=|\ca_t^1(A_n)|$ and $a_t^2(A_n)=|\ca_t^2(A_n)|$. Then, by \cite[Lemma 4.13]{XieYangZhang25}, we have $a_t^2(A_{n})=a_{t}(A_{n})-t(A_{n-1})+1$ for any $n\geq 2$. Moreover, let $\ca_t^3(A_n)$ denote the set of isoclasses of endomorphism algebras of basic tilting modules over $A_n$ which contains $P(2)$ as a direct summand and let $a_t^3(A_n)$ denote its cardinality. Then $a_t^3(A_n)=a_t(A_{n-1})$. 
For $n=1,2$, put $\ca_t^4(A_n)=\ca_t(A_n)$, which has one element only. For $n\geq 3$, let $\ca_t^4(A_n)=\ca_t^A(A_n)\cup\{\End(P(1)\oplus\ldots\oplus P(n-2)\oplus P(n)\oplus \tau^{-2}P(n))\}$.  Note that this extra element belongs to $\ca_t^1(A_n)$ but not to $\ca_t^2(\Lambda_n)$. By \cite[Lemma 4.15]{XieYangZhang25}, we obtain that the cardinality $a_t^{4}(A_n)$ of $\ca_t^4(A_n)$
$$a_t^{4}(A_n)=
\begin{cases} 1 & \text{if } n=1,2,\\
a_{t}(A_{n})-t(A_{n-1})+2 & \text{if }n\geq 3.
\end{cases}
$$
\end{remark}

\subsection{Silted algebras of type $\Lambda_{n}$} \label{subsec:silted-algebras-of-Lambda_{n}}
In this subsection, we give a classification of the silted algebras of type $\Lambda_n$ and compute the number of silted algebras of type $\Lambda_{n}$, up to isomorphism. 

\subsubsection{The classification of silted algebras of type $\Lambda_{n}$}
\label{subsec:the-classification-of-silted-algebras-of-Lambda_{n}}

Since all tilting modules over $\Lambda_n$ are 2-term silting complexes,  all tilted algebras of type $\Lambda_n$ are silted algebras. It follows that $\ca_t(\Lambda_n)$ is a subset of $\ca_s(\Lambda_n)$. Let $S$ be a 2-term silting complex over $\Lambda_n$. By Proposition \ref{prop:classification-of-silting}, $S$ belongs to one of the following three cases:
\begin{itemize}
  \item [(I)] $S=P(1)[1]\oplus T$, where $T$ is a basic tilting module of the wing of $P(2)$;
  \item[(II)] $S=P(1)[1]\oplus P(2)[1]\oplus T$, where $T$ is a basic tilting module of the wing of $P(3)$;
  \item[(III)] $S=T_1\oplus T_2$, where $T_{1}$ can be viewed as a basic 2-term tilting complex over the path algebra of some subquiver of $Q$ and $T_{2}$ is a basic tilting module of the wing of $P(i)$ for any $4\leq i\leq n$. 
\end{itemize}
For $k=\mathrm{I,II,III}$, put
\begin{align*}
\ca_s^k(\Lambda_n)&=\{\End(S)\mid \text{$S$ belongs to the family (k)}\}/\cong.
\end{align*}
It is clear that 
\[
\ca_s(\Lambda_n)=\ca_t(\Lambda_n)\cup\ca_s^{\rm I}(\Lambda_n)\cup \ca_s^{\rm II}(\Lambda_n)\cup\ca_s^{\rm III}(\Lambda_n).
\]

Case (I): If $S=P(1)[1]\oplus T$, where $T$ is a basic tilting module of the wing of $P(2)$, then by the classification of the tilting modules over $A_n$ in Proposition \ref{prop:tilting-modules-of-A}, we have the following three subcases:
\begin{itemize}
  \item [(a)] If $T=P(2)\oplus T'$ (where $T'$ is a tilting module of the wing of $P(3)$), then $$\Hom(T,P(1)[1])=\Hom(P(1)[1],T)=0.$$ This implies $\End(S)=\End(P(1)[1])\times \End(T)$. Denote by $\ca_s^{\rm{Ia}}(\Lambda_n)$ the set of isoclasses of $\End(S)$ for such $S$'s. 
      Thus, $\ca_s^{\rm{Ia}}(\Lambda_n)=\ca_t^3(A_{n{-}1})\times_s\ca_t(A_{1})$;
  \item [(b)] If $T=P(2)\oplus T'$ (where $T'$ is a tilting module of the wing of $I(n{-}1)$ in the AR-quiver $\mod A_{n-1}$), then by Remark \ref{rem:AR-quiver-parts}, $\Hom(T',P(1)[1])\neq 0$ and $\End(S)$ all satisfy a commutative relation.  Denote by $\ca_s^{\rm{Ib}}(\Lambda_n)$ the set of isoclasses of $\End(S)$ for such $S$'s. Hence, $\ca_s^{\rm{Ib}}(\Lambda_n)=\ca_t(\mathbb{D}_n)$;
  \item [(c)] If $T$ is of the form (2) of Proposition \ref{prop:tilting-modules-of-A}, then we claim that $\End(S)$ is a strictly shod algebra (See Section \ref{secsec:strictly-shod-algebras}). Denote by $\ca_s^{\rm{Ic}}(\Lambda_n)$ the set of isoclasses of $\End(S)$ for such $S$'s. Thus, $\ca_s^{\rm{Ic}}(\Lambda_n)=\ca_{ss}(\Lambda_n)$.
\end{itemize}

To summarise, we have $$\ca_s^{\rm I}(\Lambda_n)=\ca_s^{\rm{Ia}}(\Lambda_n)\cup \ca_s^{\rm{Ib}}(\Lambda_n)\cup \ca_s^{\rm{Ic}}(\Lambda_n).$$

Case (II): If $S=P(1)[1]\oplus P(2)[1]\oplus T$, where $T$ is a basic tilting module of the wing of $P(3)$, then $$\Hom(T,P(1)[1])=\Hom(P(1)[1],T)=0~~\mathrm{and}~~\Hom(T,P(2)[1])=\Hom(P(2)[1],T)=0.$$ This implies that 
$$\ca_s^{\rm II}(\Lambda_n)=(\ca_t(A_{n-2})\times_s\ca_t(A_{1}))\times_s\ca_t(A_{1}).$$

Case (III): If $S=T_1\oplus T_2$, then $\Hom(T_1,T_2)=0=\Hom(T_2,T_1)$, so $\End(S)=\End(T_1)\times\End(T_2)$. By Proposition \ref{prop:classification-of-silting}, we have the following two subcases:
\begin{itemize}
  \item [(a)] $i=4$. In this subcase,  $\End(T_1)$ is a tilted algebra of type $A_3$.  Denote by $\ca_s^{\rm{IIIa}}(\Lambda_n)$ the set of isoclasses of $\End(S)$ for such $S$'s. Then, $\ca_s^{\rm{IIIa}}(\Lambda_n)=\ca_t(A_{n-3})\times_s\ca_t(A_{3})$;
  \item [(b)] $i>4$. In this subcase, $\End(T_1)$ is a tilted algebra of type $\Lambda_{i{-}1}$ for any $5\leq i\leq n$. Denote by $\ca_s^{\rm{IIIb}}(\Lambda_n)$ the set of isoclasses of $\End(S)$ for such $S$'s. Then, $\ca_s^{\rm{IIIb}}(\Lambda_n)=\bigcup\limits_{i=5}^{n}\ca_t(\Lambda_{i-1})\times_s\ca_t(A_{n-i+1})$.
\end{itemize}
Then, we have 
$$\ca_s^{\rm III}(\Lambda_n)=\ca_s^{\rm{IIIa}}(\Lambda_n)\cup\ca_s^{\rm{IIIb}}(\Lambda_n).$$

\subsubsection{The proof of Theorem \ref{thm:classification-of-silted-algebras-of-D}}
\label{sec:the-proof-of-theorem-4.1}
In this subsection, we prove Theorem \ref{thm:classification-of-silted-algebras-of-D}. Put 
\begin{center}
$\cb_1=\ca_t(\Lambda_n)$,
$\cb_2=\ca_s^{\rm{Ib}}(\Lambda_n)$,
$\cb_3=\ca_s^{\rm{IIIb}}(\Lambda_n)$,
$\cb_4=\ca_s^{\rm{Ia}}(\Lambda_n)$,
\end{center}
\begin{center}
$\cb_5=\ca_s^{\rm{II}}(\Lambda_n)$,
$\cb_6=\ca_s^{\rm{IIIa}}(\Lambda_n)$,
$\cb_7=\ca_s^{\rm{Ic}}(\Lambda_n)$.
\end{center}
By the classification of silted algebras of type $\Lambda_{n}$ in section \ref{subsec:the-classification-of-silted-algebras-of-Lambda_{n}}, 
\[
\ca_s(\Lambda_n)=\ca_t(\Lambda_n)\cup\ca_s^{\rm I}(\Lambda_n)\cup \ca_s^{\rm II}(\Lambda_n)\cup\ca_s^{\rm III}(\Lambda_n)=\cb_1\cup\cb_2\cup\cb_3\cup\cb_4\cup\cb_5\cup\cb_6\cup\cb_7.
\]
Moreover, we have 
\begin{itemize}
  \item [(1)] $\cb_1=\ca_t(\Lambda_n)$;
  \item [(2)] $\cb_2=\ca_t(\mathbb{D}_n)$;
  \item [(3)] $\cb_3=\bigsqcup_{m=4}^{n-1}(\ca_t(\Lambda_m)\times\ca_t(A_{n-m}))$;
  \item [(4)] $\cb_4=\ca_t^{3}(A_{n-1})\times_s\ca_t(A_{1})$;
  \item [(5)] $\cb_5=(\ca_t(A_{n-2})\times_s\ca_t(A_{1}))\times_s\ca_t(A_{1})$;
  \item [(6)] $\cb_6=\ca_t(A_{n-3})\times_s\ca_t(A_{3})$;
  \item [(7)] $\cb_7=\ca_{ss}(\Lambda_n)$. 
\end{itemize}
This completes the proof.

\subsubsection{Strictly shod algebras of type $\Lambda_{n}$}\label{secsec:strictly-shod-algebras}
In this subsection, we prove the claim of Case (Ic) in Section \ref{subsec:the-classification-of-silted-algebras-of-Lambda_{n}}. Let $A$ be a finite-dimensional hereditary algebra and let $T$ be a tilting $A$-module. A classical result states that $\mathrm{gl.dim} \End(T)\leq 2$. Additionally, for a 2-term silting complex $P$ in $K^{b}(\proj A)$, Buan and Zhou \cite[Theorem 1.1 (a)]{BuanZhou18} established $\mathrm{gl.dim} \End(P)\leq 3$. Thus, by Theorem \ref{thm: silted-algebra-to-strictly-shod-algebra}, we only need to study connected silted algebras with global dimension $3$.

To give the main result in this subsection, we require some preliminaries. Let $\widetilde{Q}$ be a finite quiver without loops. An interesting question concerns the global dimension of the path algebra $K\widetilde{Q}$ under different admissible ideals $I$. Poettering \cite{Poettering10} showed in 2010 that there exists an admissible ideal $I$ such that $\mathrm{gldim}(K\widetilde{Q}/I) \leq 2$. Moreover, if $A_{m}$ is a subquiver of $\widetilde{Q}$, then there exists an admissible ideal $I$ such that $\mathrm{gldim}(K\widetilde{Q}/I) \leq k$, where $k, m \in \mathbb{N}$ with $2 \leq k < m$. Expanding on Poettering's results, Yang and Zhang \cite{YangZhang19} studied the global dimension of Nakayama algebras of type $\mathbb{A}_{n}$ and $\widetilde{\mathbb{A}}_{n}$. We begin by recalling definitions from their paper.

Let $K\widetilde{Q}/I$ be the path algebra of type $\mathbb{A}_{n}$. For $i<j$, define the relation $[i,j]$ as $p=\alpha_{i}\alpha_{i+1}\cdots\alpha_{j-1}$. The admissible ideal $I$ is then expressed as $I=\langle[i_{s},j_{s}]|s=1,\ldots,n\rangle$, where $i_{1}<\cdots<i_{n}$ and $j_{1}<\cdots<j_{n}$.

\begin{definition}
For two relations $[i,j]$ and $[r,s]$ satisfying $i<s$, we say they \textit{intersection} if $i<r<j<s$.
\end{definition}

\begin{definition} Let $I=\langle p_{s}=[i_{s},j_{s}]|s=1,\ldots,n\rangle$. A family of $m$ relations $\{p_{s},\ldots,p_{s+m-1}\}$ is said to form an {\it effective intersection} if:

(1) For all $t\in \{s,\ldots,s{+}m{-}2\}$, consecutive relations $p_{t}$ and $p_{t+1}$ intersect;

(2) Each $p_{t}$ (with $t\in\{s,\ldots,s{+}m{-}1\}$) intersects exclusively with its immediate neighbors $p_{t-1}$ and $p_{t+1}$.

\end{definition}
We define $N$ as the maximum value among the numbers of all effective intersection relations in $I$. As shown in \cite[Theorem 2.1]{YangZhang19} the global dimension of $K\widetilde{Q}/I$ is determined by $N$. To aid readability, we include a proof below.

\begin{theorem}\label{thm:the-global-dimension-of-algebra-by-intersect-relations}
Let $K\widetilde{Q}/I$ be an algebra of type $\mathbb{A}_{n}$ and let $I=\langle p_{s}=[i_{s},j_{s}]|s=1,\ldots,n\rangle$. Then $\mathrm{gldim}(K\widetilde{Q}/I)=N+1$.
\end{theorem}

\begin{proof}
Assume that $p_{a},p_{a-1},\ldots,p_{b}$ are the $N$ effective intersection relations of $I$. Then we have
\begin{center}
$i_{a}<i_{a+1}<j_{a}\leq i_{a+2}<j_{a+1}\leq i_{a+3}<\cdots\leq i_{b}<j_{b-1}<j_{b}$.
\end{center}
Thus, for each simple module $S(i_{r})~(a\leq r\leq b)$, we obtain a minimal projective resolution as follows:
\begin{center}
$0\rightarrow P(j_{b})\rightarrow P(j_{b-1})\rightarrow\cdots\rightarrow P(j_{r+1})\rightarrow P(j_{r})\rightarrow P(i_{r+1})\rightarrow P(i_{r})\rightarrow S(i_{r})\rightarrow 0$.
\end{center}
It follows that $\pd S(i_{r})=b-r+2\leq N+1$. In particular, $\pd S(i_{a})=b-a+2=N+1$. Moreover, since $N$ is the maximal number of the effective intersection relations of $I$, it follows that the projective dimension of all simple modules $S(i_{t})$ is less than or equal to $N+1$, where $i_{t}$ are the starting points of the corresponding effective intersection relations.

On the other hand, it is easy to see that for any simple module $S(i)$ with $i\in \widetilde{Q}_{0}\backslash \{i_{r}|1\leq r\leq  n\}$, 
there is a minimal projective resolution:
\begin{center}
$0\rightarrow P(i+1)\rightarrow P(i)\rightarrow S(i)\rightarrow 0$.
\end{center}
This shows that $\pd S(i)=1$. Thus, we have $\mathrm{gldim}(K\widetilde{Q}/I)=N+1$.
\end{proof}

Now we give a classification of strictly shod algebras of type $\Lambda_{n}$. 

\begin{theorem} \label{thm:classification-of-strictly-shod-algebras}
Let $S$ be a 2-term silting complex over $\Lambda_n$ with the form $P(1)[1] \oplus T_1$. Suppose $T_{1}=P(2)\oplus T'\oplus T''$, where $T'$ is a tilting module over the path algebra of the Dynkin quiver of type $\mathbb{A}_{m}$, and $T''$ is a tilting module over the path algebra of the Dynkin quiver of type $\mathbb{A}_{n-m-2}$, $1\leq m\leq n-3$. Then $\End(S)$ is a strictly shod algebra. Furthermore, all strictly shod algebras of type $\Lambda_n$ arise via this construction.
\end{theorem}

\begin{proof}
By Proposition \ref{prop:tilting-modules-of-A}, considering the AR-quiver of $\mod A_{n-1}$, we obtain that $\End(T_{1})$ has a relation
\begin{center}
$\begin{xy}
 (-15,-15)*+{P(n+1-m)}="1",
(0,0)*+{P(2)}="2",
(15,-15)*+{\tau^{-(m+1)}P(m+3)}="3",
(8,-8)*+{ }="5",
(-8,-8)*+{ }="6",
\ar@/^0.6pc/@{.}"5";"6",
\ar^{\beta}"1";"2", \ar^{\alpha}"2";"3",
\end{xy}$
\end{center} with $\alpha\beta=0$. By Remark \ref{rem:AR-quiver-parts}, the indecomposable module $\tau^{-(m+1)}P(m+3)$ in the AR-quiver of $\mod A_{n-1}$ corresponds to either $\tau^{-(m+1)} P(1)$ or $\tau^{-(m+1)}P(2)$ within the AR-quiver of $\mod \Lambda_n$. Note that 
\begin{center}
$\Hom(\tau^{-(m+1)}P(1),P(1)[1])\neq 0$ and $\Hom(P(2),P(1)[1])=0$,
\end{center}
we conclude that $\End(P(1)[1] \oplus T_1)$ contains the following subquiver:
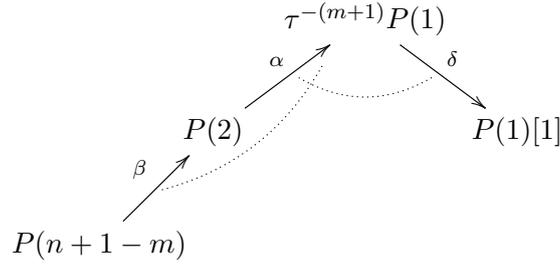
\begin{center}
$\begin{xy}
 (-15,-15)*+{P(n+1-m)}="1",
(0,0)*+{P(2)}="2",
(20,15)*+{\tau^{-(m+1)}P(1)}="3",
(40,0)*+{P(1)[1]}="4",
(15,10)*+{ }="5",
(-8,-8)*+{ }="6",
(30,8)*+{ }="7",
(10,8)*+{ }="8",
\ar@/^0.8pc/@{.}"5";"6",\ar@/^0.8pc/@{.}"7";"8",
\ar^{\beta}"1";"2", \ar^{\alpha}"2";"3", \ar^{\delta}"3";"4",
\end{xy}$
\end{center} with $\alpha\beta=0$ and $\delta\alpha=0$. This defines an algebra of type $\mathbb{A}_{k}$ and with one effective intersection relation consisting of two relations. Moreover, because $T'$ is a tilting module over the path algebra of the Dynkin quiver of type $\mathbb{A}_{m}$, and $T''$ is a tilting module over the path algebra of the Dynkin quiver of type $\mathbb{A}_{n-m-2}$, other components of the quiver of $\End(P(1)[1] \oplus T_1)$ lack effective intersection relations. Consequently, $\End(P(1)[1] \oplus T_1)$ contains only this single effective intersection relation. By Theorem \ref{thm:the-global-dimension-of-algebra-by-intersect-relations}, the global dimension of $\End(P(1)[1] \oplus T_1)$ is 3, that is, $\End(P(1)[1] \oplus T_1)$ is a strictly shod algebra of type $\Lambda_n$. Conversely, Proposition \ref{prop:classification-of-silting} implies that all strictly shod algebras of type $\Lambda_n$ arise through this construction.
\end{proof}

\begin{remark} \label{rem:morphism-composition}
\begin{itemize}
  \item [(1)] The morphism $\delta$ can be viewed as either a single morphism or the composition of morphisms. Moreover, Theorem \ref{thm:classification-of-strictly-shod-algebras} proves the claim of Case (Ic) in Section \ref{subsec:the-classification-of-silted-algebras-of-Lambda_{n}}.
  \item [(2)] Effective intersection can be used to compute the global dimension of a finitely-dimensional algebra and determine the projective dimension of injective modules. In \cite{ZhangLiuLiu25}, we utilize the concept of effective intersection to provide a criterion for the Gorensteiness of string algebras.
\end{itemize}

\end{remark}

At the end of this subsection, we show that every strictly shod algebra of type $\Lambda_{n}$ is a string algebra. 
Recall that tilted algebras of type $\mathbb{A}_{n}$ are gentle, however, it is straightforward to find tilted algebras of type $\Lambda_{n}$ that are not gentle. To begin, we recall the definitions of gentle algebras and string algebras. These two classes of algebras are particularly noteworthy in the representation theory of algebras, and we refer the reader to \cite{ButlerRingel87,AssemBrustleCharbonneauPlamondon10,BaurCoelho21,ArnesenLakingPauksztello16,BekkertMerklen03} for further details.

\begin{definition}
Let $\widetilde{Q}$ be a finite quiver and $I$ an admissible ideal of the path algebra $K\widetilde{Q}$. Then algebra $A=K\widetilde{Q}/I$ is called a {\it string algebra} provided the following conditions are satisfied:
\begin{itemize}
\item[(S1)] for any $i\in \widetilde{Q}_{0}$, there exists at most two arrows starting at $i$ and at most two arrows ending at $i$;

\item[(S2)] for any $a\in \widetilde{Q}_{1}$, there exists at most one arrow $b$ such that $ba\notin I$ and at most one arrow $c$ such that $ac\notin I$;

\item[(S3)] the ideal $I$ are generated by the paths of length greater than or equal to 2.
\end{itemize}
Moreover, algebra $A=K\widetilde{Q}/I$ is called a {\it gentle algebra} if additional the following conditions:
\begin{itemize}
\item[(S2')] for any $a\in \widetilde{Q}_{1}$, there exists at most one arrow $b^{\prime}$ such that $b^{\prime}a\in I$ and at most one arrow $c^{\prime}$ such that $ac^{\prime}\in I$;

\item[(S3')] the ideal $I$ are generated by the paths of length is equal to 2.
\end{itemize}
\end{definition}

\begin{theorem} \label{thm: strictly-shod-algebras-are-string-algebras}
Each strictly shod algebra of type $\Lambda_{n}$ is a string algebra.
\end{theorem}

\begin{proof}
From the proof of Theorem \ref{thm:classification-of-strictly-shod-algebras}, we obtain that the quiver of each strictly shod algebra of type $\Lambda_{n}$ contains the following subquiver:
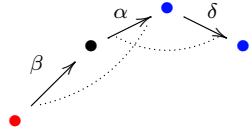
\begin{center}
$\begin{xy}
 (-10,-10)*+{\textcolor[rgb]{1.00,0.00,0.00}{\bullet}}="1",
(0,0)*+{\bullet}="2",
(10,5)*+{\textcolor[rgb]{0.00,0.07,1.00}{\bullet}}="3",
(20,0)*+{\textcolor[rgb]{0.00,0.07,1.00}{\bullet}}="4",
(8,4)*+{ }="5",
(-8,-8)*+{ }="6",
(18,2)*+{ }="7",
(2,2)*+{ }="8",
\ar@/_0.6pc/@{.}"6";"5",\ar@/_0.6pc/@{.}"8";"7",
\ar^{\beta}"1";"2", \ar^{\alpha}"2";"3", \ar^{\delta}"3";"4",
\end{xy}$
\end{center} with $\alpha\beta=0$ and $\delta\alpha=0$. Furthermore, all other components of the quiver of a strictly shod algebra of type $\Lambda_{n}$ are subquivers of tilted algebras of either type  $A_{m}$ or type $A_{n-m-2}$. Thus, we only need to analyze the connecting vertices. By the AR-quiver of $\mod A_{m}$, there are at most three arrows connecting to the red vertex: two starting from it and one ending at it. Additionally, the following relation holds:
 \begin{center}
$
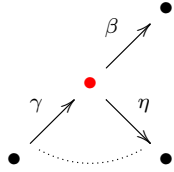
\begin{xy}
 (-10,-10)*+{\bullet}="1",
(0,0)*+{\textcolor[rgb]{1.00,0.00,0.00}{\bullet}}="2",
(10,-10)*+{\bullet}="3",
(10,10)*+{\bullet}="4",
(8,-8)*+{ }="5",
(-8,-8)*+{ }="6",
\ar@/^0.6pc/@{.}"5";"6",
\ar^{\gamma}"1";"2", \ar^{\eta}"2";"3", \ar^{\beta}"2";"4",
\end{xy}$
\end{center} with $\eta\gamma=0$. By Remark \ref{rem:morphism-composition}, tilted algebras of type $A_{n-m-2}$ fall into two cases:
\begin{itemize}
  \item Case (1) (Non-hereditary tilted algebras): The silted algebras of type $\Lambda_n$ contain the subquiver:
  \begin{center}
$
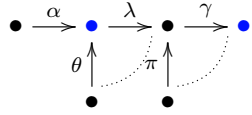
\begin{xy}
 (0,0)*+{\bullet}="1",
(10,0)*+{\textcolor[rgb]{0.00,0.07,1.00}{\bullet}}="2",
(-10,0)*+{\textcolor[rgb]{0.00,0.07,1.00}{\bullet}}="3",
(-20,0)*+{\bullet}="4",
(-10,-10)*+{\bullet}="5",
(0,-10)*+{\bullet}="6",
(8,0)*+{ }="7",
(0,-8)*+{ }="8",
(-2,0)*+{ }="9",
(-10,-8)*+{ }="10",
\ar@/^0.6pc/@{.}"7";"8",\ar@/^0.6pc/@{.}"9";"10",
\ar^{\gamma}"1";"2", \ar^{\alpha}"4";"3", \ar^{\pi}"6";"1",\ar^{\lambda}"3";"1",\ar^{\theta}"5";"3",
\end{xy}$
\end{center} with $\gamma\pi=0$, $\lambda\theta=0$ and $\gamma\lambda=\delta$. 
  \item Case (2) (Hereditary tilted algebras): The silted algebras of type $\Lambda_n$ contain the subquiver:
  \begin{center}
$
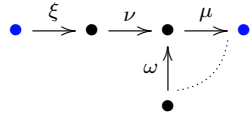
\begin{xy}
 (10,0)*+{\textcolor[rgb]{0.00,0.07,1.00}{\bullet}}="1",
(0,0)*+{\bullet}="2",
(-10,0)*+{\bullet}="3",
(-20,0)*+{\textcolor[rgb]{0.00,0.07,1.00}{\bullet}}="4",
(0,-10)*+{\bullet}="5",
(8,0)*+{ }="7",
(0,-8)*+{ }="8",
\ar@/^0.6pc/@{.}"7";"8",
\ar^{\mu}"2";"1", \ar^{\nu}"3";"2",\ar^{\xi}"4";"3",\ar^{\omega}"5";"2",
\end{xy}$
\end{center} with $\mu\omega=0$ and $\mu\nu\xi=\delta$. 
\end{itemize}
In all these cases, strictly shod algebras are shown to be string algebras.
\end{proof}


\subsubsection{The cardinalities of $\cb_{i}$, $1\leq i\leq 7$} 
In this subsection, we count the cardinalities $b_i$ of $\cb_{i}$ for all $1\leq i\leq 7$. Moreover, we also provide a formula for counting the number of these strictly shod algebras. 

\begin{proposition} \label{prop: the-number-of-the-silted-algebras-induced-by-2-term-silting}
$a_s(\Lambda_n)=b_1+b_2+b_3+b_4+b_5+b_6+b_7$, where
\begin{itemize}
  \item [(1)] $b_1=a_{t}(\Lambda_n)$, which is given in Proposition \ref{prop:tilted-algebras-of-type-D};
  \item [(2)] $b_2+b_4+b_7=
\begin{cases}
t(A_{3}) & \text{if } n=4,\\
t(A_{n-1})-1& \text{if } n>4.
\end{cases}$
  \item [(3)] $b_3=\sum\limits_{i=4}^{n-1}a_{t}(\Lambda_{i})\times a_{t}(A_{n-i})$;
  \item [(4)] $b_5=a_{t}(A_{n-2})$.
  \item [(5)] $b_6=\begin{cases}
3\times a_{t}(A_{n-3}) & \text{if } n\neq6,\\
3\times a_{t}(A_{n-3})-3& \text{if } n=6.
\end{cases}$
\end{itemize}
\end{proposition}

\begin{proof} According to the classification of the silted algebras in Section \ref{subsec:the-classification-of-silted-algebras-of-Lambda_{n}}, there are the following three cases:
\begin{itemize}
  \item [(1)] $S=P(1)[1]\oplus T$, where $T$ is a basic tilting module of the wing of $P(2)$. In this case, each tilting module $T$ gives rise to a silted algebra. However, by \cite[Lemmma 4.8(b)]{XieYangZhang25}, 
\begin{center}
$\End(P(2)\oplus P(3)\oplus\cdots\oplus P(n-1)\oplus\tau^{-1}P(n))\cong \End(P(2)\oplus P(3)\oplus \bigoplus\limits_{m=4}^{n}\tau^{3-m}P(m))$.
\end{center}
 If $n=4$, then
\begin{align*}
 P(2)\oplus P(3)\oplus\cdots\oplus P(n-1)\oplus\tau^{-1}P(n)&=P(2)\oplus P(3)\oplus\tau^{-1}P(4)\\
 &=P(2)\oplus P(3)\oplus \bigoplus\limits_{m=4}^{n}\tau^{3-m}P(m);
\end{align*}
  If $n>4$, then $P(2)\oplus P(3)\oplus\cdots\oplus P(n-1)\oplus\tau^{-1}P(n)$ and $P(2)\oplus P(3)\oplus \bigoplus\limits_{m=4}^{n}\tau^{3-m}P(m)$ are two different tilting modules. 
This implies that \\
$b_2+b_4+b_7=
\begin{cases}
t(A_{3}) & \text{if } n=4,\\
t(A_{n-1})-1& \text{if } n>4.
\end{cases}$
  \item [(2)] $S=P(1)[1]\oplus P(2)[1]\oplus T$, where $T$ is a basic tilting module of the wing of $P(3)$. In this case, $\Hom(T,P(1)[1])=0$ and $\Hom(T,P(2)[1])=0$, this shows that each tilted algebra of type $A_{n{-}2}$ gives rise to a silted algebra of type $\Lambda_n$. Thus, $b_5=a_{t}(A_{n-2})$.
  \item [(3)] $S=T_1\oplus T_2$. 
\begin{itemize}
  \item If $T_{1}$ is viewed as a basic 2-term tilting complex over the path algebra of the subquiver $\begin{xy}
(-10,0)*+{1}="2",
(0,0)*+{3}="3",
(10,0)*+{2}="4",
\ar"3";"2", \ar"3";"4",
\end{xy}$ and $T_{2}$ is a basic tilting module in the wing of $P(4)$, then the additive closure of indecomposable modules $I(1)$, $I(2)$, $P(1)[1]$, $P(2)[1]$ and $P(3)[1]$ can be viewed as the shift of the AR-quiver for the path algebra of the subquiver $\begin{xy}
(-10,0)*+{1}="2",
(0,0)*+{3}="3",
(10,0)*+{2}="4",
\ar"3";"2", \ar"3";"4".
\end{xy}$ within the AR-quiver of $\mod \Lambda_n$. Hence, the number of silted algebras of type $\Lambda_n$ is $3\times a_{t}(A_{n-3})$. Specifically, when $n=6$, the subquiver
$\begin{xy}
(-10,0)*+{4}="2",
(0,0)*+{5}="3",
(10,0)*+{6}="4",
\ar"3";"2", \ar"4";"3",
\end{xy}$ induces three tilted algebras that are isomorphic to the tilted algebras arising from the subquiver
$\begin{xy}
(-10,0)*+{1}="2",
(0,0)*+{3}="3",
(10,0)*+{2}="4",
\ar"3";"2", \ar"3";"4",
\end{xy}$. See \cite[Example 3.6]{Xing21}. As a result, we obtain $b_6$.
  \item If $T_1$ is viewed as a basic 2-term tilting complex over $\Lambda_{i}$ and $T_2$ is a basic tilting module in the wing of $P(i+1)$ for some $i\geq 4$, then by the AR-quiver of $K^{[-1,0]}(\proj \Lambda_{n})$, it is easy to see that the number of silted algebras of type $\Lambda_n$ is $\sum\limits_{i=4}^{n-1}a_{t}(\Lambda_{i})\times a_{t}(A_{n-i})$. So, we obtain $b_3$.
\end{itemize}
 
\end{itemize}
This completes the proof.
\end{proof}

\begin{example}\label{exm:silted-algebras-of-5-and-6}
  \begin{itemize}
    \item [(1)] If $n=4$, by Example \ref{exm:tilted-algebras-of-5-and-6}, $b_1=7$, $b_5=1$, $b_3=0$, $b_2+b_4+b_7=5$ and $b_6=3$. However, there are two isomorphic elements between $\cb_6$ and $\cb_4$, and there is one isomorphic element between $\cb_2$ and $\cb_1$. Thus, $a_{s}(\Lambda_{4})=13$. See \cite[Section 3.3.1]{Xing21}.
    \item [(2)] If $n=5$, then by Proposition $\ref{prop: the-number-of-the-silted-algebras-induced-by-2-term-silting}$ and Example \ref{exm:tilted-algebras-of-5-and-6}, $a_{s}(\Lambda_{5})=35+13+7+4+3=62$. See \cite[Section 3.3.2]{Xing21}.
    \item [(3)] If $n=6$, then by Proposition $\ref{prop: the-number-of-the-silted-algebras-induced-by-2-term-silting}$ and Example \ref{exm:tilted-algebras-of-5-and-6}, $a_{s}(\Lambda_{6})=126+41+7+35+10+12-3=228$.
  \end{itemize}
\end{example}

By Theorem \ref{thm:classification-of-strictly-shod-algebras}, we can compute $b_7=a_{ss}(\Lambda_{n})$: the number of strictly shod algebras of type $\Lambda_{n}$.

\begin{proposition} \label{prop:the-number-of-strictly-shod-algebras}
\begin{align*}
a_{ss}(\Lambda_{n}) &=a_{nht}(A_{n-1})-2a_{nht}(A_{n-2})  \\
& =\frac{1}{n}(\begin{smallmatrix}{2n-2}\\n-1\end{smallmatrix})-\frac{2}{n-1}(\begin{smallmatrix}{2n-4}\\n-2\end{smallmatrix}).
\end{align*}
\end{proposition}

\begin{proof}
By the proof of Theorem \ref{thm:classification-of-strictly-shod-algebras}, $\End(T_1)$ is a non-hereditary tilted algebra of type $A_{n-1}$.
Thus, Corollary \ref{cor:A_n-bijective} and Proposition \ref{prop:tilting-modules-of-A} implies that $a_{ss}(\Lambda_{n}) =a_{nht}(A_{n-1})-2a_{nht}(A_{n-2})$. This proposition then follows from $(4.1)$.
\end{proof}

The following table contains the first values of $t(A_n)$, $a_{nht}(A_n)$ and $a_{ss}(\Lambda_{n})$.

$$\begin{tabular}{|c||c|c|c|c|c|c|c|c|c| p{<20>}}
\hline
$n$ & $1$ & $2$ & $3$ & $4$ & $5$ & $6$ & $7$ & $8$ & $9$\\
\hline
$t(A_n)$ & 1 & 2 & 5 & 14 & 42 & 132 & 429 & 1430 & 4862 \\
\hline
$a_{nht}(A_n)$ & 0 & 0 & 1 & 6 & 26 & 100 & 365 & 1302 & 4606 \\
\hline
$a_{ss}(\Lambda_{n})$ &$\setminus$& $\setminus$ & $\setminus$ & 1 & 4 & 14 & 48 & 165 & 572 \\
\hline
\end{tabular} 
$$

\section{Silted algebras of type $\Gamma_{n}$}
\label{s:silted-algebra-mutated-orientation}
In this section, we classify up to isomorphism the basic silted algebras and strictly shod algebras of type $\Gamma_{n}$ ($n\geq 4$), the path algebra of the quiver
$$Q'=\begin{xy}
(-10,5)*+{1}="0",
(-10,-5)*+{2}="1",
(0,0)*+{3}="2",
(12,0)*+{\cdots}="3",
(28,0)*+{n-1}="4",
(42,0)*+{n}="5",
\ar"2";"0",\ar"2";"1", \ar"3";"2", \ar"4";"3", \ar"4";"5",
\end{xy}.$$
Then we give some recurrence formulas to compute the number of these silted algebras and strictly shod algebras. 

Since the classification of silted algebras of type $\Gamma_n$ is closely related to the classification of silted algebras of type $A_n$, the path algebra of the quiver:
$$\begin{xy}
(-12,0)*+{1}="1",
(0,0)*+{2}="2",
(12,0)*+{\cdots}="3",
(28,0)*+{n-1}="4",
(42,0)*+{n}="5",
\ar"2";"1", \ar"3";"2",\ar"4";"3",\ar"5";"4",
\end{xy}$$ and silted algebras of type $B_n$, the path algebra of the quiver 
$$\begin{xy}
(-12,0)*+{1}="1",
(0,0)*+{2}="2",
(12,0)*+{\cdots}="3",
(28,0)*+{n-1}="4",
(42,0)*+{n}="5",
\ar"2";"1", \ar"3";"2",\ar"4";"3",\ar"4";"5",
\end{xy}.$$
So, we recall some results in \cite[Section 5.3]{XieYangZhang25}. 
Let $T\in K^{[-1,0]}(\proj A_n)$ be a basic presilting complex that does not contain $P(n)$ as a direct summand. Define
$M=T\oplus X$ ($X=\tau^{-1}P(n)[1]$) and $M^*=T\oplus P(n)$, where $(-)^*$ is treated as an operator. Then, 
$M^*$ is a basic 2-term silting complex over $A_{n}$ if and only if $M$ is a basic 2-term silting complex over $B_n$.

In order to classify silted algebras of type $\Gamma_n$, we need the following result. See \cite[Lemma 4.16]{XieYangZhang25}.

\begin{lemma}\label{lem:T-to-X}
Let $T$ be a tilting module over $A_n$ that contains $P(n)$ as a direct summand. Then 
\begin{itemize}
  \item [(1)] If $T$ contains $P(n-1)$ as a direct summand, then $\Hom(T/P(n),X)=0$;
  \item [(2)] If $T$ does not contain $P(n-1)$ as a direct summand, then $\Hom(T/P(n),X)\neq0$;
\end{itemize}
\end{lemma}

Suppose $M$ is a basic 2-term silting complex over $B_n$, with $X$ as a direct summand of $M$, and that $M^*$ is a tilting module over $A_n$. Then, by Lemma \ref{lem:T-to-X}, $\End(M)$ falls into three classes:
\begin{itemize}
  \item [(1)] Tilted of type $A_{n-1}\times A_1$;
  \item [(2)] Tilted of type $A_n$;
  \item [(3)] Tilted of type $\mathbb{A}_n$ but not of type $A_n$.
\end{itemize}
Denote by $\ca_s^{\mu}(A_{n})$ the set of isoclasses of such silted algebras of type $A_n$ and by $\ca_s^{\mu}(\mathbb{A}_{n})$ the set of isoclasses of such silted algebras of type $\mathbb{A}_n$ but not of type $A_n$. In particular, if $M^*$ is a tilting module over $A_n$ that contains $P(2)$ as a direct summand, we denote by $\ca_s^{\mu,2}(\mathbb{A}_{n})$ the set of isoclasses of $\End(M)$ for such $M$'s.
 
\subsection{A classification of the silted algebras of type $\Gamma_{n}$}

Put
\begin{align*}
\ca_t^{\mu}(\mathbb{D}_n)&:=\{\text{$\End(S) \mid \End(S)$ is a basic tilted algebra of type $\mathbb{D}_n$, and $X$ is a direct summand of $S$}\}/\cong~,\\
\ca_s(\Gamma_n)&:=\{\text{basic silted algebras of type $\Gamma_n$}\}/\cong~,\\
\ca_{ss}(\Gamma_n)&:=\{\text{basic strictly shod algebras of type $\Gamma_n$}\}/\cong~.
\end{align*}
Let $a_s(\Gamma_n)$ and $a_{ss}(\Gamma_n)$ denote the cardinalities of $\ca_s(\Gamma_n)$ and $\ca_{ss}(\Gamma_n)$, respectively. Then we have a classification of the silted algebras of type $\Gamma_n$ as follows:

\begin{theorem}\label{thm:classification-of-silted-algebras-of-mutation-type-D}
$\ca_s(\Gamma_n)=\cc_1\sqcup\cc_2\sqcup\cc_3\sqcup\cc_4\sqcup\cc_5\sqcup\cc_6\sqcup\cc_7\sqcup\cc_8\sqcup\cc_9\sqcup\cc_{10}\sqcup\cc_{11}\sqcup\cc_{12}\sqcup\cc_{13}\sqcup\cc_{14}$, where
\begin{itemize}
  \item [(1)] $\cc_1=\ca_t(\Lambda_n)$;
  \item [(2)] $\cc_2=\ca_{t}(\mathbb{D}_n)\cup \ca_t^{\mu}(\mathbb{D}_n)$;
  \item [(3)] $\cc_3=\bigcup\limits_{m=2}^{n-4}(\ca_t^4(A_m) \times_s \ca_t(\Lambda_{n-m}))\cup (\ca_t^1(\Lambda_{n-1})\times_s \ca_t(A_1))$;
  \item [(4)] $\cc_4=\bigcup\limits_{m=2}^{n-4}(\ca_t(\mathbb{A}_m) \times_s \ca_t(\Lambda_{n-m}))$;
  \item [(5)] $\cc_5=((\ca_t^3(A_{n-1})\cup \ca_s^{\mu}(A_{n-1}))\cap \ca_t^4(A_{n-1}))\times_s \ca_t(A_{1})$;
  \item [(6)] $\cc_6=\ca_s^{\mu,2}(\mathbb{A}_{n-1})\times_s \ca_t(A_{1})$;
  \item [(7)] $\cc_7=\ca_t^4(A_{n-3})\times_s \ca_t(A_{3})$;
  \item [(8)] $\cc_8=\ca_s^{\mu}(\mathbb{A}_{n-3})\times_s \ca_t(B_{3})$;
  \item [(9)] $\cc_9=((\ca_t^3(A_{n-2})\cap \ca_t^1(A_{n-2}))\cup\ca_t^2(A_{n-2}))\times_s \ca_t(A_{1})\times_s \ca_t(A_{1})$;
  \item [(10)] $\cc_{10}=\ca_s^{\mu}(\mathbb{A}_{n-2})\times_s \ca_t(A_{1})\times_s \ca_t(A_{1})$;
  \item [(11)] $\cc_{11}=\ca_t^{1}(A_{n-3})\times_s \ca_t(A_{1})\times_s \ca_t(A_{1})\times_s \ca_t(A_{1})$;
  \item [(12)] $\cc_{12}=\ca_t^{1}(A_{n-4})\times_s \ca_t(A_{1})\times_s \ca_t(B_{3})$;
  \item [(13)] $\cc_{13}=\bigcup\limits_{i=4}^{n-1}\ca_t^{1}(A_{n-i-1})\times_s \ca_t(A_{1})\times_s \ca_t(\Lambda_{i})$;
  \item [(14)] $\cc_{14}=\ca_{ss}(\Gamma_n)$. 
\end{itemize}
\end{theorem}

We present the proof of this Theorem in Section~\ref{sec:the-proof-of-theorem-5.1}. Silted algebras of type $\Gamma_{n}$ forming the following families: (1) elements in $\cc_1$ are tilted algebras of type $\Lambda_{n}$; (2) elements in $\cc_2$ are tilted algebras of type $\mathbb{D}_n$ but not of type $\Lambda_n$; (3) elements in $\cc_3$ are tilted algebras of type $\Lambda_{m}\times A_{n-m}$, where $4\leq m\leq n-1$; (4) elements in $\cc_4$ are tilted algebras of type $\Lambda_{m}\times \mathbb{A}_{n-m}$, where $4\leq m\leq n-1$; (5) elements in $\cc_5$ are tilted algebras of type $A_{n-1}\times A_{1}$; (6) elements in $\cc_6$ are tilted algebras of type $\mathbb{A}_{n-1}\times A_{1}$; (7) elements in $\cc_7$ are tilted algebras of type $A_{n-3}\times A_{3}$; (8) elements in $\cc_8$ are tilted algebras of type $\mathbb{A}_{n-3}\times A_{3}$; (9) elements in $\cc_{9}$ are tilted algebras of type $A_{n-2}\times A_{1}\times A_{1}$; (10) elements in $\cc_{10}$ are tilted algebras of type $\mathbb{A}_{n-2}\times A_{1}\times A_{1}$; 
(11) elements in $\cc_{11}$ are tilted algebras of type $A_{n-4}\times A_{1}\times A_{3}$;
(13) elements in $\cc_{13}$ are tilted algebras of type $A_{n-m-1}\times A_{1}\times A_{1}\times \Lambda_{m}$ for any $4\leq m\leq n-1$;
(14) elements in $\cc_{14}$ are strictly shod algebras.

\subsection{Silted algebras of type $\Gamma_{n}$}
In this subsection, we provide a classification of the silted algebras of type $\Gamma_{n}$ and prove Theorem \ref{thm:classification-of-silted-algebras-of-mutation-type-D}.
As a consequence, we obtain some formulas for counting the number of silted algebras of type $\Gamma_{n}$.

Note that $Q'$ is derived from $Q$ through inverting the arrow starting from the unique source $n$. Consequently, there is a BGP reflection functor $F\colon K^{b}(\proj \Gamma_{n})\xrightarrow{\simeq}K^{b}(\proj \Lambda_n)$, as referenced in \cite{BernsteinGelfandPonomarev73}. It follows that the AR-quiver of $K^{[-1,0]}(\proj \Gamma_{n})$ can be identified with the full subquiver drawn in Figure 4 of the AR-quiver of $K^b(\proj \Lambda_n)$. Furthermore, $K^{[-1,0]}(\proj \Gamma_{n})$ can be recognized as the additive closure within $K^{b}(\proj \Lambda_n)$ containing the indecomposable objects from this AR-quiver. We will use this AR-quiver to study the silted algebras of type $\Gamma_n$. For instance, when we refer to $S$ as a 2-term silting complex over $\Gamma_n$, we intend that $S$ constitutes a silting complex over $\Lambda_n$ whose direct summands lie within this subquiver. Note that there is a distinguished vertex $X=\tau^{-1}P(n)[1]$. 


$$\begin{xy}
0;<3pt,0pt>:<0pt,3pt>::
(70,0)*+{P(1)[1]}="21",
(70,-10)*+{P(2)[1]}="22",
(60,-20)*+{P(3)[1]}="23",
(50,-30)*+{\cdots}="24",
(40,-40)*+{P(n)[1]}="25",
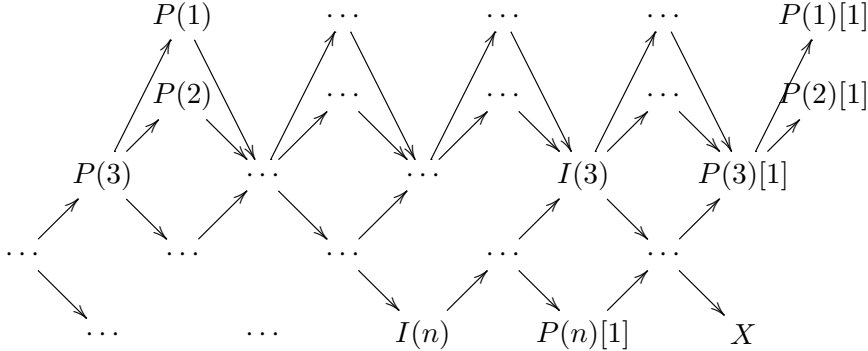
(20,-50)*+{\mathrm{Figure \ 4: The~Auslander–Reiten~quiver~of~K^{[-1,0]}(\proj \Gamma_n)}},
(-10,0)*+{P(1)}="1",
(10,0)*+{\cdots}="2",
(30,0)*+{\cdots}="3",
(50,0)*+{\cdots}="4",
(-10,-10)*+{P(2)}="5",
(10,-10)*+{\cdots}="6",
(30,-10)*+{\cdots}="7",
(50,-10)*+{\cdots}="8",
(-20,-20)*+{P(3)}="9",
(0,-20)*+{\cdots}="10",
(20,-20)*+{\cdots}="11",
(40,-20)*+{I(3)}="12",
(-30,-30)*+{\cdots}="13",
(-10,-30)*+{\cdots}="14",
(10,-30)*+{\cdots}="15",
(30,-30)*+{\cdots}="16",
(-20,-40)*+{\cdots}="18",
(0,-40)*+{\cdots}="19",
(20,-40)*+{I(n)}="20",
(60,-40)*+{X}="17",
\ar"1";"10", \ar"2";"11", \ar"3";"12", \ar"5";"10", \ar"6";"11", \ar"7";"12",
\ar"9";"1", \ar"9";"5", \ar"10";"2", \ar"10";"6", \ar"11";"3",
\ar"11";"7", \ar"12";"4", \ar"12";"8", \ar"13";"9",\ar"13";"18",\ar"14";"10",\ar"10";"15",\ar"15";"11",
\ar"15";"20",\ar"20";"16",\ar"16";"12",\ar"9";"14",
\ar"4";"23",\ar"8";"23",\ar"12";"24",\ar"16";"25",\ar"24";"17",
\ar"25";"24",\ar"24";"23",\ar"23";"21",\ar"23";"22",
\end{xy}$$

Consider $T\in K^{[-1,0]}(\proj \Lambda_n)$ as a basic presilting complex that excludes $P(n)$ as a direct summand. Put 
$S=T\oplus X$ and $S'=T\oplus P(n)$. We denote this relationship by $S=S'^*$ and $S'=S^*$, treating $(-)^*$ as an operator. Similar to the proof of \cite[Proposition 5.3]{XieYangZhang25}, we have the following result.

\begin{proposition}\label{prop:equ-of-2-term-silting-complexes}
$S$ is a basic 2-term silting complex over $\Gamma_{n}$ if and only if $S'$ is a basic 2-term silting complex over $\Lambda_n$.
\end{proposition}

Using $X$ and $(-)^*$ we divide basic 2-term silting complexes $S$ over $\Gamma_n$ into the following three families:
\begin{itemize}
\item[(I)] $S$ is a basic 2-term silting complex over $\Gamma_n$ such that $X$ is not a direct summand of $S$;
\item[(II)] $S$ is a basic 2-term silting complex over $\Gamma_n$ such that $X$ is a direct summand of $S$ and $S^*$ is a tilting module over $\Lambda_n$;
\item[(III)] $S$ is a basic 2-term silting complex over $\Gamma_n$ such that $X$ is a direct summand of $S$ and $S^*$ is not a tilting module over $\Lambda_n$.
\end{itemize}
For $k=\mathrm{I,II,III}$, put
\begin{align*}
\ca_s^k(\Gamma_n)&=\{\End(S)\mid \text{$S$ belongs to the family (k)}\}/\cong.
\end{align*}
It is clear that 
\[
\ca_s(\Gamma_n)=\ca_s^{\rm I}(\Gamma_n)\cup \ca_s^{\rm II}(\Gamma_n)\cup\ca_s^{\rm III}(\Gamma_n).
\]

\subsubsection{{\rm Case~I}} \label{sec:case I}
Let $S$ be a basic 2-term silting complex over $\Gamma_n$ such that $X$ is not a direct summand of $S$. Then $S$ is a basic 2-term silting complex over $\Lambda_n$. By Proposition \ref{prop:classification-of-silting}, $S$ belongs to one of the following six families:
\begin{itemize}
  \item[(I1)] $T$, where $T$ is a basic tilting module over $\Lambda_n$ which does not contain $P(n)$ as a direct summand;
  \item[(I2)] $\tau^{-1}T$, where $T$ is a basic tilting module over $\Lambda_n$ which contains at least one injective module as a direct summand;
  \item[(I3)] $S=P(1)[1]\oplus T_1$, where $T_1$ is a basic tilting module of the wing of $P(2)$ which does not contain $P(n)$ as a direct summand;
  \item[(I4)] $S=P(1)[1]\oplus P(2)[1]\oplus T_1$, where $T_1$ is a basic tilting module of the wing of $P(3)$ which does not contain $P(n)$ as a direct summand;
  \item[(I5)] $S=T_1\oplus T_2$, where $T_{1}$ can be viewed as a basic 2-term tilting complex over the path algebra of quiver $\begin{xy}
(-10,0)*+{1}="2",
(0,0)*+{3}="3",
(10,0)*+{2}="4",
\ar"3";"2", \ar"3";"4",
\end{xy}$ and $T_{2}$ is a basic tilting module of the wing of $P(4)$ which does not contain $P(n)$ as a direct summand. 
  \item[(I6)] $S=T_1\oplus T_2$, where $T_1$ can be viewed as a basic 2-term tilting complex over $\Lambda_{i}$ and $T_2$ is a basic tilting module of the wing of $P(i+1)$ with $i\geq 4$ which does not contain $P(n)$ as a direct summand.
\end{itemize}
For $1\leq l\leq 6$, put
\begin{align*}
\ca_s^{\rm{I}l}(\Gamma_n)&=\{\End(S)\mid \text{$S$ belongs to the family (Il)}\}/\cong.
\end{align*}
Then
\begin{align}
\label{eqn:ca-I}
\ca_s^{\rm I}(\Gamma_n)&=\ca_s^{\rm I1}(\Gamma_n)\sqcup \ca_s^{\rm I2}(\Gamma_n)\sqcup\ca_s^{\rm I3}(\Gamma_n)\sqcup\ca_s^{\rm I4}(\Gamma_n)\sqcup\ca_s^{\rm I5}(\Gamma_n)\sqcup\ca_s^{\rm I6}(\Gamma_n),\nonumber\\
\ca_s^{\rm I1}(\Gamma_n)&=\ca_t^{2}(\Lambda_n), \ca_s^{\rm I2}(\Gamma_n)=\ca_t(\Lambda_n), \ca_s^{\rm I4}(\Gamma_n)=\ca_t^{2}(A_{n-2}), \\
\ca_s^{\rm I5}(\Gamma_n)&=\ca_t(B_{3})\times_s\ca_t^{2}(A_{n-3}), \ca_s^{\rm I6}(\Gamma_n)=\bigcup\limits_{i=4}^{n-1}\ca_t^{2}(A_{n-i})\times_s \ca_t(\Lambda_{i}).\nonumber
\end{align}

In the case (I3), according to the proof of Theorem \ref{thm:classification-of-silted-algebras-of-D}, we have the following three subcases:

\emph{Subcase (I3.1):} $T_1=P(2)\oplus T'$, where $T'$ is a tilting module of the wing of $P(3)$ which does not contain $P(n)$ as a direct summand. This implies that 
$$\End(S)=\End(P(1)[1])\oplus \End(T_1),$$ where $\End(T_1)$ is a tilted algebra of type $A_{n-1}$.
Denote by $\ca_s^{\rm{I3},1}(\Gamma_n)$ the set of isoclasses of $\End(S)$ for such $S$'s. 

\emph{Subcase (I3.2):} $T_1=P(2)\oplus T'$, where $T'$ is a basic tilting module of the wing of $I(n{-}1)$ in the AR-quiver of $\mod A_{n{-}1}$. By Remark \ref{rem:AR-quiver-parts}, $\Hom(\tau^{-1}P(1),P(1)[1])\neq 0$. Moreover, $\Hom(T_1,P(1))=0$ and $\Hom(P(1)[2],T_1)=0$. So $S$ is a 2-term silting complex over $\Lambda_n$, and  hence a 2-term tilting complex over $\Gamma_n$. By Corollary \ref{cor:tilting complex to tilted algebra}, $\End(S)$ is a tilted algebra of type $\mathbb{D}_n$ but not of type $\Lambda_n$.
Denote by $\ca_s^{\rm{I3},2}(\Gamma_n)$ the set of isoclasses of $\End(S)$ for such $S$'s. 

\emph{Subcase (I3.3):} $T_1=P(2)\oplus T'\oplus T''$, where $T'$ is a basic tilting module of the wing of $P(i)$ for some $4\leq i\leq n-1$ which does not contain $P(n)$ as a direct summand and $T''$ is a basic tilting module of the wing of $I(i-2)$ within the AR-quiver of $\mod A_{n-1}$. By Theorem \ref{thm:classification-of-strictly-shod-algebras}, $\End(S)$ is a strictly shod algebra. 
Put
\begin{align*}
\ca_{ss}'(\Lambda_n)=\{\End(S)\mid S \text{ is of the form of subcase (I3.3)}\}.
\end{align*}
Denote by $\ca_{s}^{\rm{I3},3}(\Gamma_n)$ the set of isoclasses of $\End(S)$ for such $S$'s. So $\ca_{s}^{\rm{I3},2}(\Gamma_n)=\ca_{ss}'(\Lambda_n)$.

To summarise, we have 
\begin{align}
\ca_s^{\rm I3}(\Gamma_n)&=\ca_s^{\rm{I3},1}(\Gamma_n)\sqcup \ca_s^{\rm{I3},2}(\Gamma_n)\sqcup \ca_s^{\rm{I3},3}(\Gamma_n).
\end{align}

As a corollary, we have

\begin{corollary}\label{cor:tilted-are-silted-of-B}
All tilted algebras of type $\Lambda_n$ are silted of type $\Gamma_{n}$.
\end{corollary}

\subsubsection{{\rm Case~II}} \label{sec:case II}
Let $S$ be a basic 2-term silting complex over $\Gamma_n$ such that $X$ is a direct summand of $S$ and $S^*$ is a tilting module over $\Lambda_n$. By Proposition \ref{prop:the-forms-of-tilting-modules}, $S^*$ belongs to one of the following four families:
\begin{itemize}
  \item[(II1)] $S^*=P(1)\oplus P(2)\oplus T_{2}$, where $T_{2}$ is a basic tilting module of the wing of $P(3)$ which contains $P(n)$ as a direct summand;
\item[(II2)] $S^*=T_{1}\oplus T_{2}$, where $T_{1}=\tau^{-1}T'$ for some basic tilting module $T'$ over some subquiver of $Q$ and $T_{2}$ is a basic tilting module of the wing of $P(i)$ for some $4\leq i\leq n$ which contains $P(n)$ as a direct summand;
\item[(II3)] $S^*=P(1)\oplus T_{1}\oplus T_{2}$, where $T_{1}=\tau^{-1}T''$ for some basic tilting module $T''$ over some subquiver of $Q$ and $T_{2}$ is a basic tilting module of the wing of $P(i)$ for some $4\leq i\leq n$ which contains $P(n)$ as a direct summand.
\item[(II4)] $S^*=P(1)\oplus P(2)\oplus T_{1}\oplus T_{2}$, where $T_{1}=\tau^{-1}T''$ for some basic tilting module $T''$ over some subquiver of $Q$ and $T_{2}$ is a basic tilting module of the wing of $P(i)$ for some $4\leq i\leq n$ which contains $P(n)$ as a direct summand.
\end{itemize}
For $1\leq w\leq 4$, put
\begin{align*}
\ca_s^{\rm{II}w}(\Gamma_n)&=\{\End((S^*/P(n))\oplus X)\mid \text{$S^*$ belongs to the family (IIw)}\}/\cong.
\end{align*}
Then
\begin{align}
\label{eqn:ca-I}
\ca_s^{\rm II}(\Gamma_n)&=\ca_s^{\rm II1}(\Gamma_n)\sqcup \ca_s^{\rm II2}(\Gamma_n)\sqcup\ca_s^{\rm II3}(\Gamma_n)\sqcup\ca_s^{\rm II4}(\Gamma_n).
\end{align}

In the case (II1), $T=(P(1)\oplus P(2)\oplus T_2)/P(n)$, we have the following two subcases:

\emph{Subcase (II1.1):} $T_2$ contains $P(n{-}1)$ as a direct summand. By Lemma \ref{lem:T-to-X} and Remark \ref{rem:AR-quiver-parts}, $\Hom(T_2, X)=0$, and hence 
$$\End(S)=\End(T)\times \End(X),$$ 
where $\End(T)$ is a tilted algebra of type $\Lambda_{n-1}$. Denote by $\ca_s^{\rm{II1},1}(\Gamma_n)$ the set of isoclasses of $\End(S)$ for such $S$'s. 

\emph{Subcase (II1.2):} $T_2$ does not contain $P(n{-}1)$ as a direct summand. By Lemma \ref{lem:T-to-X} and Remark \ref{rem:AR-quiver-parts}, $\Hom(T_2, X)\neq0$. Moreover, 
$$\Hom(X,T[-1])=0=\Hom(T,X[-1]),$$ this shows that $S$ is a 2-term tilting complex over $\Gamma_n$. It follows that $\End(S)$ is tilted of type $\mathbb{D}_n$ but not tilted of type $\Lambda_n$. Denote by $\ca_s^{\rm{II1},2}(\Gamma_n)$ the set of isoclasses of $\End(S)$ for such $S$'s. Then,
\begin{align}
\ca_s^{\rm{II1}}(\Gamma_n)&=\ca_s^{\rm{II1},1}(\Gamma_n)\sqcup \ca_s^{\rm{II1},2}(\Gamma_n).
\end{align}

In the case (II2), $T=(T_1\oplus T_2)/P(n)$, we have the following two subcases:

\emph{Subcase (II2.1):} $i=4$ and $n\geq 5$. By Remark \ref{rem:I(1)-injective}, $T_1=I(1)\oplus I(2)\oplus I(n)$, it implies that $\Hom(T_1,X)=0$.
\begin{itemize}
  \item [(a)] $T_2$ contains $P(n-1)$ as a direct summand. $\Hom(T_2, X)=0$. Since $\Hom(T_2,T_1)\neq 0$, 
$$\End(S)=\End(T)\times \End(X),$$ 
where $\End(T)$ is a tilted algebra of type $\Lambda_{n-1}$. Denote by $\ca_s^{\rm{II2},1a}(\Gamma_n)$ the set of isoclasses of $\End(S)$ for such $S$'s. 
\end{itemize}

\begin{remark}
If $n=4$, then $\Hom(T_1,X)\neq 0$. Thus $\End(S)$ is a tilted algebra of type $\Lambda_4$.
\end{remark}

\begin{itemize}
  \item [(b)] $T_2$ does not contain $P(n-1)$ as a direct summand. $\Hom(T_2, X)\neq0$. Moreover, $\Hom(T_2,T_1)\neq 0$ and $$\Hom(X,T[-1])=0=\Hom(T,X[-1]),$$ this show that $\End(S)$ is a tilted algebra of type $\mathbb{D}_{n}$. Denote by $\ca_s^{\rm{II2},1b}(\Gamma_n)$ the set of isoclasses of $\End(S)$ for such $S$'s. Then,
\[
\ca_s^{\rm{II2},1}(\Gamma_n)=\ca_s^{\rm{II2},1a}(\Gamma_n)\sqcup\ca_s^{\rm{II2},1b}(\Gamma_n).
\]
\end{itemize}

\emph{Subcase (II2.2):} $i\geq 5$. $T_1$ can be viewed as a basic tilting module over $\Lambda_{i{-}1}$ and $T_2$ is a basic tilting module of the wing of $P(i)$.
\begin{itemize}
  \item [(a)] $T_2$ contains $P(n{-}1)$ as a direct summand. $\Hom(T_2, X)=0$. Since $\Hom(T_2,T_1)\neq 0$, 
$$\End(S)=\End(T)\times \End(X),$$ 
where $\End(T)$ is a tilted algebra of type $\Lambda_{n-1}$. Denote by $\ca_s^{\rm{II2},2a}(\Gamma_n)$ the set of isoclasses of $\End(S)$ for such $S$'s. 
\end{itemize}
\begin{itemize}
  \item [(b)] $T_2$ does not contain $P(n{-}1)$ as a direct summand. $\Hom(T_2, X)\neq0$. By Corollary \ref{cor:bijective-of-tilting-and-tilted-algebra}, $\End(S)$ is a tilted algebra of type $\mathbb{D}_n$. They form two groups: 
      \begin{itemize}
        \item If $i=n$, then $\tau^{m}S$ is a tilting module over $\Lambda_n$ for some positive integer $m$, thus, $\End(S)$ is a tilted algebra of type $\Lambda_{n}$.
        \item If $i<n$, then $S$ is a 2-term tilting complex over $\Gamma_n$, this shows that $\End(S)$ is a tilted algebra of type $\mathbb{D}_n$ but not of type $\Lambda_n$.
      \end{itemize}
      Denote by $\ca_s^{\rm{II2},2b}(\Gamma_n)$ the set of those tilted algebras of type $\Lambda_{n}$ and by $\ca_s^{\rm{II2},2b'}(\Gamma_n)$ the set of those tilted algebras of type $\mathbb{D}_n$ but not of type $\Lambda_n$.
\end{itemize}
Then, 
\begin{align}
\ca_s^{\rm{II2}}(\Gamma_n) &=\ca_s^{\rm{II2},1a}(\Gamma_n)\sqcup\ca_s^{\rm{II2},1b}(\Gamma_n)\sqcup\ca_s^{\rm{II2},2a}(\Gamma_n)\sqcup\ca_s^{\rm{II2},2b}(\Gamma_n)\sqcup\ca_s^{\rm{II2},2b'}(\Gamma_n).
\end{align}

In the case (II3), $T=(P(1)\oplus T_1\oplus T_2)/P(n)$, we have the following two subcases:
 
\emph{Subcase (II3.1):} $T_2$ contains $P(n-1)$ as a direct summand. $\Hom(T_2, X)=0$. Moreover, $\Hom(T_1, X)=0$, and hence 
$$\End(S)=\End(T)\times \End(X),$$ 
where $\End(T)$ is a tilted algebra of type $\Lambda_{n-1}$. Denote by $\ca_s^{\rm{II3},1}(\Gamma_n)$ the set of isoclasses of $\End(S)$ for such $S$'s. 

\emph{Subcase (II3.2):} $T_2$ does not contain $P(n-1)$ as a direct summand. $\Hom(T_2, X)\neq0$. Moreover, $\Hom(T_2, T_1)\neq0$, and 
$$\Hom(X,T[-1])=0=\Hom(T,X[-1]),$$ this shows that $S$ is a 2-term tilting complex over $\Gamma_n$. It follows that $\End(S)$ is tilted of type $\mathbb{D}_n$ but not tilted of type $\Lambda_n$. Denote by $\ca_s^{\rm{II3},2}(\Gamma_n)$ the set of isoclasses of $\End(S)$ for such $S$'s. Then,
\begin{align}
\ca_s^{\rm{II3}}(\Gamma_n)&=\ca_s^{\rm{II3},1}(\Gamma_n)\sqcup \ca_s^{\rm{II3},2}(\Gamma_n).
\end{align}

In the case (II4), $T=(P(1)\oplus P(2)\oplus T_1 \oplus T_2)/P(n)$, we have the following two subcases:
 
\emph{Subcase (II4.1):} $T_2$ contains $P(n-1)$ as a direct summand. $\Hom(T_2, X)=0$. Moreover, $\Hom(T_1, X)=0$, and hence 
$$\End(S)=\End(T)\times \End(X),$$ 
where $\End(T)$ is a tilted algebra of type $\Lambda_{n-1}$. Denote by $\ca_s^{\rm{II4},1}(\Gamma_n)$ the set of isoclasses of $\End(S)$ for such $S$'s. 

\emph{Subcase (II4.2):} $T_2$ does not contain $P(n-1)$ as a direct summand. $\Hom(T_2, X)\neq0$. Moreover, $\Hom(T_2, T_1)\neq0$, and 
$$\Hom(X,T[-1])=0=\Hom(T,X[-1]),$$ this shows that $S$ is a 2-term tilting complex over $\Gamma_n$. It follows that $\End(S)$ is tilted of type $\mathbb{D}_n$ but not tilted of type $\Lambda_n$. Denote by $\ca_s^{\rm{II4},2}(\Gamma_n)$ the set of isoclasses of $\End(S)$ for such $S$'s. Then,
\begin{align}
\ca_s^{\rm{II4}}(\Gamma_n)&=\ca_s^{\rm{II4},1}(\Gamma_n)\sqcup \ca_s^{\rm{II4},2}(\Gamma_n).
\end{align}

\subsubsection{{\rm Case~III}} \label{sec:case III}
$S$ is a basic 2-term silting complex over $\Gamma_n$ such that $X$ is a direct summand of $S$ and $S^*$ is not a tilting module over $\Lambda_n$. According to Proposition \ref{prop:classification-of-silting}, $S^*$ belongs to one of the following four families:
\begin{itemize}
  \item[(III1)] $S^*=P(1)[1]\oplus T_1$, where $T_1$ is a basic tilting module of the wing of $P(2)$ and which has $P(n)$ as a direct summand;
  \item[(III2)] $S^*=P(1)[1]\oplus P(2)[1]\oplus T_1$, where $T_1$ is a basic tilting module of the wing of $P(3)$ and which has $P(n)$ as a direct summand;
  \item[(III3)] $S^*=T_1\oplus T_2$, where $T_{1}$ can be viewed as a basic 2-term tilting complex over the path algebra of quiver $\begin{xy}
(-10,0)*+{1}="2",
(0,0)*+{3}="3",
(10,0)*+{2}="4",
\ar"3";"2", \ar"3";"4",
\end{xy}$ and $T_{2}$ is a basic tilting module of the wing of $P(4)$ and which has $P(n)$ as a direct summand. 
  \item[(III4)] $S^*=T_1\oplus T_2$, where $T_1$ can be viewed as a basic 2-term tilting complex over $\Lambda_{i}$ and $T_2$ is a basic tilting module of the wing of $P(i+1)$ (with $i\geq 4$) and which has $P(n)$ as a direct summand.
\end{itemize}
For $1\leq q\leq 4$, put
\begin{align*}
\ca_s^{\rm{III}q}(\Gamma_n)&=\{\End((S^{*}/P(n)))\oplus X)\mid \text{$S^{*}$ belongs to the family (IIIq)}\}/\cong.
\end{align*}

In the case (III1), we have the following two subcases:

\emph{Subcase (III1.1):} $T_1=P(2)\oplus T'$, where $T'$ is a tilting module of the wing of $P(3)$ which has $P(n)$ as a direct summand.  Moreover, $\Hom(T_1,P(1)[1])=0$.
\begin{itemize}
  \item [(a)] $T_1$ contains $P(n-1)$ as a direct summand. Thus, $\Hom(T_1,X)=0$. This implies that 
$$\End(S)=\End(P(1)[1])\oplus \End(T_1/P(n))\oplus \End(X),$$ where $\End(T_1/P(n))$ is a tilted algebra of type $A_{n-2}$. Denote by $\ca_s^{\rm{III1},1a}(\Gamma_n)$ the set of isoclasses of $\End(S)$ for such $S$'s. 

  \item [(b)] $T_1$ does not contain $P(n-1)$ as a direct summand. Thus, $\Hom(T_1,X)\neq0$. This implies that 
$$\End(S)=\End(P(1)[1])\oplus \End(T_1^*=T_1/P(n)\oplus X).$$ 
Moreover, there are triangle equivalences $\thick(T_1)=\thick(T_1^*)\simeq K^b(\proj A_{n-1})\simeq K^b(\proj B_{n-1})$, $T_1^*$ can be considered as a 2-term silting complex over $B_{n-1}$ which has $X$ as a direct summand. Thus, $\End(T_1/P(n)\oplus X)$ is a tilted algebra of type $\mathbb{A}_{n-1}$ and forms two groups: 
\begin{itemize}
  \item tilted algebra of type $A_{n-1}$;
  \item tilted algebra of type $\mathbb{A}_{n-1}$ but not of type $A_{n-1}$.
\end{itemize}
Denote by $\ca_s^{\rm{III1},1b}(\Gamma_n)$ the set of those tilted algebras of type $A_{n-1}\times A_{1}$ and by $\ca_s^{\rm{III1},1b'}(\Gamma_n)$ the set of those tilted algebras of type $\mathbb{A}_{n-1}\times A_{1}$. 
\end{itemize}

\emph{Subcase (III1.2):} $T_1=P(2)\oplus T'\oplus T''$, where $T'$ is a basic tilting module of the wing of $P(i)$ for some $4\leq i\leq n$ which has $P(n)$ as a direct summand and $T''$ is a basic tilting module of the wing of $I(i-2)$ within the AR-quiver of $\mod A_{n-1}$. Moreover, $\Hom(P(1)[1], T_1)\neq 0$.
By Theorem \ref{thm:classification-of-strictly-shod-algebras}, $\End(S^*)$ is a strictly shod algebra. 
\begin{itemize}
  \item [(a)] $T'$ contains $P(n-1)$ as a direct summand. Thus, $\Hom(T_1,X)=0$. This implies that 
$$\End(S)=\End(P(1)[1]\oplus T_1/P(n))\oplus \End(X).$$ 
By the proof of Theorem \ref{thm:classification-of-strictly-shod-algebras}, we obtain that $\End(P(1)[1]\oplus T_1/P(n))$ is a strictly shod algebra.
Denote by $\ca_s^{\rm{III1},2a}(\Gamma_n)$ the set of isoclasses of $\End(S)$ for such $S$'s. 

  \item [(b)] $T'$ does not contain $P(n-1)$ as a direct summand. Thus, $\Hom(T_1,X)\neq0$. This implies that 
$$\End(S)=\End(P(1)[1]\oplus (T_1/P(n))\oplus X).$$ 
\begin{itemize}
  \item If $i=n$, then $\End(S)$ is a tilted algebra of type $\mathbb{D}_n$;
  \item If $i<n$, then $\End(S)$ is a strictly shod algebra.
\end{itemize}
Denote by $\ca_s^{\rm{III1},2b}(\Gamma_n)$ the set of those tilted algebras of type $\mathbb{D}_n$ and by $\ca_s^{\rm{III1},2b'}(\Gamma_n)$ the set of those strictly shod algebras. 
\end{itemize}

Then, 
\begin{align}
\ca_s^{\rm{III1}}(\Gamma_n) &=\ca_s^{\rm{III1},1a}(\Gamma_n)\sqcup\ca_s^{\rm{III1},1b}\sqcup\ca_s^{\rm{III1},1b'}(\Gamma_n)\sqcup\ca_s^{\rm{III1},2a}(\Gamma_n)\sqcup\ca_s^{\rm{III1},2b}(\Gamma_n)\sqcup\ca_s^{\rm{III1},2b'}(\Gamma_n).
\end{align}

In the case (III2), we have the following two subcases:

\emph{Subcase (III2,1):} $T_1$ contains $P(n-1)$ as a direct summand. Thus, $\Hom(T_1,X)=0$. This implies that 
$$\End(S)=\End(P(1)[1])\oplus \End(P(2)[1])\oplus \End(T_1/P(n))\oplus \End(X),$$ where $\End(T_1/P(n))$ is a tilted algebra of type $A_{n-3}$. Denote by $\ca_s^{\rm{III2,1}}(\Gamma_n)$ the set of isoclasses of $\End(S)$ for such $S$'s. 
So $\ca_s^{\rm{III2,1}}(\Gamma_n)=\ca_t^{1}(A_{n-3})\times_s \ca_t(A_{1})\times_s \ca_t(A_{1})\times_s \ca_t(A_{1})$;

\emph{Subcase (III2,2):} $T_1$ does not contain $P(n-1)$ as a direct summand. Thus, $\Hom(T_1,X)\neq0$. This implies that 
$$\End(S)=\End(P(1)[1])\oplus \End(P(2)[1])\oplus\End(T_1/P(n)\oplus X).$$ Thus, $\End(T_1/P(n)\oplus X)$ is a tilted algebra of type $\mathbb{A}_{n-2}$ and forms two groups: 
\begin{itemize}
  \item [--] tilted algebra of type $A_{n-2}$;
  \item [--] tilted algebra of type $\mathbb{A}_{n-2}$ but not of type $A_{n-2}$.
\end{itemize}
Denote by $\ca_s^{\rm{III2},2a}(\Gamma_n)$ the set of those tilted algebras of type $A_{n-2}\times A_{1}\times A_{1}$ and by $\ca_s^{\rm{III2},2b}(\Gamma_n)$ the set of those tilted algebras of type $\mathbb{A}_{n-2}\times A_{1}\times A_{1}$. 
Then, 
\begin{align}
 \ca_s^{\rm{III2}}(\Gamma_n)&=\ca_s^{\rm{III2,1}}(\Gamma_n)\sqcup \ca_s^{\rm{III2},2a}(\Gamma_n) \sqcup \ca_s^{\rm{III2},2b}(\Gamma_n).
\end{align}

In the case (III3), we have the following two subcases:

\emph{Subcase (III3,1):} $T_2$ contains $P(n-1)$ as a direct summand. Thus, $\Hom(T_2,X)=0$. This implies that 
$$\End(S)=\End(T_1)\oplus \End(T_2/P(n))\oplus \End(X),$$ where $\End(T_1)$ is a tilted algebra of type $A_{3}$ and $\End(T_2/P(n))$ is a tilted algebra of type $A_{n-4}$. Denote by $\ca_s^{\rm{III3,1}}(\Gamma_n)$ the set of isoclasses of $\End(S)$ for such $S$'s. 
So $\ca_s^{\rm{III3,1}}(\Gamma_n)=\ca_t^{1}(A_{n-4})\times_s \ca_t(A_{1})\times_s \ca_t(A_{3})$;

\emph{Subcase (III3,2):} $T_2$ does not contain $P(n-1)$ as a direct summand. Thus, $\Hom(T_2,X)\neq0$. This implies that 
$$\End(S)=\End(T_1)\oplus \End(T_2/P(n)\oplus X).$$ 
Thus, $\End(T_2/P(n)\oplus X)$ is tilted of type $\mathbb{A}_{n-3}$ and forms the following two classes:
\begin{itemize}
  \item [(a)] tilted algebra of type $A_{n-3}$;
  \item [(b)] tilted algebra of type $\mathbb{A}_{n-3}$ but not of type $A_{n-3}$.
\end{itemize}
Denote by $\ca_s^{\rm{III3},2a}(\Gamma_n)$ the set of those tilted algebras of type $A_{n-3}\times A_{3}$ and by $\ca_s^{\rm{III3},2b}(\Gamma_n)$ the set of those tilted algebras of type $\mathbb{A}_{n-3}\times A_{3}$. 
Then, 
\begin{align}
 \ca_s^{\rm{III3}}(\Gamma_n)&=\ca_s^{\rm{III3,1}}(\Gamma_n)\sqcup \ca_s^{\rm{III3},2a}(\Gamma_n) \sqcup \ca_s^{\rm{III3},2b}(\Gamma_n).
\end{align}

In the case (III4), $T_2$ is a tilting module of the wing of $P(i+1)$ with $4\leq i\leq n-1$ and has $P(n)$ as a direct summand. Consider $T_2^*=T_2/P(n)\oplus X$. We have the following two subcases:

\begin{itemize}
  \item [(1)] If $i=n-1$, in this case, $\tau^{m}(S)$ is a tilting module over $\Lambda_{n}$ for some positive integer $m\geq 2$. Denote by $\ca_s^{\rm{III4,1}}(\Gamma_n)$ the set of isoclasses of $\End(S)$ for such $S$'s. 
  \item [(2)] If $i\neq n-1$, then $\End(S)=\End(T_1)\oplus \End(T_2^*)$. Moreover, there are triangle equivalences $\thick(T_2)=\thick(T_2^*)\simeq K^b(\proj A_{n-i})\simeq K^b(\proj B_{n-i})$, $T_2^*$ can be considered as a 2-term silting complex over $B_{n-i}$ which has $X$ as a direct summand. Thus, $\End(T_2^*)$ falls into the following three classes: 
\begin{itemize}
  \item tilted algebra of type $A_{n-i-1}\times A_{1}$;
  \item tilted algebra of type $A_{n-i}$;
  \item tilted algebra of type $\mathbb{A}_{n-i}$ but not of type $A_{n-i}$.
\end{itemize}
Denote by $\ca_s^{\rm{III4,2a}}(\Gamma_n)$ the set of those tilted algebras of type $A_{n-i-1}\times A_{1}\times \Lambda_{i}$. So $\ca_s^{\rm{III4,2a}}(\Gamma_n)=\bigcup\limits_{i=4}^{n-1}\ca_t^{1}(A_{n-i-1})\times_s \ca_t(A_{1})\times_s \ca_t(\Lambda_{i})$;
Denote by $\ca_s^{\rm{III4},2b}(\Gamma_n)$ the set of those tilted algebras of type $A_{n-i}\times \Lambda_{i}$ and by $\ca_s^{\rm{III4},2c}(\Gamma_n)$ the set of those tilted algebras of type $\mathbb{A}_{n-i}\times \Lambda_{i}$ for any $4\leq i\leq n-1$. 
\end{itemize}
Then, 
\begin{align}
 \ca_s^{\rm{III4}}(\Gamma_n)&=\ca_s^{\rm{III4,1}}(\Gamma_n)\sqcup \ca_s^{\rm{III4},2a}(\Gamma_n) \sqcup \ca_s^{\rm{III4},2b}(\Gamma_n) \sqcup \ca_s^{\rm{III4},2c}(\Gamma_n).
\end{align}

\subsubsection{The proof of Theorem \ref{thm:classification-of-silted-algebras-of-mutation-type-D}}
\label{sec:the-proof-of-theorem-5.1}
We combine results in Subsections~\ref{sec:case I}, \ref{sec:case II} and \ref{sec:case III} to prove Theorem~\ref{thm:classification-of-silted-algebras-of-mutation-type-D}.
Put 
\begin{align*}
\cc_1&=\ca_s^{\rm{I}1}(\Gamma_n)\cup \ca_s^{\rm{I}2}(\Gamma_n)\cup \ca_s^{\rm{II2},2b}(\Gamma_n)\cup \ca_s^{\rm{III4},1}(\Gamma_n),\\
\cc_2&=\ca_s^{\rm{I3},2}(\Gamma_n)\cup \ca_s^{\rm{II1},2}(\Gamma_n)\cup \ca_s^{\rm{II2},1b}(\Gamma_n)\cup \ca_s^{\rm{II2},2b'}(\Gamma_n)\cup \ca_s^{\rm{II3},2}(\Gamma_n)\cup \ca_s^{\rm{II4},2}(\Gamma_n)\cup \ca_s^{\rm{III1},2b}(\Gamma_n),\\
\cc_3&=\ca_s^{\rm{I6}}(\Gamma_n)\cup\ca_s^{\rm{II1},1}(\Gamma_n)\cup\ca_s^{\rm{II2},1a}(\Gamma_{n})\cup\ca_s^{\rm{II2},2a}(\Gamma_{n})\cup\ca_s^{\rm{II3},1}(\Gamma_{n})
\cup\ca_s^{\rm{II4},1}(\Gamma_{n})\cup\ca_s^{\rm{III4},2b}(\Gamma_{n}),\\
\cc_4&=\ca_s^{\rm{III4},2c}(\Gamma_{n}),\\
\cc_5&=\ca_s^{\rm{I3},1}(\Gamma_{n})\cup\ca_s^{\rm{III1},1b}(\Gamma_{n}),\\
\cc_6&=\ca_s^{\rm{III1},1b'}(\Gamma_{n}),\\
\cc_7&=\ca_s^{\rm{I5}}(\Gamma_{n})\cup\ca_s^{\rm{III3},2a}(\Gamma_{n}),\\
\cc_8&=\ca_s^{\rm{III3},2b}(\Gamma_{n}),\\
\cc_{9}&=\ca_s^{\rm{III1},1a}(\Gamma_{n})\cup\ca_s^{\rm{III2},2a}(\Gamma_{n})\cup\ca_s^{\rm{I4}}(\Gamma_{n}),\\
\cc_{10}&=\ca_s^{\rm{III2},2b}(\Gamma_{n}),\\
\cc_{11}&=\ca_s^{\rm{III2},1}(\Gamma_{n}),\\
\cc_{12}&=\ca_s^{\rm{III3},1}(\Gamma_{n}),\\
\cc_{13}&=\ca_s^{\rm{III4},2a}(\Gamma_{n}),\\
\cc_{14}&=\ca_s^{\rm{I3},3}(\Gamma_{n})\cup\ca_s^{\rm{III1},2b'}(\Gamma_{n})\cup\ca_s^{\rm{III1},2a}(\Gamma_{n}).\\
\end{align*}
According to the equalities $(5.1)-(5.11)$, we have $\ca_s(\Gamma_n)=\cc_1\sqcup\cc_2\sqcup\cc_3\sqcup\cc_4\sqcup\cc_5\sqcup\cc_6\sqcup\cc_7\sqcup\cc_8\sqcup\cc_9\sqcup\cc_{10}\sqcup\cc_{11}\sqcup\cc_{12}\sqcup\cc_{13}\sqcup\cc_{14}$.
\begin{itemize}
  \item [(1)] $\cc_1=\ca_t(\Lambda_n)$, because $\ca_s^{\rm{I}1}(\Gamma_n)$, $\ca_s^{\rm{II2},2b}(\Gamma_n)$ and $\ca_s^{\rm{III4},1}(\Gamma_n)$ are subsets of $ \ca_s^{\rm{I}2}(\Gamma_n)=\ca_t(\Lambda_n)$;
  \item [(2)] 
  \begin{align*}
   \cc_2&=\ca_s^{\rm{I3},2}(\Gamma_n)\cup \ca_s^{\rm{II1},2}(\Gamma_n)\cup \ca_s^{\rm{II2},1b}(\Gamma_n)\cup \ca_s^{\rm{II2},2b'}(\Gamma_n)\cup \ca_s^{\rm{II3},2}(\Gamma_n)\cup \ca_s^{\rm{II4},2}(\Gamma_n)\cup \ca_s^{\rm{III1},2b}(\Gamma_n),\\
   &=\ca_{t}(\mathbb{D}_n)\cup \ca_{t}^{\mu}(\mathbb{D}_n),\\
  \end{align*}
  \item [(3)] By the proof in \cite[Section 5.3.4]{XieYangZhang25}, $\ca_t^2(A_n)\cup \ca_s^{\mu}(A_{n})=\ca_t^4(A_n)$.
  \begin{align*}
  \cc_3&=\ca_s^{\rm{I6}}(\Gamma_n)\cup\ca_s^{\rm{II1},1}(\Gamma_n)\cup\ca_s^{\rm{II2},1a}(\Gamma_{n})\cup\ca_s^{\rm{II2},2a}(\Gamma_{n})\cup\ca_s^{\rm{II3},1}(\Gamma_{n})
\cup\ca_s^{\rm{II4},1}(\Gamma_{n})\cup\ca_s^{\rm{III4},2b}(\Gamma_{n}),\\ 
&=\bigcup\limits_{m=2}^{n-4}(\ca_t^2(A_m)\times_s \ca_t(\Lambda_{n-m}))\cup (\ca_t^1(\Lambda_{n-1})\times_s \ca_t(A_1))\cup\bigcup\limits_{m=2}^{n-4}(\ca_s^{\mu}(A_{m})\times_s \ca_t(\Lambda_{n-m})),\\
&=\bigcup\limits_{m=2}^{n-4}((\ca_t^2(A_m)\cup \ca_s^{\mu}(A_{m})) \times_s \ca_t(\Lambda_{n-m}))\cup (\ca_t^1(\Lambda_{n-1})\times_s \ca_t(A_1)),\\
&=\bigcup\limits_{m=2}^{n-4}(\ca_t^4(A_m) \times_s \ca_t(\Lambda_{n-m}))\cup (\ca_t^1(\Lambda_{n-1})\times_s \ca_t(A_1)),\\
  \end{align*}

  \item [(4)] $\cc_4=\ca_s^{\rm{III4},2c}(\Gamma_{n})= \bigcup\limits_{m=2}^{n-4}(\ca_s^{\mu}(\mathbb{A}_m) \times_s \ca_t(\Lambda_{n-m}))$,
  \item [(5)] 
   \begin{align*} 
  \cc_5&=\ca_s^{\rm{I3},1}(\Gamma_{n})\cup\ca_s^{\rm{III1},1b}(\Gamma_{n}),\\
  &=(\ca_t^3(A_{n-1})\cap \ca_t^2(A_{n-1}))\times_s\ca_t(A_{1})\cup (\ca_s^{\mu}(A_{n-1})\times_s \ca_t(A_{1})),\\
  &=((\ca_t^3(A_{n-1})\cup \ca_s^{\mu}(A_{n-1}))\cap (\ca_t^2(A_{n-1})\cup \ca_s^{\mu}(A_{n-1})))\times_s \ca_t(A_{1}),\\
  &=((\ca_t^3(A_{n-1})\cup \ca_s^{\mu}(A_{n-1}))\cap \ca_t^4(A_{n-1}))\times_s \ca_t(A_{1}),\\
  \end{align*}
  \item [(6)] $\cc_6=\ca_s^{\rm{III1},1b'}(\Gamma_{n})=\ca_s^{\mu,2}(\mathbb{A}_{n-1})\times_s \ca_t(A_{1})$,
  \item [(7)] 
  \begin{align*} 
  \cc_7&=\ca_s^{\rm{I5}}(\Gamma_{n})\cup\ca_s^{\rm{III3},2a}(\Gamma_{n}),\\
  &=(\ca_t(B_{3})\times_s \ca_t^2(A_{n-3}))\cup (\ca_s^{\mu}(A_{n-3})\times_s \ca_t(B_{3})),\\
  &=(\ca_t^2(A_{n-3})\cup \ca_s^{\mu}(A_{n-3})))\times_s \ca_t(B_{3}),\\
  &=\ca_t^4(A_{n-3})\times_s \ca_t(B_{3}),\\
  \end{align*}
  \item [(8)] $\cc_8=\ca_s^{\rm{III3},2b}(\Gamma_{n})=\ca_s^{\mu}(\mathbb{A}_{n-3})\times_s \ca_t(B_{3})$,
  \item [(9)] 
  \begin{align*} 
  \cc_{9}&=\ca_s^{\rm{III1},1a}(\Gamma_{n})\cup\ca_s^{\rm{III2},2a}(\Gamma_{n})\cup \ca_s^{\rm{I4}}(\Gamma_{n}),\\
  &=((\ca_t^3(A_{n-2})\cap \ca_t^1(A_{n-2}))\times_s \ca_t(A_{1})\times_s \ca_t(A_{1}))\cup ((\ca_t^2(A_{n-2})\cup \ca_s^{\mu}(A_{n-2}))\times_s \ca_t(A_{1})\times_s \ca_t(A_{1})),\\
  &=((\ca_t^3(A_{n-2})\cap \ca_t^1(A_{n-2}))\cup\ca_t^4(A_{n-2}))\times_s \ca_t(A_{1})\times_s \ca_t(A_{1}),\\
  \end{align*}
  \item [(10)] $\cc_{10}=\ca_s^{\rm{III2},2b}(\Gamma_{n})=\ca_s^{\mu}(\mathbb{A}_{n-2})\times_s \ca_t(A_{1})\times_s \ca_t(A_{1})$,
  \item [(11)] $\cc_{11}=\ca_s^{\rm{III2},1}(\Gamma_{n})=\ca_t^{1}(A_{n-3})\times_s \ca_t(A_{1})\times_s \ca_t(A_{1})\times_s \ca_t(A_{1})$,
  \item [(12)] $\cc_{12}=\ca_s^{\rm{III3},1}(\Gamma_{n})=\ca_t^{1}(A_{n-4})\times_s \ca_t(A_{1})\times_s \ca_t(B_{3})$,
  \item [(13)] $\cc_{13}=\ca_s^{\rm{III4},2a}(\Gamma_{n})=\bigcup\limits_{m=4}^{n-1}\ca_t^{1}(A_{n-m-1})\times_s \ca_t(A_{1})\times_s \ca_t(\Lambda_{m})$,
  \item [(14)] $\cc_{14}=\ca_s^{\rm{I3},3}(\Gamma_{n})\cup\ca_s^{\rm{III1},2b'}(\Gamma_{n})\cup\ca_s^{\rm{III1},2a}(\Gamma_{n})=\ca_{ss}(\Gamma_{n})$.
\end{itemize}
This completes the proof.

\subsubsection{The number of silted algebras of type $\Gamma_{n}$} In this subsection, we count the number of silted algebras of type $\Gamma_{n}$. Let $a_{s}^{\mu}(\mathbb{A}_{n})$ be the cardinality of $\ca_s^{\mu}(\mathbb{A}_{n})$. Then,  by \cite[Theorem 5.2(2)]{XieYangZhang25}, we have 
\begin{align}
a_{s}^{\mu}(\mathbb{A}_{n})=\begin{cases} 0 & \text{ if } n\leq 4,\\
2 & \text{ if } n=5,\\
t(A_{n-1})-t(A_{n-2})-2^{n-2}-2^{n-4}+\frac{n-2}{2} & \text{ if $n\geq 6$ is even},\\
t(A_{n-1})-t(A_{n-2})-2^{n-2}-2^{n-4}+2^{\frac{n-3}{2}-1}+\frac{n-3}{2} & \text{ if $n\geq 7$ is odd}.
\end{cases}
\end{align}
Let $c_{i}$ be the cardinalities of $\cc_i$ for any $1\leq i\leq 14$. By Theorem \ref{thm:classification-of-silted-algebras-of-mutation-type-D}, we have 

\begin{proposition}\label{prop:number-of-silted-algebras}
$a_s(\Gamma_n)=\sum\limits_{i=1}^{14}c_{i}$, where
\begin{itemize}
  \item [(1)] $c_{1}=a_t(\Lambda_n)$, which is given in Proposition \ref{prop:tilted-algebras-of-type-D};
  \item [(2)] 
 $c_{2}=\begin{cases} 1 & \text{ if } n= 4,\\
\sum\limits_{i=4}^{n-3}(t(A_{n{-}i{-}1})-t(A_{n{-}i{-}2}))\times (t^{1}(\Lambda_{i})+t(A_{i{-}2}))\\ + 2t(A_{n{-}2})+2t(A_{n{-}3})+2t(A_{n{-}4})-3t(A_{n{-}5})& \text{ if } n\geq 5,\\
\end{cases}$
  \item [(3)] $c_3=\sum\limits_{m=2}^{n-4}c_{3,m}+a_{t}^{1}(\Lambda_{n-1})$, where
  \begin{center}
    $c_{3,m}=\begin{cases} a_t(\Lambda_{n{-}2}) & \text{ if } m= 2,\\
(a_{t}(A_m)-t(A_{m{-}1})+2)\times a_t(\Lambda_{n{-}m})& \text{ if } m\geq 3,\\
\end{cases}$
  \end{center}
  \item [(4)] $c_4=\sum\limits_{m=2}^{n-4}a_{s}^{\mu}(\mathbb{A}_{m}) \times a_t(\Lambda_{n{-}m})$, where $a_{s}^{\mu}(\mathbb{A}_{m})$ can be obtained by (5.12);
  \item [(5)]
  $c_{5}=\begin{cases} 1 & \text{ if } n=4\\
t(A_{n-2})-t(A_{n-4})-n+4 & \text{ if } n\geq5,\\
\end{cases}$
  \item [(6)] $c_6=a_{s}^{\mu}(\mathbb{A}_{n-2})$, where $a_{s}^{\mu}(\mathbb{A}_{n-2})$ can be obtained by (5.12);
  \item [(7)] 
  $c_{7}=\begin{cases} 0 & \text{ if } n=4\\
  3 & \text{ if } n=5,\\
9& \text{ if } n=6,\\
3a_t^{4}(A_{n-3}) & \text{ if } n\geq 7,\\
\end{cases}$
  \item [(8)] $c_8=3a_{s}^{\mu}(\mathbb{A}_{n-3})$, where $a_{t}(B_3)=3$ and $a_{s}^{\mu}(\mathbb{A}_{n-3})$ can be obtained by (5.12);
    
  \item [(9)] 
  $c_{9}=\begin{cases} 1 & \text{ if } n=4,\\
4& \text{ if } n=5,\\
a_t(A_{n-2})-t(A_{n-3})+t(A_{n-4}) & \text{ if } n\geq 6,\\
\end{cases}$
  \item [(10)] $c_{10}=a_{s}^{\mu}(\mathbb{A}_{n-2})$, where $a_{s}^{\mu}(\mathbb{A}_{n-2})$ can be obtained by (5.12);
    
  \item [(11)] $c_{11}=t(A_{n-4})$;
  \item [(12)] $c_{12}=3t(A_{n-5})$;
  \item [(13)] $c_{13}=\sum\limits_{m=4}^{n-1}t(A_{n-m-2}) \times a_t(\Lambda_{m})$
  \item [(14)] $c_{14}=a_{nht}(A_{n-1})-2a_{nht}(A_{n-2})-t(A_{n-3})$.
\end{itemize}
\end{proposition}

\begin{proof} We only prove (2), (5), (7) and (9), and the others can be easily obtained.

(2) According to the Subcase (I3.2), $\mid \ca_s^{\rm{I3},2}(\Gamma_n)\mid=t(A_{n-2})$. Moreover, by Corollary \ref{cor:number-of-tilting-contains-P(n)}, we have 
\begin{itemize}
  \item [--] $\mid\ca_s^{\rm{II1},2}(\Gamma_n)\mid=t(A_{n-3})-t(A_{n-4})$,
  \item [--]   
 $\mid\ca_s^{\rm{II2},1b}(\Gamma_n) \mid=\begin{cases} 0 & \text{ if } n=4,5,\\
t(A_{n-4})-t(A_{n-5}) & \text{ if } n\geq6,\\
\end{cases}$
  \item [--] 
 $\mid\ca_s^{\rm{II2},2b'}(\Gamma_n)\mid=\begin{cases} 0 & \text{ if } n=4,5,6,\\
\sum\limits_{i=4}^{n-3}(t(A_{n-i-1})-t(A_{n-i-2}))\times t^{1}(\Lambda_i) & \text{ if } n\geq7,\\
\end{cases}$

\end{itemize}
In the Subcase (II3.2), by Corollaries \ref{cor:number-of-tilting-contains-I(1)} and \ref{cor:number-of-tilting-contains-P(n)}, we have
\begin{align*}
 \mid \ca_s^{\rm{II3},2}(\Gamma_n)\mid=\begin{cases} 3 & \text{ if } n=5,\\
\sum\limits_{i=4}^{n-2}(t(A_{n-i})-t(A_{n-i-1}))\times (t(A_{i-2})-t(A_{i-3}))+t(A_{n-2})-t(A_{n-3}) & \text{ if } n\geq6,\\
\end{cases}
\end{align*}
where $t(A_{i-2})-t(A_{i-3})$ is the number of tilting modules over $A_{i-2}$ which do not contain $I(1)$ as a direct summand. At last, it is easy to see that 
 \begin{align*}
 \mid\ca_s^{\rm{II4},2}(\Gamma_n)\mid=\begin{cases} 2 & \text{ if } n=5,\\
\sum\limits_{i=4}^{n-2}(t(A_{n-i})-t(A_{n-i-1}))\times t(A_{i-3})+t(A_{n-3}) & \text{ if } n\geq6,\\
\end{cases}
\end{align*}
and 
$\mid\ca_s^{\rm{III1},2b}(\Gamma_n)\mid=t(A_{n-3})$. Then, as a consequence, we obtain $c_2$.

(5) According to \cite[Example 3.11]{Xing21}, $c_{5}=1$ for $n=4$. By Corollary \ref{cor:number-of-tilting-contains-P(n)}, the number of tilting modules over $A_{n-1}$ that contain $P(3)$ as a direct summand but do not contain $P(n)$ as a direct summand is $t(A_{n-2})-t(A_{n-3})$. Among these tilting modules, the endomorphism algebras with only the following two elements are isomorphic.
$$
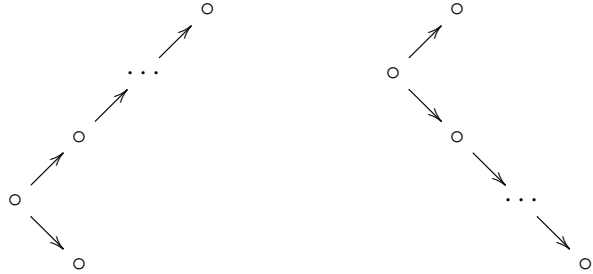
\begin{xy}
0;<3pt,0pt>:<0pt,3pt>::
(16,16)*+{\circ}="1",
(8,8)*+{\cdots}="2",
(0,0)*+{\circ}="3",
(-8,-8)*+{\circ}="4",
(0,-16)*+{\circ}="6",
\ar"4";"3", \ar"3";"2", \ar"2";"1", \ar"4";"6", 
\end{xy}
\hspace{2cm}
\begin{xy}
0;<3pt,0pt>:<0pt,3pt>::
(8,16)*+{\circ}="2",
(0,8)*+{\circ}="3",
(8,0)*+{\circ}="5",
(16,-8)*+{\cdots}="6",
(24,-16)*+{\circ}="7",
\ar"3";"2", \ar"3";"5", \ar"5";"6", \ar"6";"7",
\end{xy}$$
Thus, $\mid \ca_{s}^{{\rm I3},1(\Gamma_n)}\mid=t(A_{n-2})-t(A_n-3)-1$. On the other hand, the number of tilting modules over $A_{n-1}$ that contain $P(3)$ and $P(n)$ as direct summands but do not contain $P(n-1)$ as a direct summand is $t(A_{n-3})-t(A_{n-4})$. In this case, the following elements have isomorphic endomorphism algebras.

$$
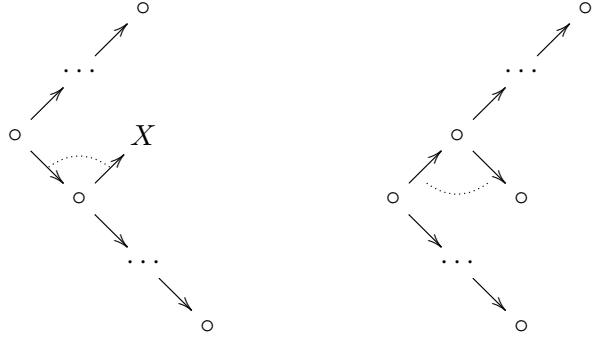
\begin{xy}
0;<3pt,0pt>:<0pt,3pt>::
(8,16)*+{\circ}="1",
(0,8)*+{\cdots}="2",
(-8,0)*+{\circ}="3",
(8,0)*+{X}="4",
(0,-8)*+{\circ}="5",
(8,-16)*+{\cdots}="6",
(16,-24)*+{\circ}="7",
(-5,-5)*+{ }="9",
(5,-5)*+{ }="10",
\ar@/^0.6pc/@{.}"9";"10",
\ar"5";"4", \ar"3";"5", \ar"3";"2", \ar"2";"1", \ar"5";"6", \ar"6";"7",
\end{xy}
\hspace{2cm}
\begin{xy}
0;<3pt,0pt>:<0pt,3pt>::
(16,16)*+{\circ}="1",
(8,8)*+{\cdots}="2",
(0,0)*+{\circ}="3",
(-8,-8)*+{\circ}="4",
(8,-8)*+{\circ}="5",
(0,-16)*+{\cdots}="6",
(8,-24)*+{\circ}="7",
(-5,-5)*+{ }="9",
(5,-5)*+{ }="10",
\ar@/_0.6pc/@{.}"9";"10",
\ar"4";"3", \ar"3";"5", \ar"3";"2", \ar"2";"1", \ar"4";"6", \ar"6";"7",
\end{xy}$$
The tilting modules of the right form  do not contain $P(n)$ as a direct summand. 
Hence, we obtain that there are $n-5$ isomorphic endomorphism algebras. So, $c_5=t(A_{n-2})-t(A_n-3)-1+(t(A_{n-3})-t(A_{n-4})-n+5)=t(A_{n-2})-t(A_{n-4})-n+4$.

(7) Since $\ca_t(B_3)$ is a subset of $\ca_t^{4}(A_3)$, for $n=6$, by Lemma \ref{lem:symmetric-product}, we have 
$$c_{7}=a_t^{4}(A_3)a_t(B_3)-\frac{a_t(B_3)\times(a_t(B_3)-1)}{2}=9.$$ Thus, 
$$c_{7}=\begin{cases} 0 & \text{ if } n=4\\
  3 & \text{ if } n=5,\\
a_t^{4}(A_3)a_t(B_3)-\frac{a_t(B_3)\times(a_t(B_3)-1)}{2}& \text{ if } n=6,\\
a_t^{4}(A_{n-3})a_t(B_3) & \text{ if } n\geq 7.\\
\end{cases}$$

(9) Recall that $\ca_s^{\rm{III2},2a}(\Gamma_{n})\cup \ca_s^{\rm{I4}}(\Gamma_{n})=\ca_s^{\rm{I4}}(\Gamma_{n})\cup \End(P(3)\oplus \cdots \oplus P(n-2)\oplus P(n)\oplus \tau^{-2}P(n))$. We have the following three cases:
\begin{itemize}
  \item [--] For $n=4$, $\mid \ca_s^{\rm{III2},2a}(\Gamma_{n})\cup \ca_s^{\rm{I4}}(\Gamma_{n})\mid=0$ and $\mid \ca_s^{\rm{III1},1a}(\Gamma_{n})\mid=1$;

  \item [--] For $n=5$, since the endomorphism algebras of all projective modules and all injective modules over $A_{n}$ are isomorphic, 
  $$\mid \ca_s^{\rm{III2},2a}(\Gamma_{n})\cup \ca_s^{\rm{I4}}(\Gamma_{n})\mid=a_t(A_{3})-t(A_{2})+1+1=4.$$ 
  Moreover, $$\ca_s^{\rm{III1},1a}(\Gamma_{n})=\End(P(2)\oplus P(3)\oplus P(4))\oplus \End(X)\oplus \End(P(1)[1])).$$ This shows that $\ca_s^{\rm{III1},1a}(\Gamma_{n})$ is a subset of $\mid \ca_s^{\rm{III2},2a}(\Gamma_{n})\cup \ca_s^{\rm{I4}}(\Gamma_{n})$.
  \item [--] For $n\geq 6$, we obtain that $\mid \ca_s^{\rm{III2},2a}(\Gamma_{n})\cup \ca_s^{\rm{I4}}(\Gamma_{n})\mid=a_t(A_{n-2})-t(A_{n-3})+2$ and  $\mid\ca_s^{\rm{III1},1a}(\Gamma_{n})\mid=t(A_{n-4})$. Moreover, 
      $$\End(P(3) \oplus \cdots \oplus P(n))\oplus \End(P(1)[1])\oplus \End(P(2)[1])$$ and 
      $$\End(P(3)\oplus \cdots \oplus P(n-2)\oplus P(n)\oplus \tau^{-2}P(n)) \oplus \End(P(1)[1])\oplus \End(P(2)[1])$$ belong to $\ca_s^{\rm{III1},1a}(\Gamma_{n})$. So $c_{9}=a_t(A_{n-2})-t(A_{n-3})+t(A_{n-4})$.
\end{itemize}

\end{proof}

\begin{example}
\begin{itemize}
  \item [(1)] For $n=4$, $c_1=7$, $c_2=1$,$c_5=1$,$c_9=1$, $c_{11}=1$ and the other $c_{i}=0$. Thus, $a_s(\Gamma_4)=11$; See \cite[Example 3.11]{Xing21}
  \item [(2)] For $n=5$, $c_1=35$, $c_2=13$,$c_3=4$,$c_5=3$, $c_7=3$, $c_9=4$, $c_{11}=1$, $c_{12}=3$,$c_{14}=2$ and the other $c_{i}=0$. It should be noted that, $\cc_{12}$ is a subset of $\cc_{9}$. Thus, $a_s(\Gamma_5)=65$;
  \item [(3)] For $n=6$, $c_1=126$, $c_2=39$,$c_3=22$,$c_5=10$, $c_7=9$, $c_9=7$, $c_{11}=2$, $c_{12}=3$,$c_{13}=7$,$c_{14}=9$ and the other $c_{i}=0$.Thus, $a_s(\Gamma_6)=234$.
\end{itemize}
\end{example}

\begin{remark}
By Proposition \ref{prop:number-of-silted-algebras}, we conclude that the number of strictly shod algebras of type $\Gamma_n$ is given by $a_{nht}(A_{n-1})-2a_{nht}(A_{n-2})-t(A_{n-3})$. Moreover, we have  that strictly shod algebras of type $\Gamma_n$ fall into two classes:
\begin{itemize}
  \item [(a)] type $\Lambda_n$;  
  \item [(b)] type $\Lambda_{n-1}\times \Lambda_{1}$.
\end{itemize}
The number of such algebras in class $(a)$ is $\sum\limits_{m=2}^{n-3}t(A_m)\times t(A_{n-m-2})$, while the number in class $(b)$ is $\sum\limits_{m=1}^{n-5}t(A_m)\times t(A_{n-m-4})$.
\end{remark}

\begin{remark}
Let $n=7$. Then $c_{10}=2$. The two elements are 
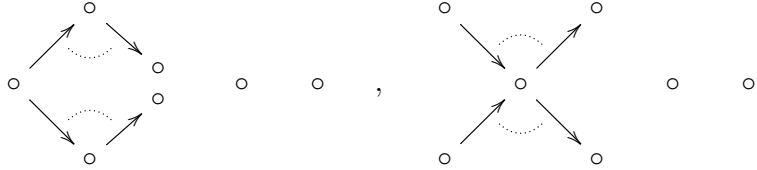
\begin{align*}
 \begin{xy}
(50,0) *+{\circ}="52",
(40,0) *+{\circ}="53",
(10,0) *+{\circ}="31",
(29,2) *+{\circ}="33",
(20,10) *+{\circ}="32",
(29,-2) *+{\circ}="36",
(16,6)*+{ }="34",
(24,6)*+{ }="35",
(20,-10)*+{\circ}="41",
(16,-6)*+{ }="44",
(24,-6)*+{ }="45",
\ar"31";"32", \ar"32";"33", \ar@/^0.6pc/@{.}"35";"34",
\ar"31";"41", \ar@/_0.6pc/@{.}"45";"44",
\ar"41";"36",
\end{xy}\hspace{15pt},\hspace{15pt}
\begin{xy}
(40,10) *+{\circ}="31",
(30,0) *+{\circ}="33",
(20,10) *+{\circ}="32",
(50,0) *+{\circ}="52",
(60,0) *+{\circ}="53",
(26,4)*+{ }="34",
(34,4)*+{ }="35",
(20,-10)*+{\circ}="41",
(26,-4)*+{ }="44",
(34,-4)*+{ }="45",
(40,-10)*+{\circ}="61",
\ar"33";"31", \ar"32";"33", \ar@/_0.6pc/@{.}"35";"34",
\ar@/^0.6pc/@{.}"45";"44",
\ar"33";"61", \ar"41";"33",
\end{xy}
\end{align*}
Note that these two non-hereditary connected subquivers are tilted algebras of $\mathbb{A}_{n-2}$ but not tilted algebras of type $A_{n-2}$.
\end{remark}

By Theorem \ref{thm: strictly-shod-algebras-are-string-algebras}, we have the following corollary.

\begin{corollary}
The strictly shod algebras of type $\Gamma_{n}$ are string algebras.
\end{corollary}

Since the number of strictly shod algebras decreases from $\Lambda_n$ to $\Gamma_n$, we therefore propose the following question.

\begin{question}
Are all strictly shod algebras of Dynkin quivers of type $\mathbb{D}_n$ with arbitrary orientations string algebras?
\end{question}

\section{The realization functor is not an equivalence}
\label{s:the-realization-functor-is-not-an-equivalence}
In this section, based on the classification of the silted algebras of type $\Lambda_{n}$ and $\Gamma_{n}$, we examine the realization functor induced by the t-structure. We begin by recalling the definition of a t-structure.

Let $\mathcal{T}$ be a triangulated category. A $\emph t$-$\emph structure$ on $\mathcal{T}$ is a pair $(\mathcal{T}^{\leq 0},\mathcal{T}^{\geq 0})$ of strict (that is, closed under isomorphisms) and full subcategories of $\mathcal{T}$, satisfying the following conditions for $\mathcal{T}^{\geq i}:=\mathcal{T}^{\geq 0}[-i]$, $\mathcal{T}^{\leq i}:=\mathcal{T}^{\leq 0}[-i]$
\begin{itemize}
\item[(1)] $\mathcal{T}^{\leq 0}\subseteq \mathcal{T}^{\leq 1}$ and $\mathcal{T}^{\geq 1}\subseteq \mathcal{T}^{\geq 0}$;
\item[(2)] $\Hom_{\mathcal{T}}(X,Y)=0$ for $X\in \mathcal{T}^{\leq 0}$ and $Y\in \mathcal{T}^{\geq 1}$;
\item[(3)] For any object $Z$ of $\ct$, there is a triangle $X\to Z\to Y\to X[1]$ with $X\in\mathcal{T}^{\leq 0}$ and $Y\in\mathcal{T}^{\geq 1}$.
\end{itemize}
A $t$-structure $(\mathcal{T}^{\leq 0},\mathcal{T}^{\geq 0})$ is said to be \emph{bounded} if
\[
\bigcup_{i\in \mathbb{Z}}\mathcal{T}^{\leq i}=\mathcal{T}=\bigcup_{i\in \mathbb{Z}}\mathcal{T}^{\geq i}.
\]
The category $\ca=\mathcal{T}^{\leq 0}\cap \mathcal{T}^{\geq 0}$ is referred to as the \emph{heart} of the $t$-structure $(\mathcal{T}^{\leq 0},\mathcal{T}^{\geq 0})$ and is an abelian category due to \cite[Th\'eor\`eme 1.3.6]{BeilinsonBernsteinDeligne82}.

Let $A$ and $B$ be finite-dimensional algebras such that the category $\mod B$ of finite-dimensional
$B$-modules forms the heart of a bounded t-structure on the bounded derived category $\mathcal{D}^{b}(\mod A)$ of $\mod A$. Then the embedding functor $\mod B\hookrightarrow \mathcal{D}^{b}(\mod A)$ can be extended to a realization functor $\mathcal{D}^{b}(\mod B)\to\mathcal{D}^{b}(\mod A)$. Recently, Martin Kalck observed an interesting phenomenon: there exists examples where $A$ and $B=\End(M)$ are derived equivalent (with $M$ a silting object), but the embedding $\mod B \hookrightarrow\mathcal{D}^{b}(\mod A)$ induced by the t-structure does not extend to a derived equivalence. This phenomenon was further studied by Yang in \cite{Yang20}, who provided concrete instances of such behavior.

The equivalence of the realization functor has been widely studied. Examples include: for the module category of finite dimensional modules over a finite-dimensional hereditary algebra, Stanley and van Roosmalen \cite{StanleyRoosmalen16} proved that the realization functor is an equivalence if and only if the t-structure is bounded and the aisle of the t-structure is closed under the Serre functor; Psaroudakis and Vit\'{o}ria \cite{PsaroudakisVitoria18} developed a non-compact tilting theory, in which non-compact objects have endomorphism rings that are not derived equivalent to the original ring. Thus, they consider the hearts of the naturally associated t-structures instead of endomorphism rings, in which case the corresponding realization functors yield derived equivalences; Moreover, Chen, Han and Zhou \cite{ChenHanZhou19} proved that the realization functor with respect to the HRS tilt is an equivalence if and only if the corresponding class in the third Yoneda extension group vanishes.

Next, we investigate this phenomenon for Dynkin quivers of type $\mathbb{D}_{n}~(n\geq 5)$ and develop a method for constructing examples exhibiting this behavior.

Let $Q$ be a finite quiver with vertex set $Q_{0}$ and arrow set $Q_{1}$. Following \cite[Section 1.1]{BautistaLiu17}, $Q$ is said to be \emph{gradable} if every closed walk has virtual degree $0$. If $Q$ is gradable, then $Q$ has no oriented cycles. The path algebra $kQ$ of $Q$ is the $k$-algebra with a basis consisting of all paths in $Q$ (including trivial paths), where multiplication is defined by path concatenation. If $Q$ is graded, then $kQ$ is naturally a graded $k$-algebra.

Let $Q$ be a gradable finite quiver and  $A = kQ/I$ with $I$ consisting of relations of length at least 3. For $i \in Q_{0}$, put $P_{i} = e_{i}A$, where $e_{i}$ is the trivial path at $i$. Let

$\cd^{\leq 0}=$ the smallest full subcategory of $\mathcal{D}^{b}(A)$ containing $P_{i}[t_{i}]$ and  closed under extensions, shift $[1]$ and direct summands;

$\cd^{\geq 0}=$ the smallest full subcategory of $\mathcal{D}^{b}(A)$ containing $P_{i}[t_{i}]$ and  closed under extensions, negative shift $[-1]$ and direct summands;

\begin{theorem}\rm \cite[Theorem 4.4]{Yang20}.
\label{thm:realization-functor-is-not-an-equivalence}
The pair $(\cd^{\leq 0}, \cd^{\geq 0})$ is a bounded t-structure on $\mathcal{D}^{b}(A)$, whose heart $\mathcal{B}$ is derived equivalent to $kQ$. Moreover, the embedding $\mathcal{B}\hookrightarrow \mathcal{D}^{b}(A)$
extends to a derived equivalence $\mathcal{D}^{b}(\mathcal{B})\rightarrow\mathcal{D}^{b}(A)$ unless $I = 0$.
\end{theorem}

In fact, Theorem \ref{thm:realization-functor-is-not-an-equivalence} implies that this phenomenon arises for all Dynkin types $\mathbb{D}_{n}~(n\geq 5)$ and $\mathbb{E}_{6,7,8}$, as shown by Yang. 
Motivated by this result, we focus on classifying silted algebras of hereditary algebras of type $\mathbb{D}_{n}~(n\geq 5)$ and provide additional examples of this phenomenon.

\begin{theorem} \label{thm:more-exmaples}
Let $A = kQ$ be the path algebra of the following quiver of type $\mathbb{D}_{n}$ (with  $n\geq 5$)
$$
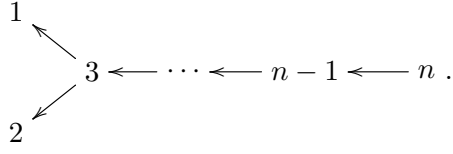
\begin{xy}
(-10,8)*+{1}="0",
(-10,-8)*+{2}="1",
(0,0)*+{3}="2",
(12,0)*+{\cdots}="3",
(28,0)*+{n-1}="4",
(44,0)*+{n}="5",
\ar"2";"0",\ar"2";"1", \ar"3";"2", \ar"4";"3", \ar"5";"4",
\end{xy}.$$
Let $P=\tau^{-1}P(1)\oplus\sum\limits_{i=5}^{n}\tau^{-1}P(i)\oplus I(2)\oplus I(n-1)\oplus I(n)[1]$, then $P$ is a 2-term silting complex of $K^{b}(\proj A)$, and its endomorphism algebra is derived equivalent to $A$. The heart $\mathcal{B}$ of the corresponding t-structure is derived equivalent to $\mod A$, but the embedding $\mathcal{B} \hookrightarrow \mathcal{D}^{b}(\mod A)$  does not extend to a derived equivalence.
\end{theorem}

\begin{proof} By the proof of \cite[Proposition 2.1]{HappelRingel81}, tilting modules induced by the idempotents $e_{1}$, $e_{5},\ldots,e_{n}$ over $A$ have the form $P(1)\oplus\sum\limits_{i=5}^{n}P(i)\oplus\tau^{-1}T$, where $T$ is a tilting module over $A'=kQ'$ with $Q'$ being the subquiver: $\begin{xy}
(-10,0)*+{2}="0",
(0,0)*+{3}="1",
(10,0)*+{4}="2",
\ar"1";"0",\ar"2";"1",
\end{xy}$
In particular, take the tilting module $e_{2}A^{\prime}\oplus e_{3}A^{\prime}\oplus \tau^{-1}e_{4}A^{\prime}$ over $A^{\prime}$. Using the AR-quiver of $\mod \Lambda_{n}$, this tilting module corresponds to $\tau^{2} I(2)\oplus \tau^{2} I(n-1)\oplus \tau I(n)$ in $\Lambda_{n}$. Thus, the corresponding tilting module over $A$ is $P(1)\oplus\sum\limits_{i=5}^{n}P(i)\oplus \tau I(2)\oplus \tau I(n-1)\oplus I(n)$. It follows that $P=\tau^{-1}P(1)\oplus\sum\limits_{i=5}^{n}\tau^{-1}P(i)\oplus I(2)\oplus I(n-1)\oplus I(n)[1]$ is a 2-term silting complex in $K^{b}(\proj A)$, with endomorphism algebra:
$$
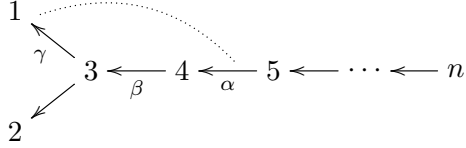
\begin{xy}
(-10,8)*+{1}="0",
(-10,-8)*+{2}="1",
(0,0)*+{3}="2",
(12,0)*+{4}="3",
(24,0)*+{5}="4",
(36,0)*+{\cdots}="5",
(48,0)*+{n}="6",
(-8,7)*+{ }="7",
(20,0)*+{ }="8",
 \ar@/_1.2pc/@{.}"8";"7",
\ar^{\gamma}"2";"0",\ar"2";"1", \ar^{\beta}"3";"2", \ar^{\alpha}"4";"3", \ar"5";"4",\ar"6";"5",
\end{xy}$$ satisfying $\alpha\beta\gamma=0$. Then by Theorem \ref{thm:realization-functor-is-not-an-equivalence}, the endomorphism algebra is derived equivalent to $A$. The heart $\mathcal{B}$ of the corresponding t-structure is derived equivalent to $\mod A$ but the embedding $\mathcal{B} \hookrightarrow \mathcal{D}^{b}(\mod A)$ does not extend to a derived equivalence.
\end{proof}

\begin{theorem} \label{thm:more-examples-of-quiver-mutation}

Let $A = kQ$ be the path algebra of the following quiver of type $\mathbb{D}_{n}$ (with  $n\geq 5$)
$$\begin{xy}
(-10,8)*+{1}="0",
(-10,-8)*+{2}="1",
(0,0)*+{3}="2",
(12,0)*+{\cdots}="3",
(28,0)*+{n-1}="4",
(44,0)*+{n}="5",
\ar"2";"0",\ar"2";"1", \ar"3";"2", \ar"4";"3", \ar"4";"5",
\end{xy}.$$
Let $P=\tau^{-2}P(1)\oplus\sum\limits_{i=5}^{n}\tau^{-2}P(i)\oplus I(2)[1]\oplus I(n-1)[1]\oplus I(n)[1]$, then $P$ is a 2-term silting complex of $K^{b}(\proj A)$, and its endomorphism algebra is derived equivalent to $A$. The heart $\mathcal{B}$ of the corresponding t-structure is derived equivalent to $\mod A$, but the embedding $\mathcal{B} \hookrightarrow \mathcal{D}^{b}(\mod A)$  does not extend to a derived equivalence.
\end{theorem}

\begin{proof}
The proof is similar to that of Theorem \ref{thm:more-exmaples}.
\end{proof}

\begin{remark}
By Remark \ref{rem:AR-quiver-parts}, tilting modules induced by the idempotents $e_{1},e_{5},\ldots,e_{n}$
and those induced by $e_{2},e_{5},\ldots,e_{n}$ have isomorphic endomorphism algebras. Consequently, 
$$P=\tau^{-1}P(2)\oplus \sum\limits_{i=5}^{n}\tau^{-1}P(i)\oplus I(1)\oplus I(n-1)\oplus I(n)[1]$$
 is also a 2-term silting complex exhibiting the phenomenon above. Similarly, 
 $$P=\tau^{-2}P(2)\oplus \sum\limits_{i=5}^{n}\tau^{-2}P(i)\oplus I(1)[1]\oplus I(n-1)[1]\oplus I(n)[1]$$ is also a 2-term silting complex with the same property.
\end{remark}

\begin{remark}
In \cite[Example 5.7]{Yang20}, Yang studied this phenomenon for the quiver algebra $A=kQ$, where $Q$ is the quiver of type $\mathbb{D}_{n}~(n\geq 5)$ and proved that there exists a silting object in $K^{b}(\proj A)$ whose endomorphism algebra is derived equivalent to $A$. Moreover, the heart $\mathcal{B}$ of the corresponding t-structure is derived equivalent to $\mod A$ but the embedding $\mathcal{B} \hookrightarrow \mathcal{D}^{b}(\mod A)$  does not extend to a derived equivalence. Theorem \ref{thm:more-exmaples} explicitly describes the form of a class of such silting objects.
\end{remark}

\def\cprime{$'$}
\providecommand{\bysame}{\leavevmode\hbox to3em{\hrulefill}\thinspace}
\providecommand{\MR}{\relax\ifhmode\unskip\space\fi MR }
\providecommand{\MRhref}[2]{%
  \href{http://www.ams.org/mathscinet-getitem?mr=#1}{#2}
}
\providecommand{\href}[2]{#2}


\begin{thebibliography}{10}

\bibitem{AdachiIyamaReiten14}
Takahide Adachi, Osamu Iyama, and Idun Reiten, \emph{{$\tau$}-tilting theory},
  Compos. Math. \textbf{150} (2014), no.~3, 415--452.

\bibitem{AiharaIyama12}
Takuma Aihara and Osamu Iyama, \emph{Silting mutation in triangulated
  categories}, J. Lond. Math. Soc. (2) \textbf{85} (2012), no.~3, 633--668.

\bibitem{ArnesenLakingPauksztello16}
Kristin~Krogh Arnesen, Rosanna Laking, and David Pauksztello, \emph{Morphisms
  between indecomposable complexes in the bounded derived category of a gentle
  algebra}, J. Algebra \textbf{467} (2016), 1--46.

\bibitem{AssemBrustleCharbonneauPlamondon10}
Ibrahim Assem, Thomas Br\"{u}stle, Gabrielle Charbonneau-Jodoin, and Pierre-Guy
  Plamondon, \emph{Gentle algebras arising from surface triangulations},
  Algebra Number Theory \textbf{4} (2010), no.~2, 201--229.

\bibitem{AssemSimsonSkowronski06}
Ibrahim Assem, Daniel Simson, and Andrzej Skowro{\'n}ski, \emph{Elements of the
  representation theory of associative algebras. {V}ol. 1}, London Mathematical
  Society Student Texts, vol.~65, Cambridge University Press, Cambridge, 2006,
  Techniques of representation theory.

\bibitem{BaurCoelho21}
Karin Baur and Raquel Coelho Sim\~{o}es, \emph{A geometric model for the module
  category of a gentle algebra}, Int. Math. Res. Not. IMRN (2021), no.~15,
  11357--11392.

\bibitem{BautistaLiu17}
Raymundo Bautista and Shiping Liu, \emph{The bounded derived category of an
  algebra with radical squared zero}, J. Algebra \textbf{482} (2017), 303--345.

\bibitem{BeilinsonBernsteinDeligne82}
Alexander~A. Beilinson, Joseph Bernstein, and Pierre Deligne, \emph{Faisceaux
  pervers}, Ast{\'e}risque, vol. 100, Soc. Math. France, 1982 (French).

\bibitem{BekkertMerklen03}
Viktor Bekkert and H{\'e}ctor~A. Merklen, \emph{Indecomposables in derived
  categories of gentle algebras}, Algebr. Represent. Theory \textbf{6} (2003),
  no.~3, 285--302.

\bibitem{BernsteinGelfandPonomarev73}
I.~N. Bern\v{s}te\u{\i}n, I.~M. Gel{\cprime}fand, and V.~A. Ponomarev,
  \emph{Coxeter functors, and {G}abriel's theorem}, Uspehi Mat. Nauk
  \textbf{28} (1973), no.~2(170), 19--33.

\bibitem{BrustleYang13}
Thomas Br\"{u}stle and Dong Yang, \emph{Ordered exchange graphs}, Advances in
  representation theory of algebras, EMS Ser. Congr. Rep., pp.~135--193.

\bibitem{BuanZhou16b}
Aslak~Bakke Buan and Yu~Zhou, \emph{Silted algebras}, Adv. Math. \textbf{303}
  (2016), 859--887.

\bibitem{BuanZhou16}
\bysame, \emph{A silting theorem}, J. Pure Appl. Algebra \textbf{220} (2016),
  no.~7, 2748--2770.

\bibitem{BuanZhou18}
\bysame, \emph{Endomorphism algebras of 2-term silting complexes}, Algebr.
  Represent. Theory \textbf{21} (2018), no.~1, 181--194.

\bibitem{ButlerRingel87}
Michael. C.~R. Butler and Claus~Michael Ringel, \emph{Auslander-{R}eiten
  sequences with few middle terms and applications to string algebras}, Comm.
  Algebra \textbf{15} (1987), no.~1-2, 145--179.

\bibitem{ChenHanZhou19}
Xiao-Wu Chen, Zhe Han, and Yu~Zhou, \emph{Derived equivalences via
  {HRS}-tilting}, Adv. Math. \textbf{354} (2019), 106749. 26.

\bibitem{CoelhoLanzilotta99}
Fl\'{a}vio~Ulhoa Coelho and Marcelo~Am\'{e}rico Lanzilotta, \emph{Algebras with
  small homological dimensions}, Manuscripta Math. \textbf{100} (1999), no.~1,
  1--11.

\bibitem{Happel88}
Dieter Happel, \emph{Triangulated categories in the representation theory of
  finite-dimensional algebras}, London Mathematical Society Lecture Note
  Series, vol. 119, Cambridge University Press, Cambridge, 1988.

\bibitem{HappelRingel81}
Dieter Happel and Claus~Michael Ringel, \emph{Construction of tilted algebras},
  Representations of algebras ({P}uebla, 1980), Lecture Notes in Math., vol.
  903, Springer, Berlin, 1981, pp.~125--144.

\bibitem{HappelRingel82}
\bysame, \emph{Tilted algebras}, Trans. Amer. Math. Soc. \textbf{274} (1982),
  no.~2, 399--443.

\bibitem{KellerVossieck88}
Bernhard Keller and Dieter Vossieck, \emph{Aisles in derived categories}, Bull.
  Soc. Math. Belg. S{\'e}r. A \textbf{40} (1988), no.~2, 239--253.

\bibitem{ObaidNaumanFakiehRingel15}
Mustafa A.~A. Obaid, S.~Khalid Nauman, Wafaa~M. Fakieh, and Claus~Michael
  Ringel, \emph{The number of support-tilting modules for a {D}ynkin algebra},
  J. Integer Seq. \textbf{18} (2015), no.~10, Article 15.10.6, 24.

\bibitem{Poettering10}
Nicolas Poettering, \emph{Quivers without loops admit global dimension 2},
  arXiv:1010.3871v1.

\bibitem{PsaroudakisVitoria18}
Chrysostomos Psaroudakis and Jorge Vit\'{o}ria, \emph{Realisation functors in
  tilting theory}, Math. Z. \textbf{288} (2018), no.~3-4, 965--1028.

\bibitem{Ringel84}
C.M. Ringel, \emph{Tame algebras and integral quadratic forms}, Lecture Notes
  in Mathematics, vol. 1099, Springer Verlag, 1984.

\bibitem{StanleyRoosmalen16}
Donald Stanley and Adam-Christiaan van Roosmalen, \emph{Derived equivalences
  for hereditary {A}rtin algebras}, Adv. Math. \textbf{303} (2016), 415--463.

\bibitem{XieYangZhang25}
Zongzhen Xie, Dong Yang, and Houjun Zhang, \emph{Classification of silted
  algebras for two quivers of Dynkin type $\mathbb{A}_{n}$}, arXiv:2504.14789v1.

\bibitem{XieYangZhang23}
Zongzhen Xie, Dong Yang, and Houjun Zhang, \emph{A differential graded approach
  to the silting theorem}, J. Pure Appl. Algebra \textbf{227} (2023), no.~2,
  Paper No. 107180, 23.

\bibitem{Xing21}
Ruoyun Xing, \emph{Some examples of silted algebras of Dynkin type},
  arXiv:2106.00970v1.

\bibitem{Yang20}
Dong Yang, \emph{Some examples of {$t$}-structures for finite-dimensional
  algebras}, J. Algebra \textbf{560} (2020), 17--47.

\bibitem{YangZhang19}
Jiajia Yang and Chao Zhang, \emph{The global dimensions of two classes of
  algebras}(in Chinese), Journal of Guizhou University (Natural Sciences) \textbf{36}
  (2019), no.~5, 18--20.

\bibitem{ZhangLiuLiu25}
Houjun Zhang, Dajun Liu, and Yu-Zhe Liu, \emph{On the Gorensteiness of string
  algebras}, arXiv:2504.00089v1.

\bibitem{ZhangLiu25}
Houjun Zhang and Yu-Zhe Liu, \emph{There are no strictly shod algebras in
  hereditary gentle algebras}, Comm. Algebra \textbf{53} (2025), no.~3,
  1062--1075.

\end{thebibliography}

\end{document}